\documentclass[reqno, 12pt]{amsart}
\usepackage{amssymb}
\usepackage{amsmath}
\usepackage{amsfonts}

\newtheorem{theorem}{Theorem}
\theoremstyle{plain}
\newtheorem{acknowledgement}{Acknowledgement}

\newtheorem{corollary}{Corollary}

\newtheorem{definition}{Definition}

\newtheorem{lemma}{Lemma}

\newtheorem{proposition}{Proposition}
\newtheorem{remark}{Remark}

\numberwithin{equation}{section}
 \numberwithin{theorem}{section}
 \numberwithin{proposition}{section}
 \numberwithin{remark}{section}
 \numberwithin{definition}{section}
 \numberwithin{lemma}{section}
 \numberwithin{corollary}{section}
 \numberwithin{example}{section}
 \numberwithin{claim}{section}
\input{tcilatex}

\begin{document}
\title[Homogenization in thin heterogeneous domains]{Sigma-convergence for
thin heterogeneous domains and application to the upscaling of
Darcy-Lapwood-Brinkmann flow}
\author{Willi J\"{a}ger}
\address{W. J\"{a}ger, Interdisciplinary Center for Scientific Computing
(IWR), University of Heidelberg, Im Neuenheimer Feld 205, 69120 Heidelberg,
Germany}
\email{wjaeger@iwr.uni-heidelberg.de}
\author{Jean Louis Woukeng}
\address{J.L. Woukeng, Department of Mathematics and Computer Science,
University of Dschang, P.O. Box 67, Dschang, Cameroon}
\curraddr{J.L. Woukeng, Institute of Advanced Scientific Studies (IHES), 35
Rte de Chartres, 91440 Bures-sur-Yvette, France}
\email{jeanlouis.woukeng@univ-dschang.org, woukeng@ihes.fr}
\date{May 2025}
\subjclass[2000]{35B40, 46J10, 35Q35}
\keywords{Thin heterogeneous domains with oscillating boundaries,
Darcy-Lapwood-Brinkmann equation, homogenization, algebras with mean value,
sigma-convergence}

\begin{abstract}
The sigma-convergence concept has been up to now used to derive macroscopic
models in full space dimensions. In this work, we generalize it to thin
heterogeneous domains given rise to phenomena in lower space dimensions.
More precisely, we provide a new approach of the sigma-convergence method
that is suitable for the study of phenomena occurring in thin heterogeneous
media. This is made through a systematic study of the sigma-convergence
method for thin heterogeneous domains. Assuming that the thin heterogeneous
layer is made of microstructures that are distributed inside in a
deterministic way including as special cases the periodic and the almost
periodic distributions, we make use of the concept of algebras with mean
value to state and prove the main compactness results. As an illustration,
we upscale a Darcy-Lapwood-Brinkmann micro-model for thin flow. We prove
that, according to the magnitude of the permeability of the porous domain,
we obtain as effective models, the Darcy law in lower dimensions. The
effective models are derived through the solvability of either the local
Stokes-Brinkmann problems or the local Hele-Shaw problems.
\end{abstract}

\maketitle

\section{Introduction and the main results\label{sec1}}

This work is concerned with two main challenges: 1) build a framework
enabling the study of physical/natural phenomena occurring in thin
heterogeneous media by taking into account the distribution of
heterogeneities inside the thin layer; 2) apply the results obtained in the
first part 1) to upscale a double porosity model, viz. a
Darcy-Lapwood-Brinkmann flow occurring in thin heterogeneous domains.

The first main challenge is to prove some important compactness results in
thin heterogeneous domains. To make it a little more precise, let $%
G_{i}\subset \mathbb{R}^{d_{i}}$ ($i=1,2$) be a bounded open set in $\mathbb{%
R}^{d_{i}}$ (integer $d_{i}\geq 1$) with $0\in \overline{G}_{2}$ (the
closure of $G_{2}$). We set $G_{\varepsilon }=G_{1}\times \varepsilon G_{2}$
(where $\varepsilon >0$ is a small parameter) and $G_{0}=G_{1}\times \{0\}$.
The first main result states as follows.

\begin{theorem}
\label{t1.1}Let $(u_{\varepsilon })_{\varepsilon \in E}$ be a sequence in $%
L^{p}(G_{\varepsilon })$ satisfying 
\begin{equation*}
\sup_{\varepsilon \in E}\varepsilon ^{-\frac{d_{2}}{p}}\left\Vert
u_{\varepsilon }\right\Vert _{L^{p}(G_{\varepsilon })}\leq C,\ \ \ \ \ \ \ \
\ \ \ \ \ \ \ \ \ \ \ \ \ \ \ \ \ \ \ \ \ \ \ \ \ \ \ \ \ \ \ \ \ \ 
\end{equation*}%
where $C>0$ is a constant independent of $\varepsilon $, $1<p<\infty $ and $%
E $ is an ordinary sequence of positive real numbers $\varepsilon $ tending
to zero. Then there exist a subsequence $E^{\prime }$ of $E$ and $u_{0}\in
L^{p}(G_{0};\mathcal{B}_{\mathcal{A}}^{p}(\mathbb{R}^{d_{1}};L^{p}(G_{2})))$
such that, as $E^{\prime }\ni \varepsilon \rightarrow 0$, 
\begin{equation}
\varepsilon ^{-d_{2}}\int_{G_{\varepsilon }}u_{\varepsilon }(x)f\left( 
\overline{x},\frac{x}{\varepsilon }\right) dx\rightarrow
\int_{G_{0}}\int_{G_{2}}M(u_{0}(\overline{x},\cdot ,\zeta )f(\overline{x}%
,\cdot ,\zeta ))d\zeta d\overline{x}  \label{e1.1}
\end{equation}%
for all $f\in L^{p^{\prime }}(G_{0};\mathcal{A}(\mathbb{R}%
^{d_{1}};L^{p^{\prime }}(G_{2})))$ ($1/p^{\prime }=1-1/p$).
\end{theorem}

We denote the convergence property (\ref{e1.1}) by "$u_{\varepsilon
}\rightarrow u_{0}$ in $L^{p}(G_{\varepsilon })$-weak $\Sigma _{\mathcal{A}}$%
".

In Theorem \ref{t1.1}, $\mathcal{A}$ is an algebra with mean value on $%
\mathbb{R}^{d_{1}}$ while $\mathcal{B}_{\mathcal{A}}^{p}(\mathbb{R}%
^{d_{1}};L^{p}(G_{2}))$ stands for the vector-valued generalized Besicovitch
space associated to $\mathcal{A}$; $M$ denotes the mean value on $\mathcal{B}%
_{\mathcal{A}}^{p}(\mathbb{R}^{d_{1}};L^{p}(G_{2}))$. We refer the reader to
Section \ref{sec2} for these concepts. Theorem \ref{t1.1} is at the heart of
some other important compactness results stated and proved in Section \ref%
{sec3} of this work. It generalizes its homologue stated in \cite{RJ2007},
which is concerned with the periodic version. It is worth noticing that the
generalization is not straightforward as it heavily relies on Lemma \ref%
{l2.2} stated and proved in \cite[Proposition 3.2]{CMP}. We use Lemma \ref%
{l2.2} to get rid of separability issue, as some algebras with mean value
are not separable, in contrast with the algebra of continuous periodic
functions on $\mathbb{R}^{d_{1}}$. To the best of our knowledge, there is no
result available so far in the literature dealing with such kind of
compactness results in thin domain beyond the periodic setting. So our
result is new and most likely, of great applicability.

The second main result deals with the sigma-convergence in thin
heterogeneous domains with oscillating lateral boundaries. Let $\Omega $ be
a bounded open Lipschitz domain in $\mathbb{R}^{d-1}$ (integer $d\geq 2$),
and let $h_{1}$ and $h_{2}$ be two Lipschtiz continuous functions defined on 
$\mathbb{R}^{d-1}$ and satisfying $h_{1},h_{2}\in W^{1,\infty }(\mathbb{R}%
^{d-1})\cap \mathcal{A}$ ($\mathcal{A}$ an algebra with mean value on $%
\mathbb{R}^{d-1}$). We define the thin domain here as follows 
\begin{equation}
\Omega ^{\varepsilon }=\left\{ x=(\overline{x},x_{d})\in \mathbb{R}^{d}:%
\overline{x}\in \Omega \text{ and }\varepsilon h_{1}\left( \frac{\overline{x}%
}{\varepsilon }\right) <x_{3}<\varepsilon h_{2}\left( \frac{\overline{x}}{%
\varepsilon }\right) \right\} .  \label{1.2}
\end{equation}%
Set $I=(\min_{\mathbb{R}^{d-1}}h_{1},\max_{\mathbb{R}^{d-1}}h_{2})$ and
assume that $0\in \lbrack \min_{\mathbb{R}^{d-1}}h_{1},\max_{\mathbb{R}%
^{d-1}}h_{2}]$, and define $G_{\varepsilon }=\Omega \times \varepsilon I$.
Finally define the set 
\begin{equation*}
\mathbb{J}=\left\{ y=(\overline{y},y_{d})\in \mathbb{R}^{d}:\overline{y}\in 
\mathbb{R}^{d-1}\text{ and }h_{1}(\overline{y})<y_{3}<h_{2}(\overline{y}%
)\right\} ,
\end{equation*}%
and we denote by $\chi _{\mathbb{J}}$ its characteristic function in $%
\mathbb{R}^{d}$. We are in a position to state the second main compactness
result.

\begin{theorem}
\label{t1.4}Let $(u_{\varepsilon })_{\varepsilon \in E}\subset L^{p}(\Omega
^{\varepsilon })$ ($1<p<\infty $) be a sequence satisfying 
\begin{equation*}
\sup_{\varepsilon \in E}\varepsilon ^{-\frac{1}{p}}\left\Vert u_{\varepsilon
}\right\Vert _{L^{p}(\Omega ^{\varepsilon })}\leq C,\ \ \ \ \ \ \ \ \ \ \ \
\ \ \ \ \ \ \ \ \ \ \ \ \ \ \ \ \ \ \ \ \ \ \ \ \ \ \ \ \ \ 
\end{equation*}%
where $C>0$ is independent of $\varepsilon $ and $E$ is as in Theorem \emph{%
\ref{t1.1}}. Assume that there exists a continuous extension operator $%
P_{\varepsilon }:L^{p}(\Omega ^{\varepsilon })\rightarrow
L^{p}(G_{\varepsilon })$ satisfying 
\begin{equation*}
\left\Vert P_{\varepsilon }v\right\Vert _{L^{p}(G_{\varepsilon })}\leq
C\left\Vert v\right\Vert _{L^{p}(\Omega ^{\varepsilon })}\ \ \forall v\in
L^{p}(\Omega ^{\varepsilon }),
\end{equation*}%
$C>0$ being independent of $\varepsilon $. Let $E^{\prime }$ and $u_{0}\in
L^{p}(G_{0};\mathcal{B}_{\mathcal{A}}^{p}(\mathbb{R}^{d-1};L^{p}(I)))$ be
defined by Theorem \emph{\ref{t1.1}} and such that $P_{\varepsilon
}u_{\varepsilon }\rightarrow u_{0}$ in $L^{p}(G_{\varepsilon })$-weak $%
\Sigma _{\mathcal{A}}$. Then, as $E^{\prime }\ni \varepsilon \rightarrow 0$, 
\begin{equation*}
\frac{1}{\varepsilon }\int_{\Omega ^{\varepsilon }}u_{\varepsilon
}(x)f\left( \overline{x},\frac{x}{\varepsilon }\right) dx\rightarrow
\int_{G_{0}}\int_{I}M(\chi _{\mathbb{J}}u_{0}(\overline{x},\cdot ,y_{d})f(%
\overline{x},\cdot ,y_{d}))dy_{d}d\overline{x}
\end{equation*}%
for all $f\in L^{p^{\prime }}(G_{0};\mathcal{A}(\mathbb{R}^{d-1};\mathcal{C}(%
\overline{I})))$, $1/p^{\prime }=1-1/p$.
\end{theorem}

To illustrate the previous results, we consider the upscaling of a double
porosity model in thin heterogeneous layers. The model problem is stated as
follows. In the thin heterogeneous layer $\Omega ^{\varepsilon }$, we
consider the flow of a fluid described by the Darcy-Lapwood-Brinkmann system 
\begin{equation}
\left\{ 
\begin{array}{l}
-\func{div}\left( A\left( \frac{x}{\varepsilon }\right) \nabla \boldsymbol{u}%
_{\varepsilon }\right) +\dfrac{\mu }{K_{\varepsilon }}\boldsymbol{u}%
_{\varepsilon }+\dfrac{\rho }{\phi ^{2}}(\boldsymbol{u}_{\varepsilon }\cdot
\nabla )\boldsymbol{u}_{\varepsilon }+\nabla p_{\varepsilon }=\boldsymbol{f}%
\text{ in }\Omega ^{\varepsilon }, \\ 
\func{div}\boldsymbol{u}_{\varepsilon }=0\text{ in }\Omega ^{\varepsilon },
\\ 
\boldsymbol{u}_{\varepsilon }=0\text{ on }\partial \Omega ^{\varepsilon },%
\end{array}%
\right.  \label{1.1}
\end{equation}%
where the assumptions on $A$ and $\boldsymbol{f}$ ensure the existence (for
each fixed $\varepsilon >0$) of a solution $(\boldsymbol{u}_{\varepsilon
},p_{\varepsilon })\in H_{0}^{1}(\Omega ^{\varepsilon })^{3}\times
L_{0}^{2}(\Omega ^{\varepsilon })$. We consider two different kind of thin
layers: a thin layer with flat parallel boundaries and a thin layer with
highly oscillating boundaries. In each case, we obtain the following
results, which are respectively, the third and the fourth main results of
this work.

\begin{theorem}
\label{t1.2}For each $\varepsilon >0$, let $\Omega ^{\varepsilon }$ be given
by \emph{(\ref{1.2})} with $h_{1}=-1$ and $h_{2}=1$ (i.e., $\Omega
^{\varepsilon }=\Omega \times (-\varepsilon ,\varepsilon )$, where $\Omega $
is a bounded open connected Lipschitz subset in $\mathbb{R}^{2}$), and let $(%
\boldsymbol{u}_{\varepsilon },p_{\varepsilon })$ be a solution of \emph{(\ref%
{1.1})}. Assume that $A\in (B_{\mathcal{A}}^{2}(\mathbb{R}%
^{2};L^{2}(I)))^{3\times 3}$. Then:

\begin{itemize}
\item[(i)] If $K_{\varepsilon }=O(\varepsilon ^{2})$ with $K_{\varepsilon
}/\varepsilon ^{2}\rightarrow K$ when $\varepsilon \rightarrow 0$, $%
0<K<\infty $, then $(\boldsymbol{u}_{\varepsilon }/\varepsilon
^{2},p_{\varepsilon })$ weakly $\Sigma _{\mathcal{A}}$-converges (as $%
\varepsilon \rightarrow 0$) in $L^{2}(\Omega ^{\varepsilon })^{3}\times
L^{2}(\Omega ^{\varepsilon })$ to $(\boldsymbol{u}_{0},p_{0})$ belonging to $%
[L^{2}(\Omega ;\mathcal{B}_{\mathcal{A}}^{1,2}(\mathbb{R}%
^{2};H_{0}^{1}(I)))]^{3}\times L^{2}(\Omega )$. Moreover $p_{0}\in
H^{1}(\Omega )$ and, defining $\boldsymbol{u}(\overline{x})=\int_{-1}^{1}M(%
\boldsymbol{u}_{0}(\overline{x},\cdot ,\zeta )d\zeta \equiv (\boldsymbol{u}%
^{\prime },u_{3})$, one has $u_{3}=0$ and $(\boldsymbol{u}^{\prime },p_{0})$
is the unique solution to the effective problem 
\begin{equation*}
\left\{ 
\begin{array}{l}
\boldsymbol{u}^{\prime }=\widehat{A}(\boldsymbol{f}_{1}-\nabla _{\overline{x}%
}p_{0})\text{ in }\Omega \\ 
\func{div}_{\overline{x}}\boldsymbol{u}^{\prime }=0\text{ in }\Omega \text{
and }\boldsymbol{u}^{\prime }\cdot \nu =0\text{ on }\partial \Omega%
\end{array}%
\right.
\end{equation*}%
where $\widehat{A}=(\widehat{a}_{ij})_{1\leq i,j\leq 2}$ is a symmetric,
positive definite $2\times 2$ matrix defined by its entries 
\begin{equation*}
\widehat{a}_{ij}=\int_{-1}^{1}M(A\overline{\nabla }_{y}\boldsymbol{w}_{i}:%
\overline{\nabla }_{y}\boldsymbol{w}_{j})d\zeta +\frac{\mu }{K}%
\int_{-1}^{1}M(\boldsymbol{w}_{i}\boldsymbol{w}_{j})d\zeta .
\end{equation*}%
Here $\boldsymbol{w}_{i}$ ($1\leq i\leq 2$) is the unique solution in $[%
\mathcal{B}_{\mathcal{A}}^{1,2}(\mathbb{R}^{2};H_{0}^{1}(I))]^{3}$ of the
Stokes-Brinkmann system 
\begin{equation*}
\left\{ 
\begin{array}{l}
-\overline{\func{div}}_{y}(A(y)\overline{\nabla }_{y}\boldsymbol{w}_{i})+%
\frac{\mu }{K}\boldsymbol{w}_{i}+\overline{\nabla }_{y}\pi _{i}=e_{i}\text{
in }\mathbb{R}^{2}\times I, \\ 
\overline{\func{div}}_{y}\boldsymbol{w}_{i}=0\text{ in }\mathbb{R}^{2}\times
I,%
\end{array}%
\right.
\end{equation*}%
$e_{i}$ being the $i$\emph{th} vector of the canonical basis in $\mathbb{R}%
^{3}$.

\item[(ii)] If $K_{\varepsilon }\ll \varepsilon ^{2}$, then $(\boldsymbol{u}%
_{\varepsilon }/(\varepsilon K_{\varepsilon }^{1/2}),(K_{\varepsilon
}^{1/2}/\varepsilon )p_{\varepsilon })_{\varepsilon >0}$ weakly $\Sigma _{%
\mathcal{A}}$-converges (as $\varepsilon \rightarrow 0$) in $L^{2}(\Omega
^{\varepsilon })^{3}\times L^{2}(\Omega ^{\varepsilon })$ toward $(%
\boldsymbol{u}_{0},p_{0})\in \lbrack L^{2}(\Omega ;\mathcal{B}_{\mathcal{A}%
}^{1,2}(\mathbb{R}^{2};H_{0}^{1}(I)))]^{3}\times L^{2}(\Omega )$.
Furthermore, defining $\boldsymbol{u}$ as in \emph{(i)} above, we have $%
u_{3}=0$ and $(\boldsymbol{u}^{\prime },p_{0})$ is a solution of 
\begin{equation*}
\boldsymbol{u}^{\prime }=-\widehat{A}\nabla _{\overline{x}}p_{0}\text{ in }%
\Omega \text{, }\func{div}_{\overline{x}}\boldsymbol{u}^{\prime }=0\text{ in 
}\Omega \text{ and }\boldsymbol{u}^{\prime }\cdot \nu =0\text{ on }\partial
\Omega ,
\end{equation*}%
where $\widehat{A}$ is a symmetric matrix defined by 
\begin{equation*}
\widehat{A}=(\widehat{a}_{ij})_{1\leq i,j\leq 2}\text{ with }\widehat{a}%
_{ij}=\mu \int_{-1}^{1}M(\boldsymbol{w}_{i}\boldsymbol{w}_{j})d\zeta ,
\end{equation*}%
$\boldsymbol{w}_{i}$ $(1\leq i\leq 2)$ being the unique solution in $%
\mathcal{B}_{\mathcal{A}}^{2}(\mathbb{R}^{2};L^{2}(I))^{3}$ of 
\begin{equation*}
\mu \boldsymbol{w}_{i}+\overline{\nabla }_{y}\pi _{i}=e_{i}\text{ in }%
\mathbb{R}^{2}\times I\text{ and }\overline{\func{div}}_{y}\boldsymbol{w}%
_{i}=0\text{ in }\mathbb{R}^{2}\times I.
\end{equation*}

\item[(iii)] If $K_{\varepsilon }\gg \varepsilon ^{2}$, then $(\boldsymbol{u}%
_{\varepsilon }/\varepsilon ^{2},p_{\varepsilon })_{\varepsilon >0}$ weakly $%
\Sigma _{\mathcal{A}}$-converges (as $\varepsilon \rightarrow 0$) in $%
L^{2}(\Omega ^{\varepsilon })^{3\times 3}\times L^{2}(\Omega ^{\varepsilon
}) $ to $(\boldsymbol{u}_{0},p_{0})$ where $\boldsymbol{u}_{0}\in \lbrack
L^{2}(\Omega ;B_{\mathcal{A}}^{1,2}(\mathbb{R}^{2};H_{0}^{1}(I)))]^{3}$ and $%
p_{0}\in L^{2}(\Omega )$. Furthermore $p_{0}\in H^{1}(\Omega )$ and,
defining $\boldsymbol{u}=(\boldsymbol{u}^{\prime },u_{3})$ as in \emph{(i)}
above, one has $u_{3}=0$ and $(\boldsymbol{u}^{\prime },p_{0})$ is the
unique solution of the effective problem 
\begin{equation*}
\left\{ 
\begin{array}{l}
\boldsymbol{u}^{\prime }=\widehat{A}(\boldsymbol{f}_{1}-\nabla _{\overline{x}%
}p_{0})\text{ in }\Omega \\ 
\func{div}_{\overline{x}}\boldsymbol{u}^{\prime }=0\text{ in }\Omega \text{
and }\boldsymbol{u}^{\prime }\cdot \nu =0\text{ on }\partial \Omega ,%
\end{array}%
\right.
\end{equation*}%
where $\widehat{A}=(\widehat{a}_{ij})_{1\leq i,j\leq 2}$ is given by 
\begin{equation*}
\widehat{a}_{ij}=\int_{-1}^{1}M(A\overline{\nabla }_{y}\boldsymbol{w}_{i}:%
\overline{\nabla }_{y}\boldsymbol{w}_{j})dy_{3},\ 1\leq i,j\leq 2,
\end{equation*}%
where here, $\boldsymbol{w}_{i}$ $(1\leq i\leq 2)$ is the unique solution in 
$(\mathcal{B}_{\mathcal{A}}^{1,2}(\mathbb{R}^{2};H_{0}^{1}(I)))^{3}$ of the
Stokes system 
\begin{equation*}
\left\{ 
\begin{array}{l}
-\overline{\func{div}}_{y}\left( A(y)\overline{\nabla }_{y}\boldsymbol{w}%
_{i}\right) +\overline{\nabla }_{y}\pi _{i}=e_{i}\text{ in }\mathbb{R}%
^{2}\times I \\ 
\overline{\func{div}}_{y}\boldsymbol{w}_{i}=0\text{ in }\mathbb{R}^{2}\times
I.%
\end{array}%
\right.
\end{equation*}
\end{itemize}
\end{theorem}

\begin{remark}
\label{r1.1}\emph{Let us recall the following concept that has been used in
the statement of Theorem \ref{t1.2} above. Let }$U$\emph{\ and }$V$\emph{\
be two positive functions of a small positive variable }$\varepsilon $\emph{%
, such that }$U,V\rightarrow 0$\emph{\ when }$\varepsilon \rightarrow 0$%
\emph{. We say that }%
\begin{eqnarray*}
U &=&O(V)\text{\emph{\ provided that }}\frac{U}{V}\rightarrow K\text{\emph{\
as }}\varepsilon \rightarrow 0\text{\emph{, where }}0<K<\infty ; \\
U &\ll &V\text{\emph{\ iff }}\frac{U}{V}\rightarrow 0\text{\emph{\ as }}%
\varepsilon \rightarrow 0\text{\emph{, and }}U\gg V\text{\emph{\ iff }}V\ll U%
\text{.}
\end{eqnarray*}
\end{remark}

In the next result, we assume that $\Omega ^{\varepsilon }$ is given by (\ref%
{1.2}) where $h_{1}$ and $h_{2}$ satisfy $\max_{\mathbb{R}^{2}}h_{1}<\min_{%
\mathbb{R}^{2}}h_{2}$, and $h_{1}$, $h_{2}\in \mathcal{A}$ with $%
M(h_{2}-h_{1})\neq 0$, $\mathcal{A}$ being an ergodic algebra with mean
value on $\mathbb{R}^{2}$. The result reads as follows.

\begin{theorem}
\label{t1.3}Assume that $\Omega ^{\varepsilon }$ is given by \emph{(\ref{1.2}%
)}. Let $(\boldsymbol{u}_{\varepsilon },p_{\varepsilon }=p_{\varepsilon
}^{0}+\varepsilon p_{\varepsilon }^{1})\in H_{0}^{1}(\Omega ^{\varepsilon
})^{3}\times L_{0}^{2}(\Omega ^{\varepsilon })$ be a solution of \emph{(\ref%
{1.1})}. Assume that $A\in (B_{\mathcal{A}}^{2}(\mathbb{R}^{2};L^{\infty
}(I)))^{3\times 3}$. Then:

\begin{itemize}
\item[(i)] If $K_{\varepsilon }=O(\varepsilon ^{2})$ with $K_{\varepsilon
}/\varepsilon ^{2}\rightarrow K$ when $\varepsilon \rightarrow 0$, $%
0<K<\infty $, then, still denoting by $\boldsymbol{u}_{\varepsilon }$ the
extension by zero of $\boldsymbol{u}_{\varepsilon }$ on $G_{\varepsilon
}=\Omega \times (\varepsilon h_{1}^{-},\varepsilon h_{2}^{+})$, one has 
\begin{equation*}
\frac{\boldsymbol{u}_{\varepsilon }}{\varepsilon ^{2}}\rightarrow 
\boldsymbol{u}_{0}\text{ in }L^{2}(G_{\varepsilon })^{3}\text{-weak }\Sigma
_{\mathcal{A}},\ \ \ \ \ \ \ \ \ \ \ \ \ \ \ 
\end{equation*}%
and 
\begin{equation*}
p_{\varepsilon }^{0}\rightarrow p_{0}\text{ in }H^{1}(\Omega )\text{-weak
and in }L^{2}(\Omega )\text{-strong.}
\end{equation*}%
Defining $\boldsymbol{u}=(\boldsymbol{u}^{\prime },u_{3})$ by $\boldsymbol{u}%
(\overline{x})=M\left( \int_{h_{1}}^{h_{2}}\boldsymbol{u}_{0}(\overline{x}%
,\cdot ,y_{3})dy_{3}\right) $, we have $u_{3}=0$ and $(\boldsymbol{u}%
^{\prime },p_{0})$ is the unique solution of the homogenized problem 
\begin{equation*}
\left\{ 
\begin{array}{l}
\boldsymbol{u}^{\prime }=\widehat{A}(\boldsymbol{f}_{1}-\nabla _{\overline{x}%
}p_{0})\text{ in }\Omega , \\ 
\func{div}_{\overline{x}}\boldsymbol{u}^{\prime }=0\text{ in }\Omega \text{,
and }\boldsymbol{u}^{\prime }\cdot \nu =0\text{ on }\partial \Omega ,%
\end{array}%
\right.
\end{equation*}%
where $\widehat{A}=(\widehat{a}_{ij})_{1\leq i,j\leq 2}$ is a symmetric,
positive definite $2\times 2$ matrix defined by its entries 
\begin{equation*}
\widehat{a}_{ij}=M\left( \int_{h_{1}}^{h_{2}}A\overline{\nabla }_{y}%
\boldsymbol{w}_{i}\cdot \overline{\nabla }_{y}\boldsymbol{w}%
_{j}dy_{3}\right) +\frac{\mu }{K}M\left( \int_{h_{1}}^{h_{2}}\boldsymbol{w}%
_{i}\cdot \boldsymbol{w}_{j}dy_{3}\right) .
\end{equation*}%
Here $\boldsymbol{w}_{i}$ ($1\leq i\leq 2$) is the unique solution in $B_{\#,%
\func{div}}^{1,2}(\mathbb{J})$ of the Stokes-Brinkmann system 
\begin{equation*}
\left\{ 
\begin{array}{l}
-\overline{\func{div}}_{y}\left( A(y)\overline{\nabla }_{y}\boldsymbol{w}%
_{i}\right) +\frac{\mu }{K}\boldsymbol{w}_{i}+\overline{\nabla }_{y}\pi
_{i}=e_{i}\text{ in }\mathbb{J} \\ 
\overline{\func{div}}_{y}\boldsymbol{w}_{i}=0\text{ in }\mathbb{J}, \\ 
\boldsymbol{w}_{i}=0\text{ on }\partial \mathbb{J}.%
\end{array}%
\right.
\end{equation*}%
$e_{i}$ being the $i$\emph{th} vector of the canonical basis in $\mathbb{R}%
^{3}$ and 
\begin{equation*}
\mathbb{J}=\left\{ y=(\overline{y},y_{3})\in \mathbb{R}^{3}:\overline{y}\in 
\mathbb{R}^{2}\text{ and }h_{1}(\overline{y})<y_{3}<h_{2}(\overline{y}%
)\right\} .
\end{equation*}

\item[(ii)] If $K_{\varepsilon }\ll \varepsilon ^{2}$, then, up to a
subsequence, one has 
\begin{equation*}
\frac{\boldsymbol{u}_{\varepsilon }}{\varepsilon K_{\varepsilon }^{\frac{1}{2%
}}}\rightarrow \boldsymbol{u}_{0}\text{ in }L^{2}(G_{\varepsilon })^{3}\text{%
-weak }\Sigma _{\mathcal{A}},\ \ \ \ \ \ 
\end{equation*}%
\begin{equation*}
\frac{K_{\varepsilon }^{\frac{1}{2}}}{\varepsilon }p_{\varepsilon
}^{0}\rightarrow p_{0}\text{ in }H^{1}(\Omega )\text{-weak and in }%
L^{2}(\Omega )\text{-strong.}\ \ \ \ \ \ \ 
\end{equation*}%
Furthermore, defining $\boldsymbol{u}$ as in \emph{(i)} above, we have $%
u_{3}=0$ and $(\boldsymbol{u}^{\prime },p_{0})$ is a solution of 
\begin{equation*}
\boldsymbol{u}^{\prime }=-\widehat{A}\nabla _{\overline{x}}p_{0}\text{ in }%
\Omega \text{, }\func{div}_{\overline{x}}\boldsymbol{u}^{\prime }=0\text{ in 
}\Omega \text{ and }\boldsymbol{u}^{\prime }\cdot \nu =0\text{ on }\partial
\Omega ,
\end{equation*}%
where $\widehat{A}$ is a symmetric matrix defined by 
\begin{equation*}
\widehat{A}=(\widehat{a}_{ij})_{1\leq i,j\leq 2}\text{ with }\widehat{a}%
_{ij}=\mu M\left( \int_{h_{1}}^{h_{2}}\boldsymbol{w}_{i}\boldsymbol{w}%
_{j}dy_{3}\right) ,
\end{equation*}%
$\boldsymbol{w}_{i}$ $(1\leq i\leq 2)$ being the unique solution in $%
\mathcal{B}_{\mathcal{A}}^{2}(\mathbb{J})^{3}$ of 
\begin{equation*}
\mu \boldsymbol{w}_{i}+\overline{\nabla }_{y}\pi _{i}=e_{i}\text{ in }%
\mathbb{J}\text{ and }\overline{\func{div}}_{y}\boldsymbol{w}_{i}=0\text{ in 
}\mathbb{J}.
\end{equation*}

\item[(iii)] If $K_{\varepsilon }\gg \varepsilon ^{2}$, then, still denoting
by $\boldsymbol{u}_{\varepsilon }$ and $p_{\varepsilon }^{1}$ the extension
of $\boldsymbol{u}_{\varepsilon }$ and $p_{\varepsilon }^{1}$ by zero on $%
G_{\varepsilon }$,we have, 
\begin{equation*}
\frac{\boldsymbol{u}_{\varepsilon }}{\varepsilon ^{2}}\rightarrow 
\boldsymbol{u}_{0}\text{ in }L^{2}(G_{\varepsilon })^{3}\text{-weak }\Sigma
_{\mathcal{A}},\ \ \ \ \ \ \ \ \ \ \ \ \ \ \ 
\end{equation*}%
\begin{equation*}
\frac{1}{\varepsilon }\nabla \boldsymbol{u}_{\varepsilon }\rightarrow 
\overline{\nabla }_{y}\boldsymbol{u}_{0}\text{ in }L^{2}(G_{\varepsilon
})^{3\times 3}\text{-weak }\Sigma _{\mathcal{A}},\ \ \ \ 
\end{equation*}%
\begin{equation*}
p_{\varepsilon }^{0}\rightarrow p_{0}\text{ in }H^{1}(\Omega )\text{-weak
and in }L^{2}(\Omega )\text{-strong,}
\end{equation*}%
Still defining $\boldsymbol{u}$ as in \emph{(i)} above, it holds that 
\begin{equation*}
\left\{ 
\begin{array}{l}
\boldsymbol{u}^{\prime }=\widehat{A}(\boldsymbol{f}_{1}-\nabla _{\overline{x}%
}p_{0})\text{ in }\Omega \\ 
\func{div}_{\overline{x}}\boldsymbol{u}^{\prime }=0\text{ in }\Omega \text{
and }\boldsymbol{u}^{\prime }\cdot \nu =0\text{ on }\partial \Omega ,%
\end{array}%
\right.
\end{equation*}%
where $\widehat{A}=(\widehat{a}_{ij})_{1\leq i,j\leq 2}$ is given by 
\begin{equation*}
\widehat{a}_{ij}=M\left( \int_{h_{1}}^{h_{2}}A\overline{\nabla }_{y}%
\boldsymbol{w}_{i}:\overline{\nabla }_{y}\boldsymbol{w}_{j}dy_{3}\right) ,\
1\leq i,j\leq 2,
\end{equation*}%
with $\boldsymbol{w}_{i}$ $(1\leq i\leq 2)$ being the unique solution in $%
B_{\#,\func{div}}^{1,2}(\mathbb{J})$ of the Stokes system 
\begin{equation*}
\left\{ 
\begin{array}{l}
-\overline{\func{div}}_{y}\left( A(y)\overline{\nabla }_{y}\boldsymbol{w}%
_{i}\right) +\overline{\nabla }_{y}\pi _{i}=e_{i}\text{ in }\mathbb{J} \\ 
\overline{\func{div}}_{y}\boldsymbol{w}_{i}=0\text{ in }\mathbb{J}.%
\end{array}%
\right.
\end{equation*}
\end{itemize}
\end{theorem}

Let us first and foremost compare our result in Theorem \ref{t1.3} with the
existing ones in the literature. In \cite{Pazanin2019}, problem (\ref{1.1})
has been considered in a thin domain with periodic oscillatory boundary. To
be more precise, the authors of \cite{Pazanin2019} considered the problem 
\begin{equation*}
\left\{ 
\begin{array}{l}
-\Delta \boldsymbol{u}_{\varepsilon }+\dfrac{\mu }{K_{\varepsilon }}%
\boldsymbol{u}_{\varepsilon }+\dfrac{\rho }{\phi ^{2}}(\boldsymbol{u}%
_{\varepsilon }\cdot \nabla )\boldsymbol{u}_{\varepsilon }+\nabla
p_{\varepsilon }=\boldsymbol{f}\text{ in }\Lambda ^{\varepsilon }, \\ 
\func{div}\boldsymbol{u}_{\varepsilon }=0\text{ in }\Lambda ^{\varepsilon },
\\ 
\boldsymbol{u}_{\varepsilon }=0\text{ on }\partial \Lambda ^{\varepsilon },%
\end{array}%
\right. 
\end{equation*}%
where $\Lambda _{\varepsilon }=\{(\overline{x},x_{3})\in \mathbb{R}^{3}:%
\overline{x}\in \Omega $\ and $\varepsilon h_{1}(\frac{\overline{x}}{%
\varepsilon })<x_{3}<\varepsilon h_{2}(\frac{\overline{x}}{\varepsilon })\}$%
\ with $h_{1},h_{2}$\ being smooth periodic functions. They proved the
analog of Theorem \ref{t1.3}. However, with respect to the results in \cite%
{Pazanin2019}, a few remarks are in order: 1) In Theorem \ref{t1.3}, the
functions $h_{1}$ and $h_{2}$ assume several behaviours such as the
periodicity, the almost periodicity, the asymptotic periodicity, the
asymptotic almost periodicity, and many more besides. So the main result in 
\cite{Pazanin2019} (Theorem 3.1 for instance) is a special case of Theorem %
\ref{t1.3}. It is also worth noting that, instead of considering only a
periodic framework like in \cite{Pazanin2019}, we consider the case where
the thin domain is highly heterogeneous, the heterogeneities being
distributed inside the domain in a general deterministic fashion. This
imposes the use of a more general concept of convergence taking into account
the properties of the domain, in order to pass to the limit; 2) We also
notice that our estimates leading to the homogenization process are not the
same compared to those obtained in \cite[Sections 2 and 3]{Pazanin2019}.
Indeed, in \cite[Sections 2 and 3]{Pazanin2019}, from the estimates 
\begin{equation*}
\left\Vert \boldsymbol{u}\right\Vert _{L^{2}(\Omega ^{\varepsilon
})^{3}}\leq C\varepsilon ^{\frac{5}{2}}\text{ and }\left\Vert \boldsymbol{u}%
\right\Vert _{L^{2}(\Omega ^{\varepsilon })^{3}}\leq C\varepsilon ^{\frac{3}{%
2}}K_{\varepsilon }^{\frac{1}{2}},
\end{equation*}%
the authors used the following estimate 
\begin{equation*}
\left\Vert \boldsymbol{u}\right\Vert _{L^{2}(\Omega ^{\varepsilon
})^{3}}\leq C\left( \varepsilon ^{\frac{5}{2}}+\varepsilon ^{\frac{3}{2}%
}K_{\varepsilon }^{\frac{1}{2}}\right) ,\ \ \ \ \ \ \ \ \ \ \ \ \ \ \ \ \ \
\ \ \ \ \ \ 
\end{equation*}%
while we rather inferred the following one (which seems to be more
realistic) 
\begin{equation*}
\left\Vert \boldsymbol{u}\right\Vert _{L^{2}(\Omega ^{\varepsilon
})^{3}}\leq C\min \left( \varepsilon ^{\frac{5}{2}},\varepsilon ^{\frac{3}{2}%
}K_{\varepsilon }^{\frac{1}{2}}\right) .\ \ \ \ \ \ \ \ \ \ \ \ \ \ \ \ \ \
\ \ \ \ \ \ 
\end{equation*}%
This gives rise in our work, to uniform estimates different from the ones
they obtained, especially in the cases when $K_{\varepsilon }\ll \varepsilon
^{2}$ and when $K_{\varepsilon }\gg \varepsilon ^{2}$; 3) instead of using
the unfolding periodic method, we use a direct approach: the
sigma-convergence method generalizing the two-scale convergence method. The
use of this approach is also justified by the fact that there is no variant
up to now, of the unfolding operator beyond the periodic setting; 4) In the
proof of Theorem \ref{t1.2}, we derive the estimates of the pressure by
using a classical tool relying on the solution of a divergence equation; see
Lemma \ref{l4.10} and Proposition \ref{p4.2}. However, in the proof of
Theorem \ref{t1.3}, the estimates of the pressure are obtained from a
crucial trick stemming from \cite[Theorem 3.1]{Casado2020}: the
decomposition of the pressure as $L_{0}^{2}(\Omega ^{\varepsilon })\ni
p_{\varepsilon }=p_{\varepsilon }^{0}+\varepsilon p_{\varepsilon }^{1}$ with 
$p_{\varepsilon }^{0}\in H^{1}(\Omega )$ and $p_{\varepsilon }^{1}\in
L^{2}(\Omega ^{\varepsilon })$ (see e.g. (\ref{5.5}) and (\ref{5.6})). This
very important result allows us to work in the same domain $\Omega
^{\varepsilon }$ instead of adding further assumptions on the geometry of $%
\Omega $ as it is the case in \cite{Pazanin2019} (see hypotheses \textbf{H1}%
, \textbf{H2} and \textbf{H3} therein).

Very few works address rigorously the asymptotic analysis of (\ref{1.1})\ or
related problems, in the literature. To the best of our knowledge the only
ones available to date are \cite{Ang3, Ang4, Casado2020, Pazanin2019,
Suarez2015}. For other works in thin heterogeneous domains, we may refer to
e.g. \cite{Ang1, Ang2, Renata, CLS2013, GMP2017, Griebel, Huang,
Marusic2000, RJ2007, Popov}, to cite a few. In \cite{Griebel, Huang, Popov}
the authors consider the formal asymptotic developments to upscale a
Stokes-Brinkmann model.

The rest of the work is organized as follows. Section \ref{sec2} deals with
some fundamental tools around the concept of algebra with mean value. In
Section \ref{sec3}, we are concerned with the concept of sigma-convergence
in thin heterogeneous domains. We prove therein some compactness results
related to that concept such as Theorems \ref{t1.1} and \ref{t1.4}. We apply
in Sections \ref{sec4} and \ref{sec5} the results developed in the previous
section to upscale a Darcy-Lapwood-Brinkmann flow occurring in thin
heterogeneous layers. Finally, in Section \ref{sec6}, we provide some
concrete applications of Theorems \ref{t1.2} and \ref{t1.3}.

Unless otherwise specified, the vector spaces throughout are assumed to be
real vector spaces, and the scalar functions are assumed to take real
values. We shall always assume that the numerical space $\mathbb{R}^{m}$
(integer $m\geq 1$) and its open sets are each provided with the Lebesgue
measure denoted by $dx=dx_{1}...dx_{m}$.

\textbf{Note}. This work inspired the works \cite{CJW2024, PW2024} where the
evolutionary counterpart of results contained in Subsection \ref{subsec3.1}
have been considered.

\section{Algebras with mean value and related Sobolev-type spaces\label{sec2}%
}

In this section we gather the reader with some basic concepts about the
algebras with mean value \cite{Jikov, ZK} and the associated Sobolev-type
spaces \cite{JW2021, CMP, NA}.

Let $\mathcal{A}$ be an algebra with mean value (algebra wmv in short) on $%
\mathbb{R}^{m}$ (integer $m\geq 1$) \cite{Jikov, ZK}, that is, a closed
subalgebra of the $\mathcal{C}^{\ast }$-algebra of bounded uniformly
continuous real-valued functions on $\mathbb{R}^{m}$, $\mathrm{BUC}(\mathbb{R%
}^{m})$, which contains the constants, is translation invariant and is such
that any of its elements possesses a mean value in the following sense: for
every $u\in \mathcal{A}$, the sequence $(u^{\varepsilon })_{\varepsilon >0}$
($u^{\varepsilon }(x)=u(x/\varepsilon )$) weakly $\ast $-converges in $%
L^{\infty }(\mathbb{R}^{m})$ to some real number $M(u)$ (called the mean
value of $u$) as $\varepsilon \rightarrow 0$. The mean value expresses as 
\begin{equation}
M(u)=\lim_{R\rightarrow \infty }%
\mathchoice {{\setbox0=\hbox{$\displaystyle{\textstyle -}{\int}$ } \vcenter{\hbox{$\textstyle -$
}}\kern-.6\wd0}}{{\setbox0=\hbox{$\textstyle{\scriptstyle -}{\int}$ } \vcenter{\hbox{$\scriptstyle -$
}}\kern-.6\wd0}}{{\setbox0=\hbox{$\scriptstyle{\scriptscriptstyle -}{\int}$
} \vcenter{\hbox{$\scriptscriptstyle -$
}}\kern-.6\wd0}}{{\setbox0=\hbox{$\scriptscriptstyle{\scriptscriptstyle
-}{\int}$ } \vcenter{\hbox{$\scriptscriptstyle -$ }}\kern-.6\wd0}}%
\!\int_{B_{R}}u(y)dy\text{ for }u\in \mathcal{A}  \label{0.1}
\end{equation}%
where we have set $%
\mathchoice {{\setbox0=\hbox{$\displaystyle{\textstyle
-}{\int}$ } \vcenter{\hbox{$\textstyle -$
}}\kern-.6\wd0}}{{\setbox0=\hbox{$\textstyle{\scriptstyle -}{\int}$ } \vcenter{\hbox{$\scriptstyle -$
}}\kern-.6\wd0}}{{\setbox0=\hbox{$\scriptstyle{\scriptscriptstyle -}{\int}$
} \vcenter{\hbox{$\scriptscriptstyle -$
}}\kern-.6\wd0}}{{\setbox0=\hbox{$\scriptscriptstyle{\scriptscriptstyle
-}{\int}$ } \vcenter{\hbox{$\scriptscriptstyle -$ }}\kern-.6\wd0}}%
\!\int_{B_{R}}=\left\vert B_{R}\right\vert ^{-1}\int_{B_{R}}$.

To an algebra with mean value $\mathcal{A}$ are associated its regular
subalgebras $\mathcal{A}^{k}=\{\psi \in \mathcal{C}^{k}(\mathbb{R}^{m}):$ $%
D_{y}^{\alpha }\psi \in \mathcal{A}$ $\forall \alpha =(\alpha
_{1},...,\alpha _{m})\in \mathbb{N}^{m}$ with $\left\vert \alpha \right\vert
\leq k\}$ ($k\geq 0$ an integer with $\mathcal{A}^{0}=\mathcal{A}$, and $%
D_{y}^{\alpha }\psi =\frac{\partial ^{\left\vert \alpha \right\vert }\psi }{%
\partial y_{1}^{\alpha _{1}}\cdot \cdot \cdot \partial y_{m}^{\alpha _{m}}}$%
). Under the norm $\left\Vert \left\vert u\right\vert \right\Vert
_{k}=\sup_{\left\vert \alpha \right\vert \leq k}\left\Vert D_{y}^{\alpha
}\psi \right\Vert _{\infty }$, $\mathcal{A}^{k}$ is a Banach space. We also
define the space $\mathcal{A}^{\infty }=\{\psi \in \mathcal{C}^{\infty }(%
\mathbb{R}^{m}):$ $D_{y}^{\alpha }\psi \in \mathcal{A}$ $\forall \alpha
=(\alpha _{1},...,\alpha _{m})\in \mathbb{N}^{m}\}$, a Fr\'{e}chet space
when endowed with the locally convex topology defined by the family of norms 
$\left\Vert \left\vert \cdot \right\vert \right\Vert _{m}$. The space $%
\mathcal{A}^{\infty }$ is dense in any $\mathcal{A}^{m}$ (integer $m\geq 0$).

The notion of a vector-valued algebra with mean value will be very useful in
this study.

Let $\mathbb{F}$ be a Banach space. We denote by \textrm{BUC}$(\mathbb{R}%
^{m};\mathbb{F})$ the Banach space of bounded uniformly continuous functions 
$u:\mathbb{R}^{m}\rightarrow \mathbb{F}$, endowed with the norm 
\begin{equation*}
\left\Vert u\right\Vert _{\infty }=\sup_{y\in \mathbb{R}^{m}}\left\Vert
u(y)\right\Vert _{\mathbb{F}}
\end{equation*}%
where $\left\Vert \cdot \right\Vert _{\mathbb{F}}$ stands for the norm in $%
\mathbb{F}$. Let $\mathcal{A}$ be an algebra with mean value on $\mathbb{R}%
^{m}$. We denote by $\mathcal{A}\otimes \mathbb{F}$ the usual space of
functions of the form 
\begin{equation*}
\sum_{\text{finite}}u_{i}\otimes v_{i}\text{ with }u_{i}\in \mathcal{A}\text{
and }v_{i}\in \mathbb{F}
\end{equation*}%
where $(u_{i}\otimes v_{i})(y)=u_{i}(y)v_{i}$ for $y\in \mathbb{R}^{m}$.
With this in mind, we define the vector-valued algebra with mean value $%
\mathcal{A}(\mathbb{R}^{m};\mathbb{F})$ as the closure of $\mathcal{A}%
\otimes \mathbb{F}$ in \textrm{BUC}$(\mathbb{R}^{m};\mathbb{F})$. Then it
holds that (see \cite{CMP}), for any $f\in \mathcal{A}(\mathbb{R}^{m};%
\mathbb{F})$, the set $\{L(f):L\in \mathbb{F}^{\prime }$ with $\left\Vert
L\right\Vert _{\mathbb{F}^{\prime }}\leq 1\}$ is relatively compact in $%
\mathcal{A}$.

Let us note that we may still define the space $\mathcal{A}(\mathbb{R}^{m};%
\mathbb{F})$ where $\mathbb{F}$ in this case is a Fr\'{e}chet space. In that
case, we replace the norm by the family of seminorms defining the topology
of $\mathbb{F}$.

Now, let $f\in \mathcal{A}(\mathbb{R}^{m};\mathbb{F})$. Then, defining $%
\left\Vert f\right\Vert _{\mathbb{F}}$ by $\left\Vert f\right\Vert _{\mathbb{%
F}}(y)=\left\Vert f(y)\right\Vert _{\mathbb{F}}$ ($y\in \mathbb{R}^{m}$), we
have that $\left\Vert f\right\Vert _{\mathbb{F}}\in \mathcal{A}$. Similarly
we can define (for $0<p<\infty $) the function $\left\Vert f\right\Vert _{%
\mathbb{F}}^{p}$ and $\left\Vert f\right\Vert _{\mathbb{F}}^{p}\in \mathcal{A%
}$. This allows us to define the Besicovitch seminorm on $\mathcal{A}(%
\mathbb{R}^{m};\mathbb{F})$ as follows: for $1\leq p<\infty $, we define the
Marcinkiewicz-type space $\mathfrak{M}^{p}(\mathbb{R}^{m};\mathbb{F})$ to be
the vector space of functions $u\in L_{loc}^{p}(\mathbb{R}^{m};\mathbb{F})$
such that 
\begin{equation*}
\left\Vert u\right\Vert _{p}=\left( \underset{R\rightarrow \infty }{\lim
\sup }%
\mathchoice {{\setbox0=\hbox{$\displaystyle{\textstyle
-}{\int}$ } \vcenter{\hbox{$\textstyle -$
}}\kern-.6\wd0}}{{\setbox0=\hbox{$\textstyle{\scriptstyle -}{\int}$ } \vcenter{\hbox{$\scriptstyle -$
}}\kern-.6\wd0}}{{\setbox0=\hbox{$\scriptstyle{\scriptscriptstyle -}{\int}$
} \vcenter{\hbox{$\scriptscriptstyle -$
}}\kern-.6\wd0}}{{\setbox0=\hbox{$\scriptscriptstyle{\scriptscriptstyle
-}{\int}$ } \vcenter{\hbox{$\scriptscriptstyle -$ }}\kern-.6\wd0}}%
\!\int_{B_{R}}\left\Vert u(y)\right\Vert _{\mathbb{F}}^{p}dy\right) ^{\frac{1%
}{p}}<\infty
\end{equation*}%
where $B_{R}$ is the open ball in $\mathbb{R}^{m}$ centered at the origin
and of radius $R$. Under the seminorm $\left\Vert \cdot \right\Vert _{p}$, $%
\mathfrak{M}^{p}(\mathbb{R}^{m};\mathbb{F})$ is a complete seminormed space
with the property that $\mathcal{A}(\mathbb{R}^{m};\mathbb{F})\subset 
\mathfrak{M}^{p}(\mathbb{R}^{m};\mathbb{F})$ since $\left\Vert u\right\Vert
_{p}<\infty $ for any $u\in \mathcal{A}(\mathbb{R}^{m};\mathbb{F})$. We
therefore define the vector-valued generalized Besicovitch space $B_{%
\mathcal{A}}^{p}(\mathbb{R}^{m};\mathbb{F})$ as the closure of $\mathcal{A}(%
\mathbb{R}^{m};\mathbb{F})$ in $\mathfrak{M}^{p}(\mathbb{R}^{m};\mathbb{F})$%
. The following hold true \cite{CMP, NA}:

\begin{itemize}
\item[(\textbf{i)}] The space $\mathcal{B}_{\mathcal{A}}^{p}(\mathbb{R}^{m};%
\mathbb{F})=B_{\mathcal{A}}^{p}(\mathbb{R}^{m};\mathbb{F})/\mathcal{N}$
(where $\mathcal{N}=\{u\in B_{\mathcal{A}}^{p}(\mathbb{R}^{m};\mathbb{F}%
):\left\Vert u\right\Vert _{p}=0\}$) is a Banach space under the norm $%
\left\Vert u+\mathcal{N}\right\Vert _{p}=\left\Vert u\right\Vert _{p}$ for $%
u\in B_{\mathcal{A}}^{p}(\mathbb{R}^{m};\mathbb{F})$.

\item[(\textbf{ii)}] The mean value $M:\mathcal{A}(\mathbb{R}^{m};\mathbb{F}%
)\rightarrow \mathbb{F}$ extends by continuity to a continuous linear
mapping (still denoted by $M$) on $B_{\mathcal{A}}^{p}(\mathbb{R}^{m};%
\mathbb{F})$ satisfying 
\begin{equation}
L(M(u))=M(L(u))\text{ for all }L\in \mathbb{F}^{\prime }\text{ and }u\in B_{%
\mathcal{A}}^{p}(\mathbb{R}^{m};\mathbb{F}).  \label{e2.1}
\end{equation}%
Moreover, for $u\in B_{\mathcal{A}}^{p}(\mathbb{R}^{m};\mathbb{F})$ we have 
\begin{equation*}
\left\Vert u\right\Vert _{p}=\left[ M(\left\Vert u\right\Vert _{\mathbb{F}%
}^{p})\right] ^{1/p}\equiv \left[ \lim_{R\rightarrow \infty }%
\mathchoice {{\setbox0=\hbox{$\displaystyle{\textstyle
-}{\int}$ } \vcenter{\hbox{$\textstyle -$
}}\kern-.6\wd0}}{{\setbox0=\hbox{$\textstyle{\scriptstyle -}{\int}$ } \vcenter{\hbox{$\scriptstyle -$
}}\kern-.6\wd0}}{{\setbox0=\hbox{$\scriptstyle{\scriptscriptstyle -}{\int}$
} \vcenter{\hbox{$\scriptscriptstyle -$
}}\kern-.6\wd0}}{{\setbox0=\hbox{$\scriptscriptstyle{\scriptscriptstyle
-}{\int}$ } \vcenter{\hbox{$\scriptscriptstyle -$ }}\kern-.6\wd0}}%
\!\int_{B_{R}}\left\Vert u(y)\right\Vert _{\mathbb{F}}^{p}dy\right] ^{\frac{1%
}{p}},
\end{equation*}%
and for $u\in \mathcal{N}$ one has $M(u)=0$.
\end{itemize}

It is worth noticing that $\mathcal{B}_{\mathcal{A}}^{2}(\mathbb{R}^{m};H)$
(when $\mathbb{F}=H$ is a Hilbert space) is a Hilbert space with inner
product 
\begin{equation}
\left( u,v\right) _{2}=M\left[ \left( u,v\right) _{H}\right] \text{ for }%
u,v\in \mathcal{B}_{\mathcal{A}}^{2}(\mathbb{R}^{m};H),  \label{1.5}
\end{equation}%
$(~,~)_{H}$ denoting the inner product in $H$ and $\left( u,v\right) _{H}$
the function $y\mapsto \left( u(y),v(y)\right) _{H}$ from $\mathbb{R}^{m}$
to $\mathbb{R}$, which belongs to $\mathcal{B}_{\mathcal{A}}^{1}(\mathbb{R}%
^{m};\mathbb{R})$.

We also define the Sobolev-Besicovitch type spaces as follows: 
\begin{equation*}
B_{\mathcal{A}}^{1,p}(\mathbb{R}^{m};\mathbb{F})=\{u\in B_{\mathcal{A}}^{p}(%
\mathbb{R}^{m};\mathbb{F}):\nabla _{y}u\in (B_{\mathcal{A}}^{p}(\mathbb{R}%
^{m};\mathbb{F}))^{m}\},
\end{equation*}%
endowed with the seminorm 
\begin{equation*}
\left\Vert u\right\Vert _{1,p}=\left( \left\Vert u\right\Vert
_{p}^{p}+\left\Vert \nabla _{y}u\right\Vert _{p}^{p}\right) ^{\frac{1}{p}},
\end{equation*}%
which is a complete seminormed space. The Banach counterpart of $B_{\mathcal{%
A}}^{1,p}(\mathbb{R}^{m};\mathbb{F})$ denoted by $\mathcal{B}_{\mathcal{A}%
}^{1,p}(\mathbb{R}^{m};\mathbb{F})$ is defined by replacing $B_{\mathcal{A}%
}^{p}(\mathbb{R}^{m};\mathbb{F})$ by $\mathcal{B}_{\mathcal{A}}^{p}(\mathbb{R%
}^{m};\mathbb{F})$ and $\partial /\partial y_{i}$ by $\overline{\partial }%
/\partial y_{i}$, where $\overline{\partial }/\partial y_{i}$ is defined by 
\begin{equation}
\frac{\overline{\partial }}{\partial y_{i}}(u+\mathcal{N}):=\frac{\partial u%
}{\partial y_{i}}+\mathcal{N}\text{ for }u\in B_{\mathcal{A}}^{1,p}(\mathbb{R%
}^{m};\mathbb{F}).  \label{e3}
\end{equation}%
It is important to note that $\overline{\partial }/\partial y_{i}$ is also
defined as the infinitesimal generator in the $i$th direction coordinate of
the strongly continuous group $\mathcal{T}(y):\mathcal{B}_{\mathcal{A}}^{p}(%
\mathbb{R}^{m};\mathbb{F})\rightarrow \mathcal{B}_{\mathcal{A}}^{p}(\mathbb{R%
}^{m};\mathbb{F});\ \mathcal{T}(y)(u+\mathcal{N})=u(\cdot +y)+\mathcal{N}$.
Let us denote by $\varrho :B_{\mathcal{A}}^{p}(\mathbb{R}^{m};\mathbb{F}%
)\rightarrow \mathcal{B}_{\mathcal{A}}^{p}(\mathbb{R}^{m};\mathbb{F})=B_{%
\mathcal{A}}^{p}(\mathbb{R}^{m};\mathbb{F})/\mathcal{N}$, $\varrho (u)=u+%
\mathcal{N}$, the canonical surjection. We remark that if $u\in B_{\mathcal{A%
}}^{1,p}(\mathbb{R}^{m};\mathbb{F})$ then $\varrho (u)\in \mathcal{B}_{%
\mathcal{A}}^{1,p}(\mathbb{R}^{m};\mathbb{F})$ with further 
\begin{equation*}
\frac{\overline{\partial }\varrho (u)}{\partial y_{i}}=\varrho \left( \frac{%
\partial u}{\partial y_{i}}\right) ,
\end{equation*}%
as seen above in (\ref{e3}). We may also define higher order derivatives $%
\overline{D}_{y}^{\alpha }=\frac{\overline{\partial }^{\left\vert \alpha
\right\vert }}{\partial y_{1}^{\alpha _{1}}\cdot \cdot \cdot \partial
y_{m}^{\alpha _{m}}}$ where $\frac{\overline{\partial }^{\alpha _{j}}}{%
\partial y_{i}^{\alpha _{j}}}=\frac{\overline{\partial }}{\partial y_{i}}%
\circ ...\circ \frac{\overline{\partial }}{\partial y_{i}}$,$\ j$ times.

We set $\mathcal{D}_{\mathcal{A}}(\mathbb{R}^{m};\mathbb{F})=\varrho (%
\mathcal{A}^{\infty }(\mathbb{R}^{m};\mathbb{F}))$ (and merely $\mathcal{D}_{%
\mathcal{A}}(\mathbb{R}^{m})=\mathcal{D}_{\mathcal{A}}(\mathbb{R}^{m};%
\mathbb{R})$), and we define the space of vector-valued distributions on $%
\mathcal{A}$, denoted by $\mathcal{D}_{\mathcal{A}}^{\prime }(\mathbb{R}^{m};%
\mathbb{F})$, as the space of continuous linear functionals $L:\mathcal{D}_{%
\mathcal{A}}(\mathbb{R}^{m})\rightarrow \mathbb{F}$. Let $1\leq p<\infty $;
for $u\in \mathcal{B}_{\mathcal{A}}^{p}(\mathbb{R}^{m};\mathbb{F})$, define $%
L_{u}:\mathcal{D}_{\mathcal{A}}(\mathbb{R}^{m})\rightarrow \mathbb{F}$ by 
\begin{equation*}
\left\langle L_{u},\varphi \right\rangle =M(\varphi u)\text{, all }\varphi
\in \mathcal{D}_{\mathcal{A}}(\mathbb{R}^{m}).\ \ \ \ \ \ \ \ \ \ \ 
\end{equation*}%
Then it is easy to see that $L_{u}\in \mathcal{D}_{\mathcal{A}}^{\prime }(%
\mathbb{R}^{m};\mathbb{F})$, so that $\mathcal{B}_{\mathcal{A}}^{p}(\mathbb{R%
}^{m};\mathbb{F})\hookrightarrow \mathcal{D}_{\mathcal{A}}^{\prime }(\mathbb{%
R}^{m};\mathbb{F})$. The mapping $u\mapsto L_{u}$ is linear continuous and
injective (this can be proven by proceeding as in \cite{Lions1968}). This
allows us to view $u$ as an element of $\mathcal{D}_{\mathcal{A}}^{\prime }(%
\mathbb{R}^{m};\mathbb{F})$ (say $L_{u}$), so that $\left\langle
L_{u},\varphi \right\rangle =\left\langle u,\varphi \right\rangle =M(\varphi
u)$ for all $\varphi \in \mathcal{D}_{\mathcal{A}}(\mathbb{R}^{m})$.
Especially, for $\varphi =1$, we have $M(u)=\left\langle u,1\right\rangle $.
We may therefore define the mean value of a a distribution $L\in \mathcal{D}%
_{\mathcal{A}}^{\prime }(\mathbb{R}^{m};\mathbb{F})$ accordingly: 
\begin{equation}
M(L)=\left\langle L,1\right\rangle .\ \ \ \ \ \ \ \ \ \ \ \ \ \ \ \ \ \ \ \
\ \ \ \ \ \ \ \ \ \ \ \ \ \ \ \ \ \ \ \ \ \ \ \ \ \   \label{e2.4}
\end{equation}%
For $L\in \mathcal{D}_{\mathcal{A}}^{\prime }(\mathbb{R}^{m};\mathbb{F})$
and $\alpha \in \mathbb{N}^{m}$, we define the partial derivative $\overline{%
D}_{\overline{y}}^{\alpha }L\in \mathcal{D}_{\mathcal{A}}^{\prime }(\mathbb{R%
}^{m};\mathbb{F})$ as follows: 
\begin{equation}
\left\langle \overline{D}_{\overline{y}}^{\alpha }L,\varphi \right\rangle
=(-1)^{\left\vert \alpha \right\vert }\left\langle L,\overline{D}_{\overline{%
y}}^{\alpha }\varphi \right\rangle \ \ \forall \varphi \in \mathcal{D}_{%
\mathcal{A}}(\mathbb{R}^{m}).  \label{e2.5}
\end{equation}%
From (\ref{e2.4}) and (\ref{e2.5}), we observe that 
\begin{equation}
M(\overline{D}_{\overline{y}}^{\alpha }L)=0\text{ for all }\alpha \in 
\mathbb{N}^{m}\backslash \{0\}.  \label{e2.6}
\end{equation}

We define a further notion by restricting ourselves to the case $\mathbb{F}=%
\mathbb{R}$. We say that the algebra $\mathcal{A}$ is ergodic if any $u\in 
\mathcal{B}_{\mathcal{A}}^{1}(\mathbb{R}^{m};\mathbb{R})$ that is invariant
under $(\mathcal{T}(y))_{y\in \mathbb{R}^{m}}$ is a constant in $\mathcal{B}%
_{\mathcal{A}}^{1}(\mathbb{R}^{m};\mathbb{R})$: this amounts to, if $%
\mathcal{T}(y)u=u$\ in $\mathcal{B}_{\mathcal{A}}^{1}(\mathbb{R}^{m};\mathbb{%
R})$ for every $y\in \mathbb{R}^{m}$, then $u=c$\ in $\mathcal{B}_{\mathcal{A%
}}^{1}(\mathbb{R}^{m};\mathbb{R})$ in the sense that $\left\Vert
u-c\right\Vert _{1}=0$, $c$ being a constant.

We end this subsection by defining the \textit{corrector function}\ spaces.
We are concerned with two special choices of the space $\mathbb{F}$: $%
\mathbb{F}=\mathbb{R}$ or $\mathbb{F}=W^{1,p}(G)$, $G$ being an open subset
of $\mathbb{R}^{N}$ (integer $N\geq 1$).

\begin{enumerate}
\item If $\mathbb{F}=\mathbb{R}$, we denote by $B_{\#\mathcal{A}}^{1,p}(%
\mathbb{R}^{m})\equiv B_{\#\mathcal{A}}^{1,p}(\mathbb{R}^{m};\mathbb{R})$
the \textit{corrector function}\ space defined by 
\begin{equation*}
B_{\#\mathcal{A}}^{1,p}(\mathbb{R}^{m})=\{u\in W_{loc}^{1,p}(\mathbb{R}%
^{m}):\nabla u\in B_{\mathcal{A}}^{p}(\mathbb{R}^{m})^{m}\text{ and }%
M(\nabla u)=0\}\text{.}
\end{equation*}%
In $B_{\#\mathcal{A}}^{1,p}(\mathbb{R}^{m})$ we identify two elements by
their gradients: $u=v$ in $B_{\#\mathcal{A}}^{1,p}(\mathbb{R}^{m})$ iff $%
\nabla _{y}(u-v)=0$, i.e. $\left\Vert \nabla _{y}(u-v)\right\Vert _{p}=0$.
We may therefore equip $B_{\#\mathcal{A}}^{1,p}(\mathbb{R}^{m})$ with the
gradient norm $\left\Vert u\right\Vert _{\#,p}=\left\Vert \nabla
_{y}u\right\Vert _{p}$. This defines a Banach space \cite[Theorem 3.12]%
{Casado} containing $B_{\mathcal{A}}^{1,p}(\mathbb{R}^{m};\mathbb{R})=B_{%
\mathcal{A}}^{1,p}(\mathbb{R}^{m})$ as a subspace.

\item For $\mathbb{F}=W^{1,p}(G)$, we define the \textit{corrector function}%
\ space $B_{\#\mathcal{A}}^{1,p}(\mathbb{R}^{m};W^{1,p}(G))$ by $B_{\#%
\mathcal{A}}^{1,p}(\mathbb{R}^{m};W^{1,p}(G))=\{u\in W_{loc}^{1,p}(\mathbb{R}%
^{m};W^{1,p}(G)):\nabla u\in B_{\mathcal{A}}^{p}(\mathbb{R}%
^{m};L^{p}(G))^{m+N}$\ and $\int_{G}M(\nabla u(\cdot ,\zeta ))d\zeta =0\}$,
where in this case $\nabla =(\nabla _{y},\nabla _{\zeta })$, $\nabla _{y}$
(resp. $\nabla _{\zeta }$) being the gradient operator with respect to the
variable $y\in \mathbb{R}^{m}$ (resp. $\zeta \in \mathbb{R}^{N}$). As in $%
B_{\#\mathcal{A}}^{1,p}(\mathbb{R}^{m})$, we still identify two elements by
their gradients in the sense that: $u=v$ in $B_{\#\mathcal{A}}^{1,p}(\mathbb{%
R}^{m};W^{1,p}(G))$ iff $\nabla (u-v)=0$, i.e. $\int_{G}\left\Vert \nabla
(u(\cdot ,\zeta )-v(\cdot ,\zeta ))\right\Vert _{p}^{p}d\zeta =0$. The space 
$B_{\#\mathcal{A}}^{1,p}(\mathbb{R}^{m};W^{1,p}(G))$ is therefore a Banach
space under the norm $\left\Vert u\right\Vert _{\#,p}=\left(
\int_{G}\left\Vert \nabla u(\cdot ,\zeta )\right\Vert _{p}^{p}d\zeta \right)
^{1/p}$.
\end{enumerate}

\section{Sigma-convergence for thin heterogeneous domains\label{sec3}}

\subsection{Sigma-convergence in thin heterogeneous domains with flat
lateral boundaries\label{subsec3.1}}

The concept of sigma-convergence was introduced in \cite{Hom1} in order to
tackle two-scale phenomena occurring in media with microstructures that are
distributed inside in a general deterministic way such as the periodic
distribution, the almost periodic one and others. The concept was concerned
with two-scale phenomena taking place in all space dimensions. In the
special case of periodic structures, it has been generalized to thin
heterogeneous media \cite{RJ2007}.

Our aim in this work is to provide a systematic study of the concept of
sigma-convergence applied to thin heterogeneous domains whose heterogeneous
structure is of general deterministic type including the periodic one and
the almost periodic one as special cases. The compactness results obtained
here generalize therefore those in \cite{RJ2007} which are concerned only
with periodic structures.

More precisely, let $d\geq 2$ be a given integer with $d=d_{1}+d_{2}$, $%
d_{1},d_{2}\geq 1$ being integers. Let $G_{1}$ and $G_{2}$ be open bounded
sets in $\mathbb{R}^{d_{1}}$ and $\mathbb{R}^{d_{2}}$, respectively. We
assume that $0\in \overline{G}_{2}$ (the closure $G_{2}$ in $\mathbb{R}%
^{d_{2}}$), where $0$ stands for the origin in $\mathbb{R}^{d_{2}}$. For $%
\varepsilon >0$ a given small parameter, we define our thin domain by 
\begin{equation*}
G_{\varepsilon }=G_{1}\times \varepsilon G_{2}.
\end{equation*}%
When $\varepsilon \rightarrow 0$, $G_{\varepsilon }$ shrinks to the
"interface" 
\begin{equation*}
G_{0}=G_{1}\times \{0\}.
\end{equation*}%
We note that in the definition of $G_{\varepsilon }$, the structure is
heterogeneous in $d_{1}$ space dimension only.

The space $\mathbb{R}_{\xi }^{m}$ is the numerical space $\mathbb{R}^{m}$ of
generic variable $\xi $. In this regard we set $\mathbb{R}^{d_{1}}=\mathbb{R}%
_{\overline{x}}^{d_{1}}$ or $\mathbb{R}_{\overline{y}}^{d_{1}}$ and $\mathbb{%
R}^{d_{2}}=\mathbb{R}_{\zeta }^{d_{2}}$, so that $x\in \mathbb{R}^{d}$
writes $(\overline{x},\zeta )$. We identify $G_{0}$ with $G_{1}$ so that the
generic element in $G_{0}$ is also denoted by $\overline{x}$ instead of $(%
\overline{x},0)$. Finally we set $G=G_{1}\times G_{2}\equiv G_{\varepsilon
=1}$.

This being so, let $\mathcal{A}$ be an algebra with mean value on $\mathbb{R}%
^{d_{1}}$. We denote by $M$ the mean value on $\mathcal{A}$ as well as its
extension on the associated generalized Besicovitch spaces $B_{\mathcal{A}%
}^{p}(\mathbb{R}^{d_{1}};L^{p}(G_{2}))$ and $\mathcal{B}_{\mathcal{A}}^{p}(%
\mathbb{R}^{d_{1}};L^{p}(G_{2}))$, $1\leq p<\infty $.

We can now introduce the concept of sigma-convergence for thin heterogeneous
domains.

\begin{definition}
\label{d2.1}\emph{1) A sequence }$(u_{\varepsilon })_{\varepsilon >0}\subset
L^{p}(G_{\varepsilon })$\emph{\ }$(1\leq p<\infty )$\emph{\ is said to }%
weakly $\Sigma $-converge\emph{\ in }$L^{p}(G_{\varepsilon })$\emph{\ to
some }$u_{0}\in L^{p}(G_{0};\mathcal{B}_{\mathcal{A}}^{p}(\mathbb{R}%
^{d_{1}};L^{p}(G_{2})))$\emph{\ if as }$\varepsilon \rightarrow 0$\emph{, we
have }%
\begin{equation}
\varepsilon ^{-d_{2}}\int_{G_{\varepsilon }}u_{\varepsilon }(x)f\left( 
\overline{x},\frac{x}{\varepsilon }\right) dx\rightarrow
\int_{G_{0}}\int_{G_{2}}M(u_{0}(\overline{x},\cdot ,\zeta )f(\overline{x}%
,\cdot ,\zeta ))d\zeta d\overline{x}  \label{2.1}
\end{equation}%
\emph{for any }$f\in L^{p^{\prime }}(G_{0};\mathcal{A}(\mathbb{R}%
^{d_{1}};L^{p^{\prime }}(G_{2})))$\emph{\ }$(\frac{1}{p^{\prime }}=1-\frac{1%
}{p})$\emph{; we denote this by }$"u_{\varepsilon }\rightarrow u_{0}$\emph{\
in }$L^{p}(G_{\varepsilon })$\emph{-weak }$\Sigma _{\mathcal{A}}"$\emph{.}

\emph{2) The sequence }$(u_{\varepsilon })_{\varepsilon >0}\subset
L^{p}(G_{\varepsilon })$\emph{\ is said to strongly }$\Sigma $\emph{%
-converge in }$L^{p}(G_{\varepsilon })$\emph{\ to some }$u_{0}\in
L^{p}(G_{0};\mathcal{B}_{\mathcal{A}}^{p}(\mathbb{R}^{d_{1}};L^{p}(G_{2})))$%
\emph{\ if it is weakly }$\Sigma $\emph{-convergent to }$u_{0}$\emph{\ and
further satisfies, as }$\varepsilon \rightarrow 0$\emph{, }%
\begin{equation}
\varepsilon ^{-\frac{d_{2}}{p}}\left\Vert u_{\varepsilon }\right\Vert
_{L^{p}(G_{\varepsilon })}\rightarrow \left\Vert u_{0}\right\Vert
_{L^{p}(G_{0};\mathcal{B}_{\mathcal{A}}^{p}(\mathbb{R}%
^{d_{1}};L^{p}(G_{2})))}.  \label{2.2}
\end{equation}%
\emph{We express this by writing }$"u_{\varepsilon }\rightarrow u_{0}$\emph{%
\ in }$L^{p}(G_{\varepsilon })$\emph{-strong }$\Sigma _{\mathcal{A}}"$\emph{.%
}
\end{definition}

\begin{remark}
\label{r2.1}\emph{It is easy to see that if }$u_{0}\in L^{p}(G_{0};\mathcal{A%
}(\mathbb{R}^{d_{1}};L^{p}(G_{2})))$\emph{\ then (\ref{2.2}) is equivalent
to }%
\begin{equation}
\varepsilon ^{-\frac{d_{2}}{p}}\left\Vert u_{\varepsilon
}-u_{0}^{\varepsilon }\right\Vert _{L^{p}(G_{\varepsilon })}\rightarrow 0%
\text{\emph{\ as }}\varepsilon \rightarrow 0,  \label{2.3}
\end{equation}%
\emph{where }$u_{0}^{\varepsilon }(x)=u_{0}(\overline{x},x/\varepsilon )$%
\emph{\ for }$x\in G_{\varepsilon }$\emph{.}
\end{remark}

Before we state the first compactness result for this section, we need a
further notation. Throughout the work, the letter $E$ will stand for any
ordinary sequence $(\varepsilon _{n})_{n\geq 1}$ with $0<\varepsilon
_{n}\leq 1$ and $\varepsilon _{n}\rightarrow 0$ when $n\rightarrow \infty $.
The generic term of $E$ will be merely denoted by $\varepsilon $ and $%
\varepsilon \rightarrow 0$ will mean $\varepsilon _{n}\rightarrow 0$ as $%
n\rightarrow \infty $. This being so, the following result holds true.

\begin{theorem}
\label{t2.1}Let $(u_{\varepsilon })_{\varepsilon \in E}$ be a sequence in $%
L^{p}(G_{\varepsilon })$ $(1<p<\infty )$ such that 
\begin{equation*}
\sup_{\varepsilon \in E}\varepsilon ^{-d_{2}/p}\left\Vert u_{\varepsilon
}\right\Vert _{L^{p}(G_{\varepsilon })}\leq C
\end{equation*}%
where $C$ is a positive constant independent of $\varepsilon $. Then there
exists a subsequence $E^{\prime }$ of $E$ such that the sequence $%
(u_{\varepsilon })_{\varepsilon \in E^{\prime }}$ weakly $\Sigma $-converges
in $L^{p}(G_{\varepsilon })$ to some $u_{0}\in L^{p}(G_{0};\mathcal{B}_{%
\mathcal{A}}^{p}(\mathbb{R}^{d_{1}};L^{p}(G_{2})))$.
\end{theorem}

The proof of the above theorem relies on the following two lemmas.

\begin{lemma}
\label{l2.1}Let $1\leq p<\infty $. For any $f\in L^{p}(G_{0};\mathcal{A}(%
\mathbb{R}^{d_{1}};L^{p}(G_{2})))$ one has 
\begin{equation}
\varepsilon ^{-d_{2}}\int_{G_{\varepsilon }}\left\vert f\left( \overline{x},%
\frac{x}{\varepsilon }\right) \right\vert ^{p}dx\leq \left\Vert f\right\Vert
_{L^{p}(G_{0};\mathcal{A}(\mathbb{R}^{d_{1}};L^{p}(G_{2})))}^{p}  \label{2.4}
\end{equation}%
and 
\begin{equation}
\varepsilon ^{-d_{2}}\int_{G_{\varepsilon }}\left\vert f\left( \overline{x},%
\frac{x}{\varepsilon }\right) \right\vert ^{p}dx\rightarrow
\int_{G_{0}}\int_{G_{2}}M(\left\vert f(\overline{x},\cdot ,\zeta
)\right\vert ^{p})d\zeta d\overline{x}.  \label{2.5}
\end{equation}
\end{lemma}

\begin{proof}
The proof of (\ref{2.4}) is obvious by making the change of variables $\zeta
=\frac{\widehat{x}}{\varepsilon }\in G_{2}$, where $\widehat{x}%
=(x_{d_{1}+1},\ldots ,x_{d})$ for $x=(x_{1},\ldots
,x_{d_{1}},x_{d_{1}+1},\ldots ,x_{d})$. Let us now turn our attention on (%
\ref{2.5}). To this end, let $f\in L^{p}(G_{0};\mathcal{A}(\mathbb{R}%
^{d_{1}};L^{p}(G_{2})))$. Then for a.e. $\overline{x}\in G_{0}$, the
function $f(\overline{x})$ belongs to $\mathcal{A}(\mathbb{R}%
^{d_{1}};L^{p}(G_{2}))$, so that, defining $g:G_{0}\times \mathbb{R}%
^{d_{1}}\rightarrow \mathbb{R}$ by $g(\overline{x},y)=\left\Vert f(\overline{%
x},y,\cdot )\right\Vert _{L^{p}(G_{2})}^{p}$, we have $g\in L^{1}(G_{0};%
\mathcal{A})$; this stems from the definition of $\mathcal{A}(\mathbb{R}%
^{d_{1}};L^{p}(G_{2}))$. It follows from the mean value property that 
\begin{equation*}
\int_{G_{0}}g\left( \overline{x},\frac{\overline{x}}{\varepsilon }\right) d%
\overline{x}\rightarrow \int_{G_{0}}M(g(\overline{x},\cdot ))d\overline{x},
\end{equation*}%
and 
\begin{eqnarray}
\int_{G_{0}}M(g(\overline{x},\cdot ))d\overline{x} &=&\int_{G_{0}}M\left(
\int_{G_{2}}\left\vert f(\overline{x},\cdot ,\zeta )\right\vert ^{p}d\zeta
\right) d\overline{x}  \label{2.6} \\
&=&\int_{G_{0}}\int_{G_{2}}M(\left\vert f(\overline{x},\cdot ,\zeta
)\right\vert ^{p})d\zeta d\overline{x},  \notag
\end{eqnarray}%
where for the last equality above we have used the continuity of the mean
value operator.

On the other hand, it holds that 
\begin{eqnarray*}
\varepsilon ^{-d_{2}}\int_{G_{\varepsilon }}\left\vert f\left( \overline{x},%
\frac{x}{\varepsilon }\right) \right\vert ^{p}dx
&=&\int_{G_{0}}\int_{G_{2}}\left\vert f\left( \overline{x},\frac{\overline{x}%
}{\varepsilon },\zeta \right) \right\vert ^{p}d\overline{x}d\zeta \\
&=&\int_{G_{0}}\left\Vert f\left( \overline{x},\frac{\overline{x}}{%
\varepsilon },\cdot \right) \right\Vert _{L^{p}(G_{2})}^{p}d\overline{x} \\
&=&\int_{G_{0}}g\left( \overline{x},\frac{\overline{x}}{\varepsilon }\right)
d\overline{x}.
\end{eqnarray*}%
Property (\ref{2.5}) therefore follows readily from (\ref{2.6}) and the last
series of equalities above.
\end{proof}

\begin{lemma}[{\protect\cite[Proposition 3.2]{CMP}}]
\label{l2.2}Let $X$ be a subspace (not necessarily closed) of a reflexive
Banach space $Y$ and let $f_{n}:X\rightarrow \mathbb{R}$ be a sequence of
linear functionals (not necessarily continuous). Assume there exists a
constant $C>0$ such that 
\begin{equation*}
\underset{n}{\lim \sup }\left\vert f_{n}(x)\right\vert \leq C\left\Vert
x\right\Vert _{Y}\text{ for all }x\in X.
\end{equation*}%
Then there exist a subsequence $(f_{n_{k}})_{k}$ of $(f_{n})$ and a
functional $f\in Y^{\prime }$ such that $\lim_{k\rightarrow \infty
}f_{n_{k}}(x)=f(x)$ for all $x\in X$.
\end{lemma}

We are now able to prove Theorem \ref{t2.1}.

\begin{proof}[Proof of Theorem \protect\ref{t2.1}]
In Lemma \ref{l2.2} we set $Y=L^{p^{\prime }}(G_{0};\mathcal{B}_{\mathcal{A}%
}^{p^{\prime }}(\mathbb{R}^{d_{1}};L^{p^{\prime }}(G_{2})))$,\ $%
X=L^{p^{\prime }}(G_{0};\mathcal{A}(\mathbb{R}^{d_{1}};L^{p^{\prime
}}(G_{2})))$ and define the mapping 
\begin{equation*}
L_{\varepsilon }(f)=\varepsilon ^{-d_{2}}\int_{G_{0}}u_{\varepsilon
}(x)f\left( \overline{x},\frac{x}{\varepsilon }\right) dx,\ \ f\in X.
\end{equation*}%
Then 
\begin{equation*}
\underset{\varepsilon \rightarrow 0}{\lim \sup }\left\vert L_{\varepsilon
}(f)\right\vert \leq C\left\Vert f\right\Vert _{Y}\text{ for all }f\in X.
\end{equation*}%
Indeed one has the inequality (arising from H\"{o}lder's inequality) 
\begin{equation*}
\left\vert L_{\varepsilon }(f)\right\vert \leq \varepsilon
^{-d_{2}/p}\left\Vert u_{\varepsilon }\right\Vert _{L^{p}(G_{\varepsilon
})}\left( \varepsilon ^{-d_{2}}\int_{G_{\varepsilon }}\left\vert f\left( 
\overline{x},\frac{x}{\varepsilon }\right) \right\vert ^{p^{\prime
}}dx\right) ^{\frac{1}{p^{\prime }}}.
\end{equation*}%
Thus, letting $\varepsilon \rightarrow 0$ we get with the help of Lemma \ref%
{l2.1}, 
\begin{eqnarray*}
\varepsilon ^{-d_{2}}\int_{G_{\varepsilon }}\left\vert f\left( \overline{x},%
\frac{x}{\varepsilon }\right) \right\vert ^{p^{\prime }}dx &\rightarrow
&\int_{G_{0}}\int_{G_{2}}M(\left\vert f(\overline{x},\cdot ,\zeta
)\right\vert ^{p^{\prime }})d\zeta d\overline{x} \\
&=&\left\Vert f\right\Vert _{Y}^{p^{\prime }}.
\end{eqnarray*}%
We therefore apply Lemma \ref{l2.2} with the above notation to derive the
existence of a subsequence $E^{\prime }$ of $E$ and of a unique $u_{0}\in
Y^{\prime }=L^{p}(G_{0};\mathcal{B}_{\mathcal{A}}^{p}(\mathbb{R}%
^{d_{1}};L^{p}(G_{2})))$ such that 
\begin{equation*}
L_{\varepsilon }(f)\rightarrow \int_{G_{0}}\int_{G_{2}}M(u_{0}(\overline{x}%
,\cdot ,\zeta )f(\overline{x},\cdot ,\zeta ))d\zeta d\overline{x}\text{ for
all }f\in X.
\end{equation*}
\end{proof}

\begin{remark}
\label{r2.2}\emph{In the proof Theorem \ref{t2.1}, there is no separability
assumption on the algebra with mean value }$\mathcal{A}$\emph{. Thus it
applies either for }$\mathcal{A}=\mathcal{C}_{per}(Y)$\emph{\ }$%
(Y=(0,1)^{d_{1}})$\emph{\ or for }$\mathcal{A}=AP(\mathbb{R}^{d_{1}})$\emph{%
\ (which is not separable). Our result generalizes the one in \cite{RJ2007}
(see for instance Proposition 4.2 in \cite{RJ2007} that corresponds to the
special case }$\mathcal{A}=\mathcal{C}_{per}(Y)$\emph{\ of our results here
for }$d_{1}=d-1$\emph{\ and }$G_{2}=(-1,1)$\emph{).}
\end{remark}

Before we state the next compactness result, we need however some
preliminary results. For a function $\mathbf{g}%
=(g_{1},...,g_{d_{1}},g_{d_{1}+1},...,g_{d})\in \lbrack \mathcal{B}_{%
\mathcal{A}}^{p}(\mathbb{R}^{d_{1}};L^{p}(G_{2}))]^{d}$ we define the
divergence $\overline{\func{div}}_{\overline{y},\zeta }\mathbf{g}$ by 
\begin{equation*}
\overline{\func{div}}_{\overline{y},\zeta }\mathbf{g}:=\sum_{i=1}^{d_{1}}%
\frac{\overline{\partial }g_{i}}{\partial y_{i}}+\sum_{i=1}^{d_{2}}\frac{%
\partial g_{d_{1}+i}}{\partial \zeta _{i}},
\end{equation*}%
that is, for any $\Phi =(\phi _{i})_{1\leq i\leq d}\in \lbrack \mathcal{B}_{%
\mathcal{A}}^{1,p^{\prime }}(\mathbb{R}^{d_{1}};W^{1,p^{\prime
}}(G_{2}))]^{d}$, 
\begin{equation*}
\left\langle \overline{\func{div}}_{\overline{y},\zeta }\mathbf{g,\Phi }%
\right\rangle =-\sum_{i=1}^{d_{1}}\int_{G_{2}}M\left( g_{i}(\cdot ,\zeta )%
\frac{\overline{\partial }\phi _{i}}{\partial y_{i}}(\cdot ,\zeta )\right)
d\zeta -\sum_{i=1}^{d_{2}}\int_{G_{2}}M\left( g_{d_{1}+i}(\cdot ,\zeta )%
\frac{\partial \phi _{d_{1}+i}}{\partial y_{d_{1}+i}}(\cdot ,\zeta )\right)
d\zeta .
\end{equation*}%
Any function belonging to an algebra with mean value on $\mathbb{R}^{m}$
will be considered as defined on the numerical space $\mathbb{R}^{m}$ of
generic variables $\overline{y}=(y_{1},...,y_{m})$.

This being so, the first preliminary result is the following one.

\begin{proposition}
\label{p2.1}Let $1<p<\infty $ and let $\mathcal{A}$ be an ergodic algebra
with mean value on $\mathbb{R}^{d_{1}}$. Finally let $L$ be a bounded linear
functional on $[\mathcal{B}_{\mathcal{A}}^{1,p^{\prime }}(\mathbb{R}%
^{d_{1}};W^{1,p^{\prime }}(G_{2}))]^{d}$ $(\frac{1}{p^{\prime }}=1-\frac{1}{p%
})$ that vanishes on the kernel of the divergence, that is 
\begin{equation*}
L(\Psi )=0\text{ for all }\Psi \in \mathcal{V}_{\func{div}}=\{\Phi \in
\lbrack \mathcal{D}_{\mathcal{A}}(\mathbb{R}^{d_{1}};\mathcal{C}_{0}^{\infty
}(G_{2}))]^{d}:\overline{\func{div}}_{\overline{y},\zeta }\Phi =0\}.
\end{equation*}%
Then there exists a function $f\in \mathcal{B}_{\mathcal{A}}^{p}(\mathbb{R}%
^{d_{1}};L^{p}(G_{2}))$ such that $L=\overline{\nabla }_{\overline{y},\zeta
}f$ (where $\overline{\nabla }_{\overline{y},\zeta }=(\overline{\nabla }_{%
\overline{y}},\nabla _{\zeta })$), i.e., 
\begin{equation*}
L(\mathbf{g})=-\int_{G_{2}}M(f(\cdot ,\zeta )\overline{\func{div}}_{%
\overline{y},\zeta }\mathbf{g}(\cdot ,\zeta ))d\zeta \text{, all }\mathbf{g}%
\in \lbrack \mathcal{B}_{\mathcal{A}}^{1,p^{\prime }}(\mathbb{R}%
^{d_{1}};W^{1,p^{\prime }}(G_{2}))]^{d}.
\end{equation*}%
Moreover $f$ is unique up to an additive constant, provided that $G_{2}$ is
connected.
\end{proposition}

Since the proof of the above proposition is similar to that of \cite[Theorem
2.1]{EJDE2014}, we will only sketch the proof. Before we can do that, let us
first give some preliminaries.

For $u\in \mathcal{A}^{\infty }\otimes \mathcal{C}_{0}^{\infty }(G_{2})$ and 
$\varphi \in \mathcal{C}_{0}^{\infty }(\mathbb{R}^{d_{1}})\otimes \mathcal{C}%
_{0}^{\infty }(G_{2})$ one may easily show that $u\ast \varphi \in \mathcal{A%
}^{\infty }\otimes \mathcal{C}_{0}^{\infty }(G_{2})$ (see e.g. \cite{ACAP}).
Using the density of $\mathcal{A}^{\infty }\otimes \mathcal{C}_{0}^{\infty
}(G_{2})$ (resp. $\mathcal{C}_{0}^{\infty }(\mathbb{R}^{d_{1}})\otimes 
\mathcal{C}_{0}^{\infty }(G_{2})$) in $\mathcal{A}^{\infty }(\mathbb{R}%
^{d_{1}};\mathcal{C}_{0}^{\infty }(G_{2}))$ (resp. $\mathcal{C}_{0}^{\infty
}(\mathbb{R}^{d_{1}}\times G_{2})$), we show that $u\ast \varphi \in 
\mathcal{A}^{\infty }(\mathbb{R}^{d_{1}};\mathcal{C}_{0}^{\infty }(G_{2}))$
for $u\in \mathcal{A}^{\infty }(\mathbb{R}^{d_{1}};\mathcal{C}_{0}^{\infty
}(G_{2}))$ and $\varphi \in \mathcal{C}_{0}^{\infty }(\mathbb{R}%
^{d_{1}}\times G_{2})$. Here $\ast $ stands for the usual convolution
operator. Using also a density argument, we may define $u\ast \varphi $ for $%
u\in B_{\mathcal{A}}^{p}(\mathbb{R}^{d_{1}};L^{p}(G_{2}))$ ($1\leq p<\infty $%
) and $\varphi \in \mathcal{C}_{0}^{\infty }(\mathbb{R}^{d_{1}}\times G_{2})$%
, and we have $u\ast \varphi \in B_{\mathcal{A}}^{p}(\mathbb{R}%
^{d_{1}};L^{p}(G_{2}))$. From the equality $D_{\overline{y},\zeta }^{\alpha
}(u\ast \varphi )=u\ast D_{\overline{y},\zeta }^{\alpha }\varphi $, we
deduce that in fact $u\ast \varphi \in \mathcal{A}^{\infty }(\mathbb{R}%
^{d_{1}};\mathcal{C}^{\infty }(\overline{G}_{2}))$. Moreover it holds that 
\begin{equation}
\left\Vert u\ast \varphi \right\Vert _{p}\leq \left\vert \text{\textrm{supp}}%
\varphi \right\vert ^{1/p}\left\Vert \varphi \right\Vert _{L^{p^{\prime }}(%
\mathbb{R}^{d_{1}}\times G_{2})}\left\Vert u\right\Vert _{p},  \label{*3}
\end{equation}%
where \textrm{supp}$\varphi $ denotes the support of $\varphi $ and $%
\left\vert \text{\textrm{supp}}\varphi \right\vert $ its Lebesgue measure.
To see this, we have 
\begin{eqnarray*}
\left\Vert u\ast \varphi \right\Vert _{p}^{p} &=&\int_{G_{2}}M(\left\vert
u\ast \varphi \right\vert ^{p})d\zeta =\int_{G_{2}}\left( \lim_{R\rightarrow
\infty }%
\mathchoice {{\setbox0=\hbox{$\displaystyle{\textstyle -}{\int}$ } \vcenter{\hbox{$\textstyle -$
}}\kern-.6\wd0}}{{\setbox0=\hbox{$\textstyle{\scriptstyle -}{\int}$ } \vcenter{\hbox{$\scriptstyle -$
}}\kern-.6\wd0}}{{\setbox0=\hbox{$\scriptstyle{\scriptscriptstyle -}{\int}$
} \vcenter{\hbox{$\scriptscriptstyle -$
}}\kern-.6\wd0}}{{\setbox0=\hbox{$\scriptscriptstyle{\scriptscriptstyle
-}{\int}$ } \vcenter{\hbox{$\scriptscriptstyle -$ }}\kern-.6\wd0}}%
\!\int_{B_{R}}\left\vert (u\ast \varphi )(\overline{y},\zeta )\right\vert
^{p}d\overline{y}\right) d\zeta \\
&=&\lim_{R\rightarrow \infty }\int_{G_{2}}%
\mathchoice {{\setbox0=\hbox{$\displaystyle{\textstyle -}{\int}$ } \vcenter{\hbox{$\textstyle -$
}}\kern-.6\wd0}}{{\setbox0=\hbox{$\textstyle{\scriptstyle -}{\int}$ } \vcenter{\hbox{$\scriptstyle -$
}}\kern-.6\wd0}}{{\setbox0=\hbox{$\scriptstyle{\scriptscriptstyle -}{\int}$
} \vcenter{\hbox{$\scriptscriptstyle -$
}}\kern-.6\wd0}}{{\setbox0=\hbox{$\scriptscriptstyle{\scriptscriptstyle
-}{\int}$ } \vcenter{\hbox{$\scriptscriptstyle -$ }}\kern-.6\wd0}}%
\!\int_{B_{R}}\left\vert (u\ast \varphi )(\overline{y},\zeta )\right\vert
^{p}d\overline{y}d\zeta .
\end{eqnarray*}%
But 
\begin{eqnarray*}
\int_{G_{2}}\int_{B_{R}}\left\vert (u\ast \varphi )(\overline{y},\zeta
)\right\vert ^{p}d\overline{y}d\zeta &\leq &\left(
\int_{G_{2}}\int_{B_{R}}\left\vert \varphi \right\vert d\overline{y}d\zeta
\right) ^{p}\iint_{B_{R}\times G_{2}}\left\vert u\right\vert ^{p}d\overline{y%
}d\zeta \\
&\leq &\left\vert (B_{R}\times G_{2})\cap \text{\textrm{supp}}\varphi
\right\vert \left\Vert \varphi \right\Vert _{L^{p^{\prime }}(B_{R}\times
G_{2})}^{p}\iint_{B_{R}\times G_{2}}\left\vert u\right\vert ^{p}d\overline{y}%
d\zeta ,
\end{eqnarray*}%
from which (\ref{*3}).

From the obvious inequality $\left\Vert \varphi \right\Vert _{L^{p^{\prime
}}(B_{R}\times G_{2})}\leq \left\vert \text{\textrm{supp}}\varphi
\right\vert ^{1/p^{\prime }}\left\Vert \varphi \right\Vert _{\infty }$ we
infer from (\ref{*3}) that 
\begin{equation}
\left\Vert u\ast \varphi \right\Vert _{p}\leq \left\Vert u\right\Vert
_{p}\left\Vert \varphi \right\Vert _{\infty }\ \ \forall u\in B_{\mathcal{A}%
}^{p}(\mathbb{R}^{d_{1}};L^{p}(G_{2})),\varphi \in \mathcal{C}_{0}^{\infty }(%
\mathbb{R}^{d_{1}}\times G_{2}).  \label{*4}
\end{equation}

This being so, let us sketch the proof of Proposition \ref{p2.1}.

\begin{proof}[Proof of Proposition \emph{\protect\ref{p2.1}}]
Let $u\in \mathcal{A}^{\infty }(\mathbb{R}^{d_{1}};\mathcal{C}_{0}^{\infty
}(G_{2}))$ be freely fixed, and define $L_{u}:\mathcal{D}(\mathbb{R}%
^{d_{1}}\times G_{2})^{d}\rightarrow \mathbb{R}$ by $L_{u}(\varphi
)=L(\varrho (u\ast \varphi ))$ for $\varphi =(\varphi _{i})_{1\leq i\leq d}$%
. Then since $L\in ([\mathcal{B}_{A}^{1,p^{\prime }}(\mathbb{R}%
^{d_{1}};W^{1,p^{\prime }}(G_{2}))]^{d})^{^{\prime }}$, one has 
\begin{equation*}
\left\vert L_{u}(\varphi )\right\vert \leq \left\Vert L\right\Vert
\left\Vert u\ast \varphi \right\Vert _{1,p^{\prime }}\leq \left\Vert
L\right\Vert \left\Vert u\right\Vert _{p}\left\Vert \varphi \right\Vert
_{\infty },
\end{equation*}%
where the last inequality above stems from both (\ref{*4}) and the equality $%
\nabla _{\overline{y},\zeta }(u\ast \varphi )=(\nabla _{\overline{y},\zeta
}u)\ast \varphi $. So $L_{u}$ defines a distribution on $\mathcal{D}(\mathbb{%
R}^{d_{1}}\times G_{2})^{d}$. In addition, if $\func{div}_{\overline{y}%
,\zeta }\varphi =0$, then $L_{u}(\varphi )=0$, i.e. $L_{u}$ vanishes on the
kernel of the divergence in $\mathcal{D}(\mathbb{R}^{d_{1}}\times G_{2})^{d}$%
. Appealing to the usual De Rahm theorem, we get the existence of a
distribution $S(u)\in \mathcal{D}^{\prime }(\mathbb{R}^{d_{1}}\times G_{2})$
such that $L_{u}=\nabla _{\overline{y},\zeta }S(u)$, thereby defining an
operator 
\begin{equation*}
S:\mathcal{A}^{\infty }(\mathbb{R}^{d_{1}};\mathcal{C}_{0}^{\infty
}(G_{2}))\rightarrow \mathcal{D}^{\prime }(\mathbb{R}^{d_{1}}\times G_{2});\
u\mapsto S(u).
\end{equation*}%
The operator $S$ enjoys the properties:

\begin{itemize}
\item[(i)] $S(u(\cdot +y))=S(u)(\cdot -y)$ \ $\forall y=(\overline{y},\zeta
)\in \mathbb{R}^{d},\ \forall u\in \mathcal{A}^{\infty }(\mathbb{R}^{d_{1}};%
\mathcal{C}_{0}^{\infty }(G_{2}))$;

\item[(ii)] $S$ maps continuously and linearly $\mathcal{A}^{\infty }(%
\mathbb{R}^{d_{1}};\mathcal{C}_{0}^{\infty }(G_{2}))$ into $%
L_{loc}^{p^{\prime }}(\mathbb{R}^{d_{1}}\times G_{2})$;

\item[(iii)] It holds that 
\begin{equation*}
\left\Vert S(u)\right\Vert _{L^{p^{\prime }}(B_{R}\times G_{2})}\leq
C_{R}\left\Vert L\right\Vert \left\vert B_{R}\times G_{2}\right\vert
^{1/p^{\prime }}\left\Vert u\right\Vert _{p^{\prime }}
\end{equation*}%
where $C_{R}>0$ is a locally bounded function of $R>0$.
\end{itemize}

Properties (i), (ii) and (iii) are easily obtained by following the same
line of reasoning as for their homologues in \cite{EJDE2014}. Let us just
point out that, for $\varphi \in \mathcal{D}(\mathbb{R}^{d_{1}}\times
G_{2})^{d}$ with \textrm{supp}$\varphi _{i}\subset B_{R}\times G_{2}$ for
all $1\leq i\leq d$, one has 
\begin{equation*}
\left\vert L_{u}(\varphi )\right\vert \leq \max_{1\leq i\leq d}\left\vert 
\mathrm{supp}\varphi _{i}\right\vert ^{1/p^{\prime }}\left\Vert L\right\Vert
\left\Vert u\right\Vert _{p^{\prime }}\left\Vert \varphi \right\Vert
_{W^{1,p}(B_{R}\times G_{2})},
\end{equation*}%
so that, because of the fact that \textrm{supp}$\varphi _{i}\subset
B_{R}\times G_{2}$ ($1\leq i\leq d$), 
\begin{equation*}
\left\Vert L_{u}\right\Vert _{W^{-1,p^{\prime }}(B_{R}\times G_{2})}\leq
\left\Vert L\right\Vert \left\vert B_{R}\times G_{2}\right\vert
^{1/p^{\prime }}\left\Vert u\right\Vert _{p^{\prime }}.
\end{equation*}%
It therefore follows that there exists $C=C(p,R)>0$ such that 
\begin{equation*}
\left\Vert S(u)\right\Vert _{L^{p^{\prime }}(B_{R}\times G_{2})}\leq
C\left\Vert L\right\Vert \left\vert B_{R}\times G_{2}\right\vert
^{1/p^{\prime }}\left\Vert u\right\Vert _{p^{\prime }};
\end{equation*}%
for the last inequality above, see the proof of \cite[Theorem 2.6]{EJDE2014}%
. Hence we obtain, as in \cite[Theorem 2.6]{EJDE2014}, that $S(u)\in 
\mathcal{C}^{\infty }(\mathbb{R}^{d_{1}}\times G_{2})$ with $D_{y}^{\alpha
}S(u)=S(D_{y}^{\alpha }u)$ for all $\alpha \in \mathbb{N}^{d}$, so that 
\begin{equation*}
\left\vert S(u)(0)\right\vert \leq C\left\Vert L\right\Vert \left\Vert
u\right\Vert _{p^{\prime }}\ \ \text{for all }u\in \mathcal{A}^{\infty }(%
\mathbb{R}^{d_{1}};\mathcal{C}_{0}^{\infty }(G_{2})).
\end{equation*}%
So we define $\widetilde{S}:\mathcal{D}_{\mathcal{A}}(\mathbb{R}^{d_{1}};%
\mathcal{C}_{0}^{\infty }(G_{2}))\rightarrow \mathbb{R}$ by $\widetilde{S}%
(\varrho (u))=S(u)(0)$ for $u\in \mathcal{A}^{\infty }(\mathbb{R}^{d_{1}};%
\mathcal{C}_{0}^{\infty }(G_{2}))$. Then $\widetilde{S}$ is linear and
satisfies 
\begin{equation}
\left\vert \widetilde{S}(\varrho (u))\right\vert \leq C\left\Vert
L\right\Vert \left\Vert u\right\Vert _{p^{\prime }}\ \ \text{for all }u\in 
\mathcal{A}^{\infty }(\mathbb{R}^{d_{1}};\mathcal{C}_{0}^{\infty }(G_{2})).
\label{*5}
\end{equation}%
We derive from (\ref{*5}) together with the density of $\mathcal{A}^{\infty
}(\mathbb{R}^{d_{1}};\mathcal{C}_{0}^{\infty }(G_{2}))$ in $B_{\mathcal{A}%
}^{p^{\prime }}(\mathbb{R}^{d_{1}};L^{p^{\prime }}(G_{2}))$ that there
exists $f\in \mathcal{B}_{\mathcal{A}}^{p}(\mathbb{R}^{d_{1}};L^{p}(G_{2}))$
such that 
\begin{equation*}
\widetilde{S}(v)=\int_{G_{2}}M(fv)d\zeta \ \ \forall v\in \mathcal{B}_{%
\mathcal{A}}^{p^{\prime }}(\mathbb{R}^{d_{1}};L^{p^{\prime }}(G_{2}))
\end{equation*}%
and 
\begin{equation*}
\left\Vert f\right\Vert _{p}\leq C\left\Vert L\right\Vert .\ \ \ \ \ \ \ \ \
\ \ \ \ \ \ \ \ \ \ \ \ \ \ \ \ \ \ \ \ \ \ \ \ \ \ \ \ \ \ \ \ \ 
\end{equation*}%
As in \cite{EJDE2014}, we obtain that $L=\overline{\nabla }_{\overline{y}%
,\zeta }f$, and since $\mathcal{A}$ is ergodic and $G_{2}$ is connected, $f$
is unique up to addition of a constant.
\end{proof}

The next corollary is of interest in the forthcoming compactness result.

\begin{corollary}
\label{c2.1}Let $1<p<\infty $ and let $\mathbf{f}\in \lbrack \mathcal{B}_{%
\mathcal{A}}^{p}(\mathbb{R}^{d_{1}};L^{p}(G_{2}))]^{d}$ be such that 
\begin{equation*}
\int_{G_{2}}M(\mathbf{f}(\cdot ,\zeta )\cdot \mathbf{g}(\cdot ,\zeta
))d\zeta =0\text{ for all }\mathbf{g}\in \mathcal{V}_{\func{div}},
\end{equation*}%
where $\mathcal{V}_{\func{div}}$ is defined as in Proposition \emph{\ref%
{p2.1}} and where $G_{2}$ is a connected open subset of $\mathbb{R}^{d_{2}}$%
. Then there exists a function $u\in B_{\#\mathcal{A}}^{1,p}(\mathbb{R}%
^{d_{1}};W^{1,p}(G_{2}))$, uniquely determined modulo constants, such that $%
\mathbf{f}=\nabla _{\overline{y},\zeta }u$, where $\nabla _{\overline{y}%
,\zeta }=(\nabla _{\overline{y}},\nabla _{\zeta })$.
\end{corollary}

\begin{proof}
Let us first recall that $B_{\#\mathcal{A}}^{1,p}(\mathbb{R}%
^{d_{1}};W^{1,p}(G_{2}))$ is the space of functions $u\in W_{loc}^{1,p}(%
\mathbb{R}^{d_{1}};W^{1,p}(G_{2}))$ satisfying $\nabla _{\overline{y},\zeta
}u\in (B_{\mathcal{A}}^{p}(\mathbb{R}^{d_{1}};L^{p}(G_{2})))^{d}$ and $%
\int_{G_{2}}M(\nabla _{\overline{y},\zeta }u(\cdot ,\zeta ))d\zeta =0$. This
being so, let $(\varphi _{n})_{n\geq 1}\subset \mathcal{C}_{0}^{\infty }(%
\mathbb{R}^{d_{1}}\times \mathbb{R}^{d_{2}})$ be a mollifier satisfying $%
\varphi _{n}(-\overline{y},-\zeta )=\varphi _{n}(\overline{y},\zeta )$ for
all $(\overline{y},\zeta )\in \mathbb{R}^{d}$. We extend $\mathbf{f}$ by $0$
outside $\mathbb{R}^{d_{1}}\times G_{2}$ and we still denote by $\mathbf{f}$
its extension on $\mathbb{R}^{d}$. We define the convolution product $%
\mathbf{f}_{n}:=\mathbf{f}\circledast \varphi _{n}\equiv (f_{i}\circledast
\varphi _{n})_{1\leq i\leq d}$ as follows: let $\mathbf{f}_{0}$ be a
representative of $\mathbf{f}$, that is $\mathbf{f}=\mathbf{f}_{0}+\mathcal{N%
}$ where $\mathbf{f}_{0}\in (B_{\mathcal{A}}^{p}(\mathbb{R}%
^{d_{1}};L^{p}(G_{2})))^{d}$; we know that $\mathbf{f}_{0}\ast \varphi _{n}$
is well defined as an element of $(B_{\mathcal{A}}^{p}(\mathbb{R}%
^{d_{1}};L^{p}(G_{2})))^{d}$ (see e.g. \cite[Page 9]{EJDE2014}). We
therefore set 
\begin{equation*}
\mathbf{f}\circledast \varphi _{n}:=\mathbf{f}_{0}\ast \varphi _{n}+\mathcal{%
N}=\varrho (\mathbf{f}_{0}\ast \varphi _{n}).
\end{equation*}%
As in \cite[Page 9]{EJDE2014} we can easily show that $\mathbf{f}_{n}\equiv 
\mathbf{f}\circledast \varphi _{n}\in \lbrack \mathcal{D}_{\mathcal{A}}(%
\mathbb{R}^{d_{1}};\mathcal{C}^{\infty }(\overline{G}_{2}))]^{d}$ with $%
\overline{D}_{\overline{y},\zeta }^{\alpha }\mathbf{f}_{n}=\varrho (\mathbf{f%
}_{0}\ast D_{\overline{y},\zeta }^{\alpha }\varphi _{n})$ for all $\alpha
\in \mathbb{N}^{d}$. Moreover, from the convergence result 
\begin{equation*}
\mathbf{f}_{0}\ast \varphi _{n}\rightarrow \mathbf{f}_{0}\text{ in }(B_{%
\mathcal{A}}^{p}(\mathbb{R}^{d_{1}};L^{p}(G_{2})))^{d}\text{ as }%
n\rightarrow \infty ,
\end{equation*}%
we infer 
\begin{equation}
\mathbf{f}_{n}\rightarrow \mathbf{f}\text{ in }(\mathcal{B}_{\mathcal{A}%
}^{p}(\mathbb{R}^{d_{1}};L^{p}(G_{2})))^{d}\text{ when }n\rightarrow \infty .
\label{2.6'}
\end{equation}%
It further holds (using the equality $\varphi _{n}(-\overline{y},-\zeta
)=\varphi _{n}(\overline{y},\zeta )$) that 
\begin{equation*}
\int_{G_{2}}M(\mathbf{f}_{n}(\cdot ,\zeta )\cdot \mathbf{g}(\cdot ,\zeta
))d\zeta =\int_{G_{2}}M(\mathbf{f}(\cdot ,\zeta )\cdot (\mathbf{g}%
\circledast \varphi _{n})(\cdot ,\zeta ))d\zeta
\end{equation*}%
for any $\mathbf{g}\in \lbrack \mathcal{D}_{\mathcal{A}}(\mathbb{R}^{d_{1}};%
\mathcal{C}_{0}^{\infty }(G_{2}))]^{d}$, so that, if $\overline{\func{div}}_{%
\overline{y},\zeta }\mathbf{g}=0$, then $\int_{G_{2}}M(\mathbf{f}_{n}(\cdot
,\zeta )\cdot \mathbf{g}(\cdot ,\zeta ))d\zeta =0$.

Now we define the mapping 
\begin{equation*}
\mathbf{g}\mapsto \int_{G_{2}}M(\mathbf{f}_{n}(\cdot ,\zeta )\cdot \mathbf{g}%
(\cdot ,\zeta ))d\zeta ,
\end{equation*}%
which is easily seen to belong to $([\mathcal{B}_{\mathcal{A}}^{1,p^{\prime
}}(\mathbb{R}^{d_{1}};W^{1,p^{\prime }}(G_{2}))]^{d})^{\prime }$. We deduce
from Proposition \ref{p2.1} the existence of $u_{n}\in \mathcal{B}_{\mathcal{%
A}}^{p}(\mathbb{R}^{d_{1}};L^{p}(G_{2}))$ such that 
\begin{equation}
\mathbf{f}_{n}=\overline{\nabla }_{\overline{y},\zeta }u_{n}.\ \ \ \ \ \ \ \
\ \ \ \ \ \ \ \ \ \ \   \label{2.6''}
\end{equation}%
Since $\mathbf{f}_{n}\in \lbrack \mathcal{D}_{\mathcal{A}}(\mathbb{R}%
^{d_{1}};\mathcal{C}^{\infty }(\overline{G}_{2}))]^{d}$, it follows that $%
u_{n}\in \mathcal{D}_{\mathcal{A}}(\mathbb{R}^{d_{1}};\mathcal{C}^{\infty }(%
\overline{G}_{2}))$. Hence, identifying $u_{n}$ with any of its
representative in $\mathcal{A}^{\infty }(\mathbb{R}^{d_{1}};\mathcal{C}%
^{\infty }(\overline{G}_{2}))$ and using the uniqueness of its gradient we
get that $u_{n}\in B_{\#\mathcal{A}}^{1,p}(\mathbb{R}%
^{d_{1}};W^{1,p}(G_{2})) $. The sequence $(\mathbf{f}_{n})_{n}$ being
convergent in the norm topology of $(\mathcal{B}_{\mathcal{A}}^{p}(\mathbb{R}%
^{d_{1}};L^{p}(G_{2})))^{d}$, the sequence $(u_{n})_{n}$ is a Cauchy
sequence in $B_{\#\mathcal{A}}^{1,p}(\mathbb{R}^{d_{1}};W^{1,p}(G_{2}))$ for
if 
\begin{eqnarray*}
\left\Vert u_{n}-u_{m}\right\Vert _{\#,p}^{p} &=&\int_{G_{2}}\left\Vert
\nabla _{\overline{y},\zeta }u_{n}(\cdot ,\zeta )-\nabla _{\overline{y}%
,\zeta }u_{m}(\cdot ,\zeta )\right\Vert _{p}^{p}d\zeta \\
&=&\int_{G_{2}}\left\Vert \mathbf{f}_{n}(\cdot ,\zeta )-\mathbf{f}_{m}(\cdot
,\zeta )\right\Vert _{p}^{p}d\zeta \rightarrow 0\text{ when }n,m\rightarrow
\infty .
\end{eqnarray*}%
It follows that there exists $u\in B_{\#\mathcal{A}}^{1,p}(\mathbb{R}%
^{d_{1}};W^{1,p}(G_{2}))$ such that $u_{n}\rightarrow u$ in $B_{\#\mathcal{A}%
}^{1,p}(\mathbb{R}^{d_{1}};W^{1,p}(G_{2}))$. From (\ref{2.6''}) we get
readily $\mathbf{f}=\overline{\nabla }_{\overline{y},\zeta }u$.
\end{proof}

We are now able to state and prove the next compactness result dealing with
the convergence of the gradient.

\begin{theorem}
\label{t2.2}Assume that $\mathcal{A}$ is an ergodic algebra with mean value
on $\mathbb{R}^{d_{1}}$ and that $G_{2}$ is connected. Let $(u_{\varepsilon
})_{\varepsilon \in E}$ be a sequence in $W^{1,p}(G_{\varepsilon })$ ($%
1<p<\infty $) such that 
\begin{equation}
\sup_{\varepsilon \in E}\left( \varepsilon ^{-d_{2}/p}\left\Vert
u_{\varepsilon }\right\Vert _{L^{p}(G_{\varepsilon })}+\varepsilon
^{-d_{2}/p}\left\Vert \nabla u_{\varepsilon }\right\Vert
_{L^{p}(G_{\varepsilon })}\right) \leq C  \label{2.7}
\end{equation}%
where $C>0$ is independent of $\varepsilon $. Then there exist a subsequence 
$E^{\prime }$ of $E$ and a couple $(u_{0},u_{1})$ with $u_{0}\in
W^{1,p}(G_{0})$ and $u_{1}\in L^{p}(G_{0};B_{\#\mathcal{A}}^{1,p}(\mathbb{R}%
^{d_{1}};W^{1,p}(G_{2})))$ such that, as $E^{\prime }\ni \varepsilon
\rightarrow 0$, 
\begin{equation}
u_{\varepsilon }\rightarrow u_{0}\text{ in }L^{p}(G_{\varepsilon })\text{%
-weak }\Sigma _{\mathcal{A}},  \label{2.8}
\end{equation}%
\begin{equation}
\frac{\partial u_{\varepsilon }}{\partial x_{i}}\rightarrow \frac{\partial
u_{0}}{\partial \overline{x}_{i}}+\frac{\partial u_{1}}{\partial \overline{y}%
_{i}}\text{ in }L^{p}(G_{\varepsilon })\text{-weak }\Sigma _{\mathcal{A}},\
1\leq i\leq d_{1},  \label{2.9}
\end{equation}%
and 
\begin{equation}
\frac{\partial u_{\varepsilon }}{\partial x_{d_{1}+i}}\rightarrow \frac{%
\partial u_{1}}{\partial \zeta _{i}}\text{ in }L^{p}(G_{\varepsilon })\text{%
-weak }\Sigma _{\mathcal{A}},\ 1\leq i\leq d_{2}.  \label{2.10}
\end{equation}
\end{theorem}

\begin{remark}
\label{r2.3}\emph{If we set }$\zeta =(y_{d_{1}+1},...,y_{d})$\emph{\ (so
that }$(\overline{y},\zeta )=y$\emph{) and }%
\begin{equation*}
\nabla _{\overline{x}}u_{0}=\left( \frac{\partial u_{0}}{\partial \overline{x%
}_{1}},...,\frac{\partial u_{0}}{\partial \overline{x}_{d_{1}}}%
,0,...,0\right) ,
\end{equation*}%
\emph{then (\ref{2.9}) and (\ref{2.10}) are equivalent to }%
\begin{equation*}
\nabla u_{\varepsilon }\rightarrow \nabla _{\overline{x}}u_{0}+\nabla
_{y}u_{1}\text{\emph{\ in }}L^{p}(G_{\varepsilon })^{d}\text{\emph{-weak }}%
\Sigma _{\mathcal{A}}.
\end{equation*}
\end{remark}

\begin{proof}[Proof of Theorem \protect\ref{t2.2}]
In view of the assumption (\ref{2.7}), we appeal to Theorem \ref{t2.1} to
derive the existence of a subsequence $E^{\prime }$ of $E$ and $u_{0}\in
L^{p}(G_{0};\mathcal{B}_{\mathcal{A}}^{p}(\mathbb{R}^{d_{1}};L^{p}(G_{2})))$
and $\mathbf{v}=(v_{i})_{1\leq i\leq d}\in \lbrack L^{p}(G_{0};\mathcal{B}_{%
\mathcal{A}}^{p}(\mathbb{R}^{d_{1}};L^{p}(G_{2})))]^{d}$ such that 
\begin{equation}
u_{\varepsilon }\rightarrow u_{0}\text{ in }L^{p}(G_{\varepsilon })\text{%
-weak }\Sigma _{\mathcal{A}},  \label{2.11}
\end{equation}%
\begin{equation}
\frac{\partial u_{\varepsilon }}{\partial \overline{x}_{i}}\rightarrow v_{i}%
\text{ in }L^{p}(G_{\varepsilon })\text{-weak }\Sigma _{\mathcal{A}},\ 1\leq
i\leq d_{1},  \label{2.12}
\end{equation}%
and 
\begin{equation}
\frac{\partial u_{\varepsilon }}{\partial \zeta _{i}}\rightarrow v_{d_{1}+i}%
\text{ in }L^{p}(G_{\varepsilon })\text{-weak }\Sigma _{\mathcal{A}},\ 1\leq
i\leq d_{2},  \label{2.13}
\end{equation}%
where for $x=(x_{1},...,x_{d_{1}},x_{d_{1}+1},...,x_{d})$ we set $x=(%
\overline{x},\zeta )$ with $\overline{x}=(x_{i})_{1\leq i\leq d_{1}}$ and $%
\zeta =(x_{d_{1}+i})_{1\leq i\leq d_{2}}$ and thus, $\nabla =(\nabla _{%
\overline{x}},\nabla _{\zeta })$. Let us first show that $u_{0}$ does not
depend on $(\overline{y},\zeta )$. To that end, let $\Phi \in (\mathcal{C}%
_{0}^{\infty }(G_{0})\otimes \mathcal{A}^{\infty }\otimes \mathcal{C}%
_{0}^{\infty }(G_{2}))^{d_{1}}$. One has 
\begin{eqnarray*}
&&\varepsilon ^{-d_{2}}\int_{G_{\varepsilon }}\varepsilon \nabla _{\overline{%
x}}u_{\varepsilon }(x)\cdot \Phi \left( \overline{x},\frac{x}{\varepsilon }%
\right) dx \\
&=&-\int_{G_{\varepsilon }}\varepsilon ^{-d_{2}}u_{\varepsilon }(x)\left(
\varepsilon (\func{div}_{\overline{x}}\Phi )\left( \overline{x},\frac{x}{%
\varepsilon }\right) +(\func{div}_{\overline{y}}\Phi )\left( \overline{x},%
\frac{x}{\varepsilon }\right) \right) dx.
\end{eqnarray*}%
Letting $E^{\prime }\ni \varepsilon \rightarrow 0$ and using (\ref{2.11})-(%
\ref{2.12}), we get 
\begin{equation*}
\int_{G_{0}}\int_{G_{2}}M(u_{0}(\overline{x},\cdot ,\zeta )\func{div}_{%
\overline{y}}\Phi (\overline{x},\cdot ,\zeta ))d\zeta d\overline{x}=0.
\end{equation*}%
This shows that $\overline{\nabla }_{\overline{y}}u_{0}(\overline{x},\cdot
,\zeta )=0$ for a.e. $(\overline{x},\zeta )$, which amounts to $u_{0}(%
\overline{x},\cdot ,\zeta )$ is an invariant function. Since the algebra $%
\mathcal{A}$ is ergodic, $u_{0}(\overline{x},\cdot ,\zeta )$ does not depend
on $\overline{y}$, that is $u_{0}(\overline{x},\cdot ,\zeta )=u_{0}(%
\overline{x},\zeta )$.

Let us now show that $u_{0}$ is independent of $\zeta $. Let this time $\Phi
\in (\mathcal{C}_{0}^{\infty }(G_{0})\otimes \mathcal{C}_{0}^{\infty
}(G_{2}))^{d_{2}}$. It is easily seen that 
\begin{equation*}
\varepsilon ^{-d_{2}}\int_{G_{\varepsilon }}\varepsilon \nabla _{\zeta
}u_{\varepsilon }\cdot \Phi \left( \overline{x},\frac{\zeta }{\varepsilon }%
\right) dx=-\int_{G_{\varepsilon }}\varepsilon ^{-d_{2}}u_{\varepsilon }(%
\func{div}_{\zeta }\Phi )\left( \overline{x},\frac{\zeta }{\varepsilon }%
\right) dx.
\end{equation*}%
Letting once again $E^{\prime }\ni \varepsilon \rightarrow 0$ and using (\ref%
{2.11}) and (\ref{2.13}), we obtain 
\begin{equation*}
\int_{G_{0}}\int_{G_{2}}u_{0}(\overline{x},\zeta )\func{div}_{\zeta }\Phi (%
\overline{x},\zeta )d\zeta d\overline{x}=0,
\end{equation*}%
which shows that $u_{0}$ is independent of $\zeta $. Thus $u_{0}(\overline{x}%
,\zeta )=u_{0}(\overline{x})$.

Next let $\Phi _{\varepsilon }(x)=\varphi (\overline{x})\Psi (x/\varepsilon
) $ ($x\in G_{\varepsilon }$) with $\varphi \in \mathcal{C}_{0}^{\infty
}(G_{0})$ and $\Psi \in (\mathcal{A}^{\infty }(\mathbb{R}^{d_{1}};\mathcal{C}%
_{0}^{\infty }(G_{2}))^{d}$ with $\func{div}_{\overline{y},\zeta }\Psi =0$.
We set $\Psi =(\Psi _{\overline{x}},\Psi _{\zeta })$ with $\Psi _{\overline{x%
}}=(\psi _{j})_{1\leq j\leq d_{1}}$ and $\Psi _{\zeta }=(\psi
_{d_{1}+j})_{1\leq j\leq d_{2}}$. We clearly have 
\begin{eqnarray}
&&\int_{G_{\varepsilon }}\varepsilon ^{-d_{2}}\left( \nabla _{\overline{x}%
}u_{\varepsilon }(x)\cdot \Psi _{\overline{x}}\left( \frac{x}{\varepsilon }%
\right) +\nabla _{\zeta }u_{\varepsilon }(x)\cdot \Psi _{\zeta }\left( \frac{%
x}{\varepsilon }\right) \right) \varphi (\overline{x})dx  \label{2.14} \\
&=&-\int_{G_{\varepsilon }}\varepsilon ^{-d_{2}}u_{\varepsilon }(x)\Psi _{%
\overline{x}}\left( \frac{x}{\varepsilon }\right) \cdot \nabla _{\overline{x}%
}\varphi (\overline{x})dx.  \notag
\end{eqnarray}%
Letting $E^{\prime }\ni \varepsilon \rightarrow 0$ in (\ref{2.14}) yields 
\begin{eqnarray}
&&\int_{G_{0}}\int_{G_{2}}M(\mathbf{v}(\overline{x},\cdot ,\zeta )\cdot \Psi
(\cdot ,\zeta ))\varphi (\overline{x})d\overline{x}d\zeta  \label{2.15} \\
&=&-\int_{G_{0}}\int_{G_{2}}u_{0}(\overline{x})M(\Psi _{\overline{x}}(\cdot
,\zeta ))\cdot \nabla _{\overline{x}}\varphi (\overline{x})d\overline{x}%
d\zeta .  \notag
\end{eqnarray}%
First, taking in (\ref{2.15}) $\Psi =(\varphi \delta _{ij})_{1\leq i\leq d}$
(for each fixed $1\leq j\leq d$) with $\varphi \in \mathcal{C}_{0}^{\infty
}(G_{0})$ and where $\delta _{ij}$ are the Kronecker delta, we obtain 
\begin{equation}
\int_{G_{0}}\left( \int_{G_{2}}M(v_{j}(\overline{x},\cdot ,\zeta )d\zeta
\right) \varphi (\overline{x})d\overline{x}=-\left\vert G_{2}\right\vert
\int_{G_{0}}u_{0}\frac{\partial \varphi }{\partial \overline{x}_{j}}d%
\overline{x},  \label{2.16}
\end{equation}%
where $\mathbf{v}=(v_{j})_{1\leq j\leq d}$ and $\left\vert G_{2}\right\vert $
stands for the Lebesgue measure of $G_{2}$. Recalling that $v_{j}\in
L^{p}(G_{0};\mathcal{B}_{\mathcal{A}}^{p}(\mathbb{R}^{d_{1}};L^{p}(G_{2})))$%
, we infer that the function $\overline{x}\mapsto \int_{G_{2}}M(v_{j}(%
\overline{x},\cdot ,\zeta )d\zeta $ belongs to $L^{p}(G_{0})$, so that (\ref%
{2.16}) yields $\partial u_{0}/\partial \overline{x}_{j}\in L^{p}(G_{0})$
for $1\leq j\leq d_{1}$, where $\partial u_{0}/\partial \overline{x}_{j}$ is
the distributional derivative of $u_{0}$ with respect to $\overline{x}_{j}$.
We deduce that $u_{0}\in W^{1,p}(G_{0})$. Coming back to (\ref{2.15}) and
integrating its right-hand side with respect to $\overline{x}$, we have 
\begin{eqnarray*}
&&\int_{G_{0}}\int_{G_{2}}M(\mathbf{v}(\overline{x},\cdot ,\zeta )\cdot \Psi
(\cdot ,\zeta ))\varphi (\overline{x})d\overline{x}d\zeta \\
&=&\int_{G_{0}}\int_{G_{2}}\left( \nabla _{\overline{x}}u_{0}(\overline{x}%
)\cdot M(\Psi _{\overline{x}}(\cdot ,\zeta )\right) \varphi (\overline{x})d%
\overline{x}d\zeta \\
&=&\int_{G_{0}}\int_{G_{2}}\left( \nabla _{\overline{x}}u_{0}(\overline{x}%
)\cdot M(\Psi (\cdot ,\zeta )\right) \varphi (\overline{x})d\overline{x}%
d\zeta ,
\end{eqnarray*}%
where the last equality above arises from the equality $\nabla _{\overline{x}%
}u_{0}=\left( \frac{\partial u_{0}}{\partial \overline{x}_{1}},...,\frac{%
\partial u_{0}}{\partial \overline{x}_{d_{1}}},0,...,0\right) $. We obtain
readily 
\begin{equation}
\int_{G_{0}}\left( \int_{G_{2}}M\left( (\mathbf{v}(\overline{x},\cdot ,\zeta
)-\nabla _{\overline{x}}u_{0}(\overline{x}))\cdot \Psi (\cdot ,\zeta
)\right) d\zeta \right) \varphi (\overline{x})d\overline{x}=0.  \label{2.17}
\end{equation}%
From the arbitrariness of $\varphi $, (\ref{2.17}) entails 
\begin{equation*}
\int_{G_{2}}M\left( (\mathbf{v}(\overline{x},\cdot ,\zeta )-\nabla _{%
\overline{x}}u_{0}(\overline{x}))\cdot \Psi (\cdot ,\zeta )\right) d\zeta =0%
\text{ for a.e. }\overline{x}\in G_{0},
\end{equation*}%
and for all $\Psi \in (\mathcal{A}^{\infty }(\mathbb{R}^{d_{1}};\mathcal{C}%
_{0}^{\infty }(G_{2}))^{d}$ with $\func{div}_{\overline{y},\zeta }\Psi =0$.
We make use of Corollary \ref{c2.1} to deduce the existence of $u_{1}(%
\overline{x},\cdot ,\cdot )\in B_{\#\mathcal{A}}^{1,p}(\mathbb{R}%
^{d_{1}};W^{1,p}(G_{2}))$ such that 
\begin{equation*}
\mathbf{v}(\overline{x},\cdot ,\zeta )-\nabla _{\overline{x}}u_{0}(\overline{%
x})=\nabla _{\overline{y},\zeta }u_{1}(\overline{x},\cdot ,\cdot )\text{ for
a.e. }\overline{x}\in G_{0}.
\end{equation*}%
Hence the existence of a function $\overline{x}\mapsto u_{1}(\overline{x}%
,\cdot ,\cdot )$ from $G_{0}$ into $B_{\#\mathcal{A}}^{1,p}(\mathbb{R}%
^{d_{1}};W^{1,p}(G_{2}))$, which belongs to $L^{p}(G_{0};B_{\#\mathcal{A}%
}^{1,p}(\mathbb{R}^{d_{1}};W^{1,p}(G_{2})))$, such that $\mathbf{v}=\nabla _{%
\overline{x}}u_{0}+\nabla _{\overline{y},\zeta }u_{1}$.
\end{proof}

The following result provides us with sufficient conditions for which the
convergence result in (\ref{2.8}) is strong.

\begin{theorem}
\label{t2.3}The assumptions are those of Theorem \emph{\ref{t2.2}}. Assume
that $G_{1}$ is regular enough so that the embedding $W^{1,p}(G_{1})%
\hookrightarrow L^{p}(G_{1})$ is compact and further $G_{2}$ is convex. Let $%
(u_{0},u_{1})$ and $E^{\prime }$ be as in Theorem \emph{\ref{t2.2}}. Then,
as $E^{\prime }\ni \varepsilon \rightarrow 0$, the conclusions of Theorem 
\emph{\ref{t2.2}} hold and further 
\begin{equation}
u_{\varepsilon }\rightarrow u_{0}\text{ in }L^{p}(G_{\varepsilon })\text{%
-strong }\Sigma _{\mathcal{A}}.  \label{2.8'}
\end{equation}
\end{theorem}

\begin{proof}
Let us first define the average $M_{\varepsilon }$ in the thin directions as
follows: 
\begin{equation*}
(M_{\varepsilon }u_{\varepsilon })(\overline{x})=%
\mathchoice {{\setbox0=\hbox{$\displaystyle{\textstyle
-}{\int}$ } \vcenter{\hbox{$\textstyle -$
}}\kern-.6\wd0}}{{\setbox0=\hbox{$\textstyle{\scriptstyle -}{\int}$ } \vcenter{\hbox{$\scriptstyle -$
}}\kern-.6\wd0}}{{\setbox0=\hbox{$\scriptstyle{\scriptscriptstyle -}{\int}$
} \vcenter{\hbox{$\scriptscriptstyle -$
}}\kern-.6\wd0}}{{\setbox0=\hbox{$\scriptscriptstyle{\scriptscriptstyle
-}{\int}$ } \vcenter{\hbox{$\scriptscriptstyle -$ }}\kern-.6\wd0}}%
\!\int_{\varepsilon G_{2}}u_{\varepsilon }(\overline{x},\zeta )d\zeta \text{
for }\overline{x}\in G_{1}.
\end{equation*}%
Then the Lebesgue dominated convergence theorem yields $M_{\varepsilon
}\nabla _{\overline{x}}=\nabla _{\overline{x}}M_{\varepsilon }$. It follows
therefore that $M_{\varepsilon }u_{\varepsilon }\in W^{1,p}(G_{1})$ with 
\begin{equation}
\left\Vert M_{\varepsilon }u_{\varepsilon }\right\Vert _{W^{1,p}(G_{1})}\leq
C\ \ \ \ \ \ \ \ \ \ \ \ \ \ \ \ \ \ \ \ \ \ \ \ \ \   \label{*0}
\end{equation}%
where $C$ is a positive constant independent of $\varepsilon $, the last
inequality above being a consequence of (\ref{2.7}). Next the following
Poincar\'{e}-Wirtinger inequality holds: 
\begin{equation}
\varepsilon ^{-\frac{d_{2}}{p}}\left\Vert u_{\varepsilon }-M_{\varepsilon
}u_{\varepsilon }\right\Vert _{L^{p}(G_{\varepsilon })}\leq C\varepsilon
\left\Vert \nabla u_{\varepsilon }\right\Vert _{L^{p}(G_{\varepsilon })},
\label{*1}
\end{equation}%
where $C>0$ is independent of $\varepsilon $. Indeed, from the density of $%
\mathcal{C}^{1}(\overline{G_{\varepsilon }})$ in $W^{1,p}(G_{\varepsilon })$%
, we may assume, without lost of generality, that $u_{\varepsilon }$ is
smooth enough. In that case, one has, for $\xi \in \varepsilon G_{2}$, 
\begin{eqnarray*}
u_{\varepsilon }(\overline{x},\xi )-M_{\varepsilon }u_{\varepsilon }(%
\overline{x}) &=&%
\mathchoice {{\setbox0=\hbox{$\displaystyle{\textstyle
-}{\int}$ } \vcenter{\hbox{$\textstyle -$
}}\kern-.6\wd0}}{{\setbox0=\hbox{$\textstyle{\scriptstyle -}{\int}$ } \vcenter{\hbox{$\scriptstyle -$
}}\kern-.6\wd0}}{{\setbox0=\hbox{$\scriptstyle{\scriptscriptstyle -}{\int}$
} \vcenter{\hbox{$\scriptscriptstyle -$
}}\kern-.6\wd0}}{{\setbox0=\hbox{$\scriptscriptstyle{\scriptscriptstyle
-}{\int}$ } \vcenter{\hbox{$\scriptscriptstyle -$ }}\kern-.6\wd0}}%
\!\int_{\varepsilon G_{2}}(u_{\varepsilon }(\overline{x},\xi
)-u_{\varepsilon }(\overline{x},z))dz \\
&=&%
\mathchoice {{\setbox0=\hbox{$\displaystyle{\textstyle
-}{\int}$ } \vcenter{\hbox{$\textstyle -$
}}\kern-.6\wd0}}{{\setbox0=\hbox{$\textstyle{\scriptstyle -}{\int}$ } \vcenter{\hbox{$\scriptstyle -$
}}\kern-.6\wd0}}{{\setbox0=\hbox{$\scriptstyle{\scriptscriptstyle -}{\int}$
} \vcenter{\hbox{$\scriptscriptstyle -$
}}\kern-.6\wd0}}{{\setbox0=\hbox{$\scriptscriptstyle{\scriptscriptstyle
-}{\int}$ } \vcenter{\hbox{$\scriptscriptstyle -$ }}\kern-.6\wd0}}%
\!\int_{\varepsilon G_{2}}\left( \int_{0}^{1}\nabla _{\zeta }u_{\varepsilon
}(\overline{x},z+t(\xi -z))\cdot (\xi -z)dt\right) dz,
\end{eqnarray*}%
so that, using Young's and H\"{o}lder's inequalities, 
\begin{eqnarray*}
\left\vert u_{\varepsilon }(\overline{x},\xi )-M_{\varepsilon
}u_{\varepsilon }(\overline{x})\right\vert ^{p} &\leq &%
\mathchoice {{\setbox0=\hbox{$\displaystyle{\textstyle
-}{\int}$ } \vcenter{\hbox{$\textstyle -$
}}\kern-.6\wd0}}{{\setbox0=\hbox{$\textstyle{\scriptstyle -}{\int}$ } \vcenter{\hbox{$\scriptstyle -$
}}\kern-.6\wd0}}{{\setbox0=\hbox{$\scriptstyle{\scriptscriptstyle -}{\int}$
} \vcenter{\hbox{$\scriptscriptstyle -$
}}\kern-.6\wd0}}{{\setbox0=\hbox{$\scriptscriptstyle{\scriptscriptstyle
-}{\int}$ } \vcenter{\hbox{$\scriptscriptstyle -$ }}\kern-.6\wd0}}%
\!\int_{\varepsilon G_{2}}\int_{0}^{1}\left\vert \nabla _{\zeta
}u_{\varepsilon }(\overline{x},z+t(\xi -z))\right\vert ^{p}\left\vert \xi
-z\right\vert ^{p}dtdz \\
&\leq &%
\mathchoice {{\setbox0=\hbox{$\displaystyle{\textstyle
-}{\int}$ } \vcenter{\hbox{$\textstyle -$
}}\kern-.6\wd0}}{{\setbox0=\hbox{$\textstyle{\scriptstyle -}{\int}$ } \vcenter{\hbox{$\scriptstyle -$
}}\kern-.6\wd0}}{{\setbox0=\hbox{$\scriptstyle{\scriptscriptstyle -}{\int}$
} \vcenter{\hbox{$\scriptscriptstyle -$
}}\kern-.6\wd0}}{{\setbox0=\hbox{$\scriptscriptstyle{\scriptscriptstyle
-}{\int}$ } \vcenter{\hbox{$\scriptscriptstyle -$ }}\kern-.6\wd0}}%
\!\int_{\varepsilon G_{2}}\left\vert \xi -z\right\vert ^{p}dz\left(
\int_{\varepsilon G_{2}}\left\vert \nabla _{\zeta }u_{\varepsilon }(%
\overline{x},\eta )\right\vert ^{p}d\eta \right) \\
&\leq &C\varepsilon ^{p}\int_{\varepsilon G_{2}}\left\vert \nabla _{\zeta
}u_{\varepsilon }(\overline{x},\eta )\right\vert ^{p}d\eta ,
\end{eqnarray*}%
where $C>0$ depends only on $G_{2}$ and $d_{2}$. Integrating over $%
G_{\varepsilon }$ the last series of inequalities above gives (\ref{*1}).

With all that in mind, we infer from both (\ref{*0}) and the compactness of
the embedding $W^{1,p}(G_{1})\hookrightarrow L^{p}(G_{1})$, the existence of
a subsequence of $E^{\prime }$ not relabeled, such that, as $E^{\prime }\ni
\varepsilon \rightarrow 0$, 
\begin{equation}
M_{\varepsilon }u_{\varepsilon }\rightarrow u_{0}\text{ in }L^{p}(G_{1})%
\text{-strong.}  \label{*2}
\end{equation}%
Now the inequality (\ref{*1}) yields, as $E^{\prime }\ni \varepsilon
\rightarrow 0$, 
\begin{equation}
\varepsilon ^{-\frac{d_{2}}{p}}\left\Vert u_{\varepsilon }-M_{\varepsilon
}u_{\varepsilon }\right\Vert _{L^{p}(G_{\varepsilon })}\rightarrow 0\text{.}
\label{*13}
\end{equation}%
Next, we have 
\begin{equation*}
\varepsilon ^{-\frac{d_{2}}{p}}\left\Vert u_{\varepsilon }-u_{0}\right\Vert
_{L^{p}(G_{\varepsilon })}\leq \varepsilon ^{-\frac{d_{2}}{p}}\left\Vert
u_{\varepsilon }-M_{\varepsilon }u_{\varepsilon }\right\Vert
_{L^{p}(G_{\varepsilon })}+\varepsilon ^{-\frac{d_{2}}{p}}\left\Vert
M_{\varepsilon }u_{\varepsilon }-u_{0}\right\Vert _{L^{p}(G_{\varepsilon })},
\end{equation*}%
and 
\begin{equation*}
\varepsilon ^{-\frac{d_{2}}{p}}\left\Vert M_{\varepsilon }u_{\varepsilon
}-u_{0}\right\Vert _{L^{p}(G_{\varepsilon })}=\left\vert G_{2}\right\vert ^{%
\frac{1}{p}}\left\Vert M_{\varepsilon }u_{\varepsilon }-u_{0}\right\Vert
_{L^{p}(G_{1})}.
\end{equation*}%
It follows readily from (\ref{*2}) and (\ref{*13}) that, as $E^{\prime }\ni
\varepsilon \rightarrow 0$, 
\begin{equation*}
\varepsilon ^{-\frac{d_{2}}{p}}\left\Vert u_{\varepsilon }-u_{0}\right\Vert
_{L^{p}(G_{\varepsilon })}\rightarrow 0.
\end{equation*}%
The proof is complete.
\end{proof}

The next result and its corollary are proved exactly as their homologues in 
\cite[Theorem 6 and Corollary 5]{DPDE} (see also \cite{Deterhom}).

\begin{theorem}
\label{t2.4}Let $1<p,q<\infty $ and $r\geq 1$ be such that $1/r=1/p+1/q\leq
1 $. Assume $(u_{\varepsilon })_{\varepsilon \in E}\subset
L^{q}(G_{\varepsilon })$ is weakly $\Sigma _{\mathcal{A}}$-convergent in $%
L^{q}(G_{\varepsilon })$ to some $u_{0}\in L^{q}(G_{0};\mathcal{B}_{\mathcal{%
A}}^{q}(\mathbb{R}^{d_{1}};L^{q}(G_{2})))$, and $(v_{\varepsilon
})_{\varepsilon \in E}\subset L^{p}(G_{\varepsilon })$ is strongly $\Sigma _{%
\mathcal{A}}$-convergent in $L^{p}(G_{\varepsilon })$ to some $v_{0}\in
L^{p}(G_{0};\mathcal{B}_{\mathcal{A}}^{p}(\mathbb{R}^{d_{1}};L^{p}(G_{2})))$%
. Then the sequence $(u_{\varepsilon }v_{\varepsilon })_{\varepsilon \in E}$
is weakly $\Sigma _{\mathcal{A}}$-convergent in $L^{r}(G_{\varepsilon })$ to 
$u_{0}v_{0}$.
\end{theorem}

\begin{corollary}
\label{c2.2}Let $(u_{\varepsilon })_{\varepsilon \in E}\subset
L^{p}(G_{\varepsilon })$ and $(v_{\varepsilon })_{\varepsilon \in E}\subset
L^{p^{\prime }}(G_{\varepsilon })\cap L^{\infty }(G_{\varepsilon })$ ($%
1<p<\infty $ and $p^{\prime }=p/(p-1)$) be two sequences such that:

\begin{itemize}
\item[(i)] $u_{\varepsilon }\rightarrow u_{0}$ in $L^{p}(G_{\varepsilon })$%
-weak $\Sigma _{\mathcal{A}}$;

\item[(ii)] $v_{\varepsilon }\rightarrow v_{0}$ in $L^{p^{\prime
}}(G_{\varepsilon })$-strong $\Sigma _{\mathcal{A}}$;

\item[(iii)] $(v_{\varepsilon })_{\varepsilon \in E}$ is bounded in $%
L^{\infty }(G_{\varepsilon })$.
\end{itemize}

\noindent Then $u_{\varepsilon }v_{\varepsilon }\rightarrow u_{0}v_{0}$ in $%
L^{p}(G_{\varepsilon })$-weak $\Sigma _{\mathcal{A}}$.
\end{corollary}

Another important result is the following proposition.

\begin{proposition}
\label{p2.2}Assume that $\mathcal{A}$ is an ergodic algebra with mean value
on $\mathbb{R}^{d_{1}}$ and further that $G_{2}$ is connected. Let $%
(u_{\varepsilon })_{\varepsilon \in E}$ be a sequence in $%
W^{1,p}(G_{\varepsilon })$ such that 
\begin{equation*}
\sup_{\varepsilon \in E}\left( \varepsilon ^{-d_{2}/p}\left\Vert
u_{\varepsilon }\right\Vert _{L^{p}(G_{\varepsilon })}+\varepsilon
^{1-d_{2}/p}\left\Vert \nabla u_{\varepsilon }\right\Vert
_{L^{p}(G_{\varepsilon })}\right) \leq C
\end{equation*}%
where $C>0$ is independent of $\varepsilon $. Then there exist a subsequence 
$E^{\prime }$ of $E$ and a function $u\in L^{p}(G_{0};B_{\#\mathcal{A}%
}^{1,p}(\mathbb{R}^{d_{1}};W^{1,p}(G_{2})))$ with $u_{0}=\varrho (u)\in
L^{p}(G_{0};\mathcal{B}_{\mathcal{A}}^{1,p}(\mathbb{R}%
^{d_{1}};W^{1,p}(G_{2})))$ such that, as $E^{\prime }\ni \varepsilon
\rightarrow 0$, 
\begin{equation*}
u_{\varepsilon }\rightarrow u_{0}\text{ in }L^{p}(G_{\varepsilon })\text{%
-weak }\Sigma _{\mathcal{A}},
\end{equation*}%
and 
\begin{equation*}
\varepsilon \nabla u_{\varepsilon }\rightarrow \nabla _{\overline{y},\zeta }u%
\text{ in }L^{p}(G_{\varepsilon })^{d}\text{-weak }\Sigma _{\mathcal{A}}.
\end{equation*}
\end{proposition}

\begin{proof}
From Theorem \ref{t2.1}, we can find a subsequence $E^{\prime }$ from $E$
and a couple $(u_{0},u_{1})\in L^{p}(G_{0};\mathcal{B}_{A}^{p}(\mathbb{R}%
^{d_{1}};L^{p}(G_{2})))\times L^{p}(G_{0};\mathcal{B}_{A}^{p}(\mathbb{R}%
^{d_{1}};L^{p}(G_{2})))^{d}$ such that, as $E^{\prime }\ni \varepsilon
\rightarrow 0$, 
\begin{eqnarray*}
u_{\varepsilon } &\rightarrow &u_{0}\text{ in }L^{p}(G_{\varepsilon })\text{%
-weak }\Sigma _{\mathcal{A}}, \\
\varepsilon \nabla u_{\varepsilon } &\rightarrow &u_{1}\text{ in }%
L^{p}(G_{\varepsilon })^{d}\text{-weak }\Sigma _{\mathcal{A}}.
\end{eqnarray*}%
Let us characterize $u_{1}$ in terms of $u_{0}$. To that end, let $\Phi \in (%
\mathcal{C}_{0}^{\infty }(G_{0})\otimes \mathcal{A}^{\infty }(\mathbb{R}%
^{d_{1}};\mathcal{C}_{0}^{\infty }(G_{2})))^{d}$; then we have 
\begin{equation*}
\varepsilon ^{-d_{2}}\int_{G_{\varepsilon }}\varepsilon \nabla
u_{\varepsilon }\cdot \Phi ^{\varepsilon }dx=-\varepsilon
^{-d_{2}}\int_{G_{\varepsilon }}u_{\varepsilon }\left[ (\func{div}_{%
\overline{x}}\Phi )^{\varepsilon }+\frac{1}{\varepsilon }(\func{div}_{y}\Phi
)^{\varepsilon }\right] dx.
\end{equation*}%
Letting $E^{\prime }\ni \varepsilon \rightarrow 0$, we get 
\begin{equation}
\int_{G_{0}}\int_{G_{2}}M(u_{1}(\overline{x},\cdot ,\zeta )\cdot \Phi (%
\overline{x},\cdot ,\zeta ))d\zeta d\overline{x}=-\int_{G_{0}}%
\int_{G_{2}}M(u_{0}(\overline{x},\cdot ,\zeta )\func{div}_{y}\Phi (\overline{%
x},\cdot ,\zeta ))d\zeta d\overline{x}.  \label{*6}
\end{equation}%
This shows that $u_{1}=\overline{\nabla }_{\overline{y},\zeta }u_{0}$, so
that $u_{0}\in L^{p}(G_{0};\mathcal{B}_{\mathcal{A}}^{1,p}(\mathbb{R}%
^{d_{1}};W^{1,p}(G_{2})))$.

Now, coming back to (\ref{*6}) and choosing there $\Phi $ such that $\func{%
div}_{y}\Phi =0$, we readily get 
\begin{equation*}
\int_{G_{0}}\int_{G_{2}}M(u_{1}(\overline{x},\cdot ,\zeta )\cdot \Phi (%
\overline{x},\cdot ,\zeta ))d\zeta d\overline{x}=0\text{ for all such }\Phi 
\text{.}
\end{equation*}%
Owing to Corollary \ref{c2.1}, there exists $u\in L^{p}(G_{0};B_{\#\mathcal{A%
}}^{1,p}(\mathbb{R}^{d_{1}};W^{1,p}(G_{2})))$ such that $u_{1}=\nabla _{%
\overline{y},\zeta }u$. This yields, since $\mathcal{A}$ is ergodic and $%
G_{2}$ is connected, that $u_{0}=\varrho (u)+c$ where $c$ is a constant
possibly depending on $\overline{x}$. This shows that $u_{0}$ actually
belongs to $L^{p}(G_{0};B_{\#\mathcal{A}}^{1,p}(\mathbb{R}%
^{d_{1}};W^{1,p}(G_{2})))$ with $\overline{\nabla }_{\overline{y},\zeta
}u_{0}=\nabla _{\overline{y},\zeta }u$. This concludes the proof.
\end{proof}

\subsection{Sigma-convergence in thin heterogeneous domains with oscillating
boundaries}

For the sake of simplicity, we assume here that $d_{1}=d-1$ and $d_{2}=1$,
where integer $d\geq 2$. Let $h_{1}$, $h_{2}\in W^{1,\infty }(\mathbb{R}%
^{d-1})$ be two bounded Lipschitz continuous functions defined on $\mathbb{R}%
^{d-1}$ and satisfying $\max_{\mathbb{R}^{d-1}}h_{1}<\min_{\mathbb{R}%
^{d-1}}h_{2}$. Let $\Omega $ be a bounded open Lipschitz domain in $\mathbb{R%
}^{d-1}$. We define the thin heterogeneous domain with oscillating
boundaries, $\Omega ^{\varepsilon }$ in $\mathbb{R}^{d}$, as follows: 
\begin{equation*}
\Omega ^{\varepsilon }=\left\{ x=(\overline{x},x_{d})\in \mathbb{R}^{d}:%
\overline{x}\in \Omega \text{ and }\varepsilon h_{1}\left( \frac{\overline{x}%
}{\varepsilon }\right) <x_{d}<\varepsilon h_{2}\left( \frac{\overline{x}}{%
\varepsilon }\right) \right\} .
\end{equation*}%
We set 
\begin{equation*}
h_{1}^{-}=\min_{\mathbb{R}^{d-1}}h_{1}\text{ and }h_{2}^{+}=\max_{\mathbb{R}%
^{d-1}}h_{2},\text{\ \ \ \ }I=\left( h_{1}^{-},h_{2}^{+}\right) ,\ \ \ \ 
\end{equation*}%
and we define $G_{\varepsilon }=\Omega \times \left( \varepsilon
h_{1}^{-},\varepsilon h_{2}^{+}\right) $. Then $G_{\varepsilon }$ has flat
lateral boundaries $y_{d}=\varepsilon h_{1}^{-},\varepsilon h_{2}^{+}$, and
further, $\Omega ^{\varepsilon }\subset G_{\varepsilon }$. We also assume
that $0\in \lbrack h_{1}^{-},h_{2}^{+}]$.

For further needs, we assume that there exists an extension operator $%
P_{\varepsilon }:L^{p}(\Omega ^{\varepsilon })\rightarrow
L^{p}(G_{\varepsilon })$ such that 
\begin{equation*}
\left\Vert P_{\varepsilon }u\right\Vert _{L^{p}(G_{\varepsilon })}\leq
C\left\Vert u\right\Vert _{L^{p}(\Omega ^{\varepsilon })}\text{ for all }%
u\in L^{p}(\Omega ^{\varepsilon }),
\end{equation*}%
where $C$ is a positive constant independent of both $\varepsilon $ and $u$.
We set $\widetilde{u}=P_{\varepsilon }u$ for $u\in L^{p}(\Omega
^{\varepsilon })$.

Finally, we assume that 
\begin{equation}
h_{1},h_{2}\in \mathcal{A},\ \ \ \ \ \ \ \ \ \ \ \ \ \ \ \ \ \ \ \ \ \ \ \ \
\ \ \ \ \ \ \ \ \   \label{3.30}
\end{equation}%
where $\mathcal{A}$ is an algebra with mean value on $\mathbb{R}^{d-1}$. We
define the set 
\begin{equation*}
\mathbb{J}=\{y=(\overline{y},y_{d})\in \mathbb{R}^{d}:\overline{y}\in 
\mathbb{R}^{d-1}\text{ and }h_{1}(\overline{y})<y_{d}<h_{2}(\overline{y})\},
\end{equation*}%
and we denote by $\chi _{(h_{1},h_{2})}$ the characteristic function of $%
\mathbb{J}$:%
\begin{equation*}
\begin{array}{cc}
\chi _{\mathbb{J}}(y)=\chi _{(h_{1}(\overline{y}),h_{2}(\overline{y}%
))}(y_{d})= & \left\{ 
\begin{array}{l}
1\text{ if }y\in \mathbb{J} \\ 
0\text{ elsewhere.}%
\end{array}%
\right.%
\end{array}%
\end{equation*}%
The following result holds.

\begin{theorem}
\label{t3.5}Let $(u_{\varepsilon })_{\varepsilon \in E}\subset L^{p}(\Omega
^{\varepsilon })$ ($1\leq p<\infty $) be such that $P_{\varepsilon
}u_{\varepsilon }\rightarrow u_{0}$ in $L^{p}(G_{\varepsilon })$-weak $%
\Sigma _{\mathcal{A}}$ as $E\ni \varepsilon \rightarrow 0$, where $u_{0}\in
L^{p}(G_{0};\mathcal{B}_{\mathcal{A}}^{p}(\mathbb{R}^{d-1};L^{p}(I)))$.
Then, as $E\ni \varepsilon \rightarrow 0$, 
\begin{equation}
\frac{1}{\varepsilon }\int_{\Omega ^{\varepsilon }}u_{\varepsilon
}(x)f\left( \overline{x},\frac{x}{\varepsilon }\right) dx\rightarrow
\int_{G_{0}}\int_{I}M(\chi _{(h_{1},h_{2})}(y_{d})u_{0}(\overline{x},\cdot
,y_{d})f(\overline{x},\cdot ,y_{d}))dy_{d}d\overline{x}  \label{3.31}
\end{equation}%
for all $f\in L^{p^{\prime }}(G_{0};\mathcal{A}(\mathbb{R}^{d-1};\mathcal{C}(%
\overline{I})))$, $1/p^{\prime }=1-1/p$.
\end{theorem}

\begin{proof}
The proof is done in two steps.

\emph{Step 1}. Let $p$ be as above. Let us first check that $\chi _{\mathbb{J%
}}\in B_{\mathcal{A}}^{p}(\mathbb{R}^{d-1};L^{p}(I))$. To proceed with, we
need to check the following:

\begin{itemize}
\item[1)] The sequence $(\chi _{\mathbb{J}})^{\varepsilon }$ defined by $%
(\chi _{\mathbb{J}})^{\varepsilon }(x)=\chi _{(h_{1}(\overline{x}%
/\varepsilon ),h_{2}(\overline{x}/\varepsilon ))}(x_{d}/\varepsilon )$ for $%
x\in G_{\varepsilon }$ satisfies 
\begin{equation*}
\left\Vert (\chi _{\mathbb{J}})^{\varepsilon }\right\Vert
_{L^{p}(G_{\varepsilon })}\leq C\varepsilon ^{\frac{1}{p}},\ \ \ \ \ \ \ \ \
\ \ \ \ \ \ \ \ \ \ \ \ \ \ \ \ \ \ 
\end{equation*}%
so that, up to a subsequence, is weakly sigma-convergent towards some $u\in
L^{p}(G_{0};\mathcal{B}_{\mathcal{A}}^{p}(\mathbb{R}^{d-1};L^{p}(I)))$;

\item[2)] The limit $u$ determined above, has the form $u=\varrho (\chi _{%
\mathbb{J}})$. As a result, $\chi _{\mathbb{J}}\in B_{\mathcal{A}}^{p}(%
\mathbb{R}^{d-1};L^{p}(I))$.
\end{itemize}

Let us prove 1) above. We have 
\begin{eqnarray*}
\left\Vert (\chi _{\mathbb{J}})^{\varepsilon }\right\Vert
_{L^{p}(G_{\varepsilon })}^{p} &=&\int_{\Omega }\int_{\varepsilon I}\chi
_{(h_{1}^{\varepsilon }(\overline{x}),h_{2}^{\varepsilon }(\overline{x}%
))}\left( \frac{x_{d}}{\varepsilon }\right) dx=\varepsilon \int_{\Omega
}\int_{I}\chi _{(h_{1}^{\varepsilon },h_{2}^{\varepsilon })}(y_{d})d%
\overline{x}dy_{d} \\
&\leq &C\varepsilon \text{, }C=C(\Omega ,h_{1}^{-},h_{2}^{+})>0\text{ and }%
h_{i}^{\varepsilon }(\overline{x})=h_{i}(\overline{x}/\varepsilon )\text{.}
\end{eqnarray*}%
Thus, up to a subsequence of $E$ not relabelled, we have that 
\begin{equation}
(\chi _{\mathbb{J}})^{\varepsilon }\rightarrow u\text{ in }%
L^{p}(G_{\varepsilon })\text{-weak }\Sigma _{\mathcal{A}},\ \ \ \ \ \ \ \ \
\ \ \   \label{3.32}
\end{equation}%
where $u\in L^{p}(G_{0};\mathcal{B}_{\mathcal{A}}^{p}(\mathbb{R}%
^{d-1};L^{p}(I)))$.

Let us check point 2) above, that is, $u=\varrho (\chi _{\mathbb{J}})$,
where $\varrho $ is the canonical mapping of $B_{\mathcal{A}}^{p}(\mathbb{R}%
^{d-1};L^{p}(I))$ into $\mathcal{B}_{\mathcal{A}}^{p}(\mathbb{R}%
^{d-1};L^{p}(I))$. To this end, let $f\in L^{p^{\prime }}(G_{0};\mathcal{A}(%
\mathbb{R}^{d-1};\mathcal{C}(\overline{I})))$. Then, up to the same
subsequence as above, we have 
\begin{eqnarray*}
&&\frac{1}{\varepsilon }\int_{G_{\varepsilon }}\chi _{(h_{1}^{\varepsilon }(%
\overline{x}),h_{2}^{\varepsilon }(\overline{x}))}\left( \frac{x_{d}}{%
\varepsilon }\right) f\left( \overline{x},\frac{x}{\varepsilon }\right) dx \\
&=&\int_{\Omega }\int_{h_{1}^{-}}^{h_{2}^{+}}\chi _{(h_{1}^{\varepsilon }(%
\overline{x}),h_{2}^{\varepsilon }(\overline{x}))}(y_{d})f\left( \overline{x}%
,\frac{\overline{x}}{\varepsilon },y_{d}\right) dy_{d}d\overline{x} \\
&=&\int_{\Omega }\int_{h_{1}^{\varepsilon }(\overline{x})}^{h_{2}^{%
\varepsilon }(\overline{x})}f\left( \overline{x},\frac{\overline{x}}{%
\varepsilon },y_{d}\right) dy_{d}d\overline{x} \\
&=&\int_{\Omega }\int_{0}^{1}f\left( \overline{x},\frac{\overline{x}}{%
\varepsilon },(1-t)h_{1}^{\varepsilon }(\overline{x})+th_{2}^{\varepsilon }(%
\overline{x})\right) (h_{2}^{\varepsilon }(\overline{x})-h_{1}^{\varepsilon
}(\overline{x}))dtd\overline{x} \\
&\rightarrow &\int_{\Omega }\int_{0}^{1}M\left( [f\left( \overline{x},\cdot
,(1-t)h_{1}+th_{2}\right) ](h_{2}-h_{1})\right) d\overline{x} \\
&=&\int_{\Omega }M\left( \int_{h_{1}}^{h_{2}}f\left( \overline{x},\cdot
,y_{d}\right) dy_{d}\right) d\overline{x} \\
&=&\int_{\Omega }M\left( \int_{h_{1}^{-}}^{h_{2}^{+}}\chi _{\mathbb{J}%
}(\cdot ,y_{d})f\left( \overline{x},\cdot ,y_{d}\right) dy_{d}\right) d%
\overline{x} \\
&=&\int_{\Omega }\int_{I}M\left( \chi _{\mathbb{J}}(\cdot ,y_{d})f\left( 
\overline{x},\cdot ,y_{d}\right) \right) dy_{d}d\overline{x},
\end{eqnarray*}%
where here above, we have used the fact that, for any $t\in (0,1)$, the
function $\overline{y}\mapsto f(\cdot ,\overline{y},(1-t)h_{1}(\overline{y}%
)+th_{2}(\overline{y}))$ belongs to $L^{p}(G_{0};\mathcal{A})$ together with
the property of the mean value to obtain the part $^{\prime \prime
}\rightarrow ^{\prime \prime }$ and the property (\ref{e2.1}) (the
interchangeability of the integral and the mean value).

We infer from the uniqueness of the limit that $u=\varrho (\chi _{\mathbb{J}%
})$ since $\chi _{\mathbb{J}}\in L_{loc}^{p}(\mathbb{R}^{d})$. This gives at
once $\chi _{\mathbb{J}}\in B_{\mathcal{A}}^{p}(\mathbb{R}^{d-1};L^{p}(I))$,
as $\chi _{\mathbb{J}}$ is independent of $\overline{x}\in G_{0}$. As a
byproduct we have $\chi _{\mathbb{J}}\in B_{\mathcal{A}}^{\infty }(\mathbb{R}%
^{d-1};L^{\infty }(I))$.

\emph{Step 2}. Let $f\in \mathcal{C}(\overline{G}_{0})\otimes \mathcal{A}(%
\mathbb{R}^{d-1};\mathcal{C}(\overline{I}))$. Since $\chi _{\mathbb{J}}\in
B_{\mathcal{A}}^{p^{\prime }}(\mathbb{R}^{d-1};L^{p^{\prime }}(I))$, we have
that $\chi _{\mathbb{J}}f\in \mathcal{C}(\overline{G}_{0})\otimes B_{%
\mathcal{A}}^{p^{\prime }}(\mathbb{R}^{d-1};L^{p^{\prime }}(I))$. It can
therefore be taken as test function, so that 
\begin{eqnarray*}
\frac{1}{\varepsilon }\int_{\Omega ^{\varepsilon }}u_{\varepsilon
}(x)f\left( \overline{x},\frac{x}{\varepsilon }\right) dx &=&\frac{1}{%
\varepsilon }\int_{G_{\varepsilon }}(P_{\varepsilon }u_{\varepsilon
})(x)(\chi _{\mathbb{J}})\left( \frac{x}{\varepsilon }\right) f\left( 
\overline{x},\frac{x}{\varepsilon }\right) dx \\
&\rightarrow &\int_{G_{0}}\int_{h_{1}^{-}}^{h_{2}^{+}}M(u_{0}(\overline{x}%
,\cdot ,y_{d})\chi _{\mathbb{J}}(\cdot ,y_{d})f(\overline{x},\cdot
,y_{d}))dy_{d}d\overline{x}.
\end{eqnarray*}%
The convergence result (\ref{3.31}) follows from the last convergence result
above associated to the density of $\mathcal{C}(\overline{G}_{0})\otimes 
\mathcal{A}(\mathbb{R}^{d-1};\mathcal{C}(\overline{I}))$ in $L^{p^{\prime
}}(G_{0};\mathcal{A}(\mathbb{R}^{d-1};\mathcal{C}(\overline{I})))$.
\end{proof}

Theorems \ref{t2.1} and \ref{t2.2} have their evolutionary counterparts. To
see this, we first need to state the time-dependent version of the
sigma-convergence concept for thin heterogeneous domains. The domain $%
G_{\varepsilon }$ is defined as in the beginning of this section. Let $T$ be
a positive real number. All the notations are as in this section.

A sequence $(u_{\varepsilon })_{\varepsilon >0}\subset L^{p}((0,T)\times
G_{\varepsilon })$ is said to

\begin{itemize}
\item[(i)] weakly $\Sigma $-converge in $L^{p}((0,T)\times G_{\varepsilon })$
to $u_{0}\in L^{p}((0,T)\times G_{0};\mathcal{B}_{\mathcal{A}}^{p}(\mathbb{R}%
^{d_{1}};L^{p}(G_{2})))$ if as $\varepsilon \rightarrow 0$, 
\begin{eqnarray*}
&&\varepsilon ^{-d_{2}}\int_{(0,T)\times G_{\varepsilon }}u_{\varepsilon
}(t,x)f\left( t,\overline{x},\frac{x}{\varepsilon }\right) dxdt \\
&\rightarrow &\int_{(0,T)\times G_{0}}\int_{G_{2}}M(u_{0}(t,\overline{x}%
,\cdot ,\zeta )f(t,\overline{x},\cdot ,\zeta ))d\zeta d\overline{x}dt
\end{eqnarray*}%
for any $f\in L^{p^{\prime }}((0,T)\times G_{0};\mathcal{A}(\mathbb{R}%
^{d_{1}};L^{p^{\prime }}(G_{2})))$; we denote this by "$u_{\varepsilon
}\rightarrow u_{0}$ in $L^{p}((0,T)\times G_{\varepsilon })$-weak $\Sigma _{%
\mathcal{A}}$";

\item[(ii)] strongly $\Sigma $-converge in $L^{p}((0,T)\times G_{\varepsilon
})$ to $u_{0}\in L^{p}((0,T)\times G_{0};\mathcal{B}_{\mathcal{A}}^{p}(%
\mathbb{R}^{d_{1}};L^{p}(G_{2})))$ if it is weakly sigma-convergent and
further 
\begin{equation*}
\varepsilon ^{-d_{2}/p}\left\Vert u_{\varepsilon }\right\Vert
_{L^{p}((0,T)\times G_{\varepsilon })}\rightarrow \left\Vert
u_{0}\right\Vert _{L^{p}((0,T)\times G_{0};\mathcal{B}_{\mathcal{A}}^{p}(%
\mathbb{R}^{d_{1}};L^{p}(G_{2})))};
\end{equation*}%
we denote this by "$u_{\varepsilon }\rightarrow u_{0}$ in $L^{p}((0,T)\times
G_{\varepsilon })$-strong $\Sigma _{\mathcal{A}}$".
\end{itemize}

The time-dependent versions of Theorems \ref{t2.1} and \ref{t2.2} are stated
here below, and are proven exactly in the same way:

\begin{itemize}
\item Any sequence $(u_{\varepsilon })_{\varepsilon \in E}$ in $%
L^{p}((0,T)\times G_{\varepsilon })$ ($1<p<\infty $) such that 
\begin{equation*}
\sup_{\varepsilon \in E}\varepsilon ^{-d_{2}/p}\left\Vert u_{\varepsilon
}\right\Vert _{L^{p}((0,T)\times G_{\varepsilon })}<\infty
\end{equation*}%
possesses a weakly $\Sigma $-convergent subsequence;

\item Let $(u_{\varepsilon })_{\varepsilon \in E}$ be a sequence in $%
L^{p}(0,T;W^{1,p}(G_{\varepsilon }))$ ($1<p<\infty $) such that 
\begin{equation*}
\sup_{\varepsilon \in E}\left( \varepsilon ^{-d_{2}/p}\left( \left\Vert
u_{\varepsilon }\right\Vert _{L^{p}((0,T)\times G_{\varepsilon
})}+\left\Vert \nabla u_{\varepsilon }\right\Vert _{L^{p}((0,T)\times
G_{\varepsilon })}\right) \right) <\infty .
\end{equation*}%
Then there exist a subsequence $E^{\prime }$ of $E$ and a couple $%
(u_{0},u_{1})$ with $u_{0}\in L^{p}(0,T;W^{1,p}(G_{0}))$ and $u_{1}\in
L^{p}((0,T)\times G_{0};B_{\#\mathcal{A}}^{1,p}(\mathbb{R}%
^{d_{1}};W^{1,p}(G_{2})))$ such that, as $E^{\prime }\ni \varepsilon
\rightarrow 0$, 
\begin{equation*}
u_{\varepsilon }\rightarrow u_{0}\text{ in }L^{p}((0,T)\times G_{\varepsilon
})\text{-weak }\Sigma _{\mathcal{A}},
\end{equation*}%
\begin{equation*}
\frac{\partial u_{\varepsilon }}{\partial \overline{x}_{i}}\rightarrow \frac{%
\partial u_{0}}{\partial \overline{x}_{i}}+\frac{\partial u_{1}}{\partial 
\overline{y}_{i}}\text{ in }L^{p}((0,T)\times G_{\varepsilon })\text{-weak }%
\Sigma _{\mathcal{A}}\text{, }1\leq i\leq d_{1},
\end{equation*}%
\begin{equation*}
\frac{\partial u_{\varepsilon }}{\partial x_{d_{1}+i}}\rightarrow \frac{%
\partial u_{1}}{\partial \zeta _{i}}\text{ in }L^{p}((0,T)\times
G_{\varepsilon })\text{-weak }\Sigma _{\mathcal{A}}\text{, }1\leq i\leq
d_{2}.
\end{equation*}
\end{itemize}

The above time-dependent properties have been proved in \cite{CJW2024} (see
also \cite{PW2024}).

\section{Homogenization of the Darcy-Lapwood-Brinkmann equation in thin
heterogeneous domain: case of flat lateral boundaries\label{sec4}}

In this section, we deal with non oscillating boundaries.

\subsection{Statement of the problem and a priori estimates\label{subsec4.1}}

Let $\Omega $ be a bounded open connected Lipschitz subset in $\mathbb{R}%
^{2} $. For $\varepsilon >0$, we define the thin heterogeneous domain $%
\Omega ^{\varepsilon }$ in $\mathbb{R}^{3}$ by 
\begin{equation*}
\Omega ^{\varepsilon }=\Omega \times \left( -\varepsilon ,\varepsilon
\right) \equiv \{(\overline{x},x_{3})\in \mathbb{R}^{3}:\overline{x}\in
\Omega \text{ and }-\varepsilon <x_{3}<\varepsilon \}.
\end{equation*}%
In the fracture $\Omega ^{\varepsilon }$, the flow of fluid at the
micro-scale is described by the Darcy-Lapwood-Brinkmann (DLB) system 
\begin{equation}
\left\{ 
\begin{array}{l}
-\func{div}\left( A\left( \frac{x}{\varepsilon }\right) \nabla \boldsymbol{u}%
_{\varepsilon }\right) +\frac{\mu }{K_{\varepsilon }}\boldsymbol{u}%
_{\varepsilon }+\frac{\rho }{\phi ^{2}}(\boldsymbol{u}_{\varepsilon }\cdot
\nabla )\boldsymbol{u}_{\varepsilon }+\nabla p_{\varepsilon }=\boldsymbol{f}%
\text{ in }\Omega ^{\varepsilon } \\ 
\func{div}\boldsymbol{u}_{\varepsilon }=0\text{ in }\Omega ^{\varepsilon }%
\text{ and }\boldsymbol{u}_{\varepsilon }=0\text{ on }\partial \Omega
^{\varepsilon },%
\end{array}%
\right.  \label{4.1}
\end{equation}%
where

\begin{itemize}
\item[(\textbf{A1})] $A\in L^{\infty }(\mathbb{R}^{3})^{3\times 3}$ is a
symmetric matrix satisfying 
\begin{equation*}
\alpha \left\vert \lambda \right\vert ^{2}\leq A(y)\lambda \cdot \lambda
\leq \beta \left\vert \lambda \right\vert ^{2}\text{ for all }\lambda \in 
\mathbb{R}^{3}\text{ and a.e. }y\in \mathbb{R}^{3},
\end{equation*}%
$\alpha $ and $\beta $ being two positive real numbers;

\item[(\textbf{A2})] The right-hand side $\boldsymbol{f}$ has the form $%
\boldsymbol{f}(x)=(\boldsymbol{f}_{1}(\overline{x}),0)$ for a.e. $x=(%
\overline{x},x_{3})\in \Omega \times (-1,1)$, where $\boldsymbol{f}$ belongs
to $(L^{2}(\Omega \times (-1,1)))^{3}$;

\item[(\textbf{A3})] $A\in (B_{\mathcal{A}}^{2}(\mathbb{R}^{2};L^{\infty
}(I)))^{3\times 3}$, where $I=\left( -1,1\right) $.
\end{itemize}

In (\ref{4.1}), $\boldsymbol{u}_{\varepsilon }$ and $p_{\varepsilon }$ are
respectively the velocity of the fluid and the pressure; $\rho $ represents
the fluid density while $\phi $ stands for the porosity of the medium; $%
K_{\varepsilon }$ is the permeability of the porous medium and $\mu $ is the
dynamic coefficient of the viscosity.

With the above assumptions (\textbf{A1}) and (\textbf{A2}) on $A$ and $%
\boldsymbol{f}$ respectively, Eq. (\ref{4.1}) possesses at least (for each
fixed $\varepsilon >0$) a solution $(\boldsymbol{u}_{\varepsilon
},p_{\varepsilon })\in H_{0}^{1}(\Omega ^{\varepsilon })^{3}\times
L_{0}^{2}(\Omega ^{\varepsilon })$, where $L_{0}^{2}(\Omega ^{\varepsilon
})=\{u\in L^{2}(\Omega ^{\varepsilon }):\int_{\Omega ^{\varepsilon }}udx=0\}$%
.

For the sequel we adopt the following notation. If $A=(a_{ij})_{1\leq
i,j\leq 3}$ and $B=(b_{ij})_{1\leq i,j\leq 3}$ we set 
\begin{eqnarray*}
A\cdot B &=&\sum_{i,j=1}^{3}a_{ij}b_{ij}\text{, }AB=(c_{ij})_{1\leq i,j\leq
3}\text{ with }c_{ij}=\sum_{k=1}^{3}a_{ik}b_{kj}; \\
\text{For }x &=&(x_{i})_{1\leq i\leq 3}\text{ and }y=(y_{i})_{1\leq i\leq 3}%
\text{, }x\cdot y=\sum_{i=1}^{3}x_{i}y_{i}.
\end{eqnarray*}

The following technical result whose proof can be found in \cite{Marusic2000}
will be useful in the sequel.

\begin{lemma}
\label{l4.1}It holds that 
\begin{equation}
\left\Vert \boldsymbol{u}\right\Vert _{L^{2}(\Omega ^{\varepsilon
})^{3}}\leq C\varepsilon \left\Vert \nabla \boldsymbol{u}\right\Vert
_{L^{2}(\Omega ^{\varepsilon })^{3\times 3}},  \label{4.2}
\end{equation}%
\begin{equation}
\left\Vert \boldsymbol{u}\right\Vert _{L^{4}(\Omega ^{\varepsilon
})^{3}}\leq C\varepsilon ^{\frac{1}{2}}\left\Vert \nabla \boldsymbol{u}%
\right\Vert _{L^{2}(\Omega ^{\varepsilon })^{3\times 3}}  \label{4.3}
\end{equation}%
for any $\boldsymbol{u}\in H_{0}^{1}(\Omega ^{\varepsilon })^{3}$, where $C$
is a positive constant independent of $\varepsilon >0$.
\end{lemma}

The following estimates hold for the velocity.

\begin{proposition}
\label{p4.1}Let $\boldsymbol{u}_{\varepsilon }$ be determined by \emph{(\ref%
{4.1})}. Then for all $\varepsilon ,$%
\begin{equation}
\left\Vert \boldsymbol{u}_{\varepsilon }\right\Vert _{L^{2}(\Omega
^{\varepsilon })^{3}}\leq C\min \left( \varepsilon ^{\frac{5}{2}%
},\varepsilon ^{\frac{3}{2}}K_{\varepsilon }^{\frac{1}{2}}\right) ,\ \ \ \ \
\ \ \ \ \ \ \ \ \ \ \ \   \label{4.4}
\end{equation}%
\begin{equation}
\left\Vert \nabla \boldsymbol{u}_{\varepsilon }\right\Vert _{L^{2}(\Omega
^{\varepsilon })^{3\times 3}}\leq C\varepsilon ^{\frac{3}{2}},\ \ \ \ \ \ \
\ \ \ \ \ \ \ \ \ \ \ \ \ \ \ \ \ \ \ \ \ \ \ \ \   \label{4.5}
\end{equation}%
where $C$ is a positive constant independent of $\varepsilon $.
\end{proposition}

\begin{proof}
For any $\boldsymbol{\varphi }\in V=\{\boldsymbol{u}\in H_{0}^{1}(\Omega
^{\varepsilon })^{3}:\func{div}\boldsymbol{u}=0\}$, we have 
\begin{equation}
\int_{\Omega ^{\varepsilon }}A\left( \frac{x}{\varepsilon }\right) \nabla 
\boldsymbol{u}_{\varepsilon }\cdot \nabla \boldsymbol{\varphi }dx+\frac{\mu 
}{K_{\varepsilon }}\int_{\Omega ^{\varepsilon }}\boldsymbol{u}_{\varepsilon }%
\boldsymbol{\varphi }dx+\frac{\rho }{\phi ^{2}}\int_{\Omega ^{\varepsilon }}(%
\boldsymbol{u}_{\varepsilon }\cdot \nabla )\boldsymbol{u}_{\varepsilon }%
\boldsymbol{\varphi }dx=\int_{\Omega ^{\varepsilon }}\boldsymbol{f}%
\boldsymbol{\varphi }dx.  \label{4.6}
\end{equation}%
Choosing $\boldsymbol{\varphi }=\boldsymbol{u}_{\varepsilon }$ in (\ref{4.6}%
) and using the properties of the matrix $A$, the fact that $\int_{\Omega
^{\varepsilon }}(\boldsymbol{u}_{\varepsilon }\cdot \nabla )\boldsymbol{u}%
_{\varepsilon }\cdot \boldsymbol{u}_{\varepsilon }dx=0$ and the expression
of $\boldsymbol{f}$, we get 
\begin{equation}
\alpha \int_{\Omega ^{\varepsilon }}\left\vert \nabla \boldsymbol{u}%
_{\varepsilon }\right\vert ^{2}dx+\frac{\mu }{K_{\varepsilon }}\int_{\Omega
^{\varepsilon }}\left\vert \boldsymbol{u}_{\varepsilon }\right\vert
^{2}dx\leq \int_{\Omega ^{\varepsilon }}\boldsymbol{f}_{1}(\overline{x})%
\boldsymbol{u}_{\varepsilon }^{\prime }(x)dx,  \label{4.7}
\end{equation}%
where $\boldsymbol{u}_{\varepsilon }^{\prime }(x)=(u_{\varepsilon
,i})_{1\leq i\leq 2}$. Now using the fact that $\boldsymbol{f}_{1}\in
L^{2}(\Omega )^{2}$, (\ref{4.2}) and the Cauchy-Schwarz inequality, we
obtain 
\begin{equation*}
\left\vert \int_{\Omega ^{\varepsilon }}\boldsymbol{f}_{1}(\overline{x})%
\boldsymbol{u}_{\varepsilon }^{\prime }(x)dx\right\vert \leq C\varepsilon ^{%
\frac{3}{2}}\left\Vert \nabla \boldsymbol{u}_{\varepsilon }\right\Vert
_{L^{2}(\Omega ^{\varepsilon })^{3\times 3}}.
\end{equation*}%
It follows readily from (\ref{4.7}) that $\left\Vert \nabla \boldsymbol{u}%
_{\varepsilon }\right\Vert _{L^{2}(\Omega ^{\varepsilon })^{3\times 3}}\leq
C\varepsilon ^{\frac{3}{2}}$, that is, (\ref{4.5}). Using once again (\ref%
{4.2}), we get 
\begin{equation}
\left\Vert \boldsymbol{u}_{\varepsilon }\right\Vert _{L^{2}(\Omega
^{\varepsilon })^{3}}\leq C\varepsilon ^{\frac{5}{2}}.\ \ \ \ \ \ \ \ \ \ \
\ \ \ \ \ \ \ \   \label{4.8}
\end{equation}%
Coming back to (\ref{4.7}), we obtain 
\begin{equation}
\left\Vert \boldsymbol{u}_{\varepsilon }\right\Vert _{L^{2}(\Omega
^{\varepsilon })^{3}}\leq C\varepsilon ^{\frac{3}{2}}K_{\varepsilon }^{\frac{%
1}{2}}.\ \ \ \ \   \label{4.9}
\end{equation}%
Putting together (\ref{4.8}) and (\ref{4.9}) we are led to (\ref{4.4}). This
completes the proof.
\end{proof}

\begin{remark}
\label{r4.1}\emph{It follows from (\ref{4.4}) that }

\begin{itemize}
\item[(i)] \emph{If }$K_{\varepsilon }=O(\varepsilon ^{2})$\emph{\ or if }$%
K_{\varepsilon }\gg \varepsilon ^{2}$\emph{, then }$\left\Vert \boldsymbol{u}%
_{\varepsilon }\right\Vert _{L^{2}(\Omega ^{\varepsilon })^{3}}\leq
C\varepsilon ^{\frac{5}{2}};$

\item[(ii)] \emph{If }$K_{\varepsilon }\ll \varepsilon ^{2}$\emph{, then }$%
\left\Vert \boldsymbol{u}_{\varepsilon }\right\Vert _{L^{2}(\Omega
^{\varepsilon })^{3}}\leq C\varepsilon ^{\frac{3}{2}}K_{\varepsilon }^{\frac{%
1}{2}}.$
\end{itemize}
\end{remark}

Next, we need to derive the estimates for the pressure. To this end, we need
the following well known result whose proof can be found in \cite[Lemma 20]%
{Marusic2000}.

\begin{lemma}[{\protect\cite[Lemma 20]{Marusic2000}}]
\label{l4.10}For any $g_{\varepsilon }\in L_{0}^{2}(\Omega ^{\varepsilon })$%
, there exists a unique $\varphi _{\varepsilon }\in H_{0}^{1}(\Omega
^{\varepsilon })^{3}$ satisfying $\func{div}\varphi _{\varepsilon
}=g_{\varepsilon }$ and 
\begin{equation*}
\left\Vert \varphi _{\varepsilon }\right\Vert _{L^{2}(\Omega ^{\varepsilon
})^{3}}\leq C\left\Vert g_{\varepsilon }\right\Vert _{L^{2}(\Omega
^{\varepsilon })}\text{, }\left\Vert \nabla \varphi _{\varepsilon
}\right\Vert _{L^{2}(\Omega ^{\varepsilon })^{3\times 3}}\leq \frac{C}{%
\varepsilon }\left\Vert g_{\varepsilon }\right\Vert _{L^{2}(\Omega
^{\varepsilon })},
\end{equation*}%
where the positive constant $C$ is independent of $\varepsilon $.
\end{lemma}

The following result holds true.

\begin{proposition}
\label{p4.2}Let $p_{\varepsilon }\in L_{0}^{2}(\Omega ^{\varepsilon })$
satisfy \emph{(\ref{4.1})}. Then

\begin{itemize}
\item[(i)] If $K_{\varepsilon }=O(\varepsilon ^{2})$ or if $K_{\varepsilon
}\gg \varepsilon ^{2}$, then 
\begin{equation}
\left\Vert p_{\varepsilon }\right\Vert _{L^{2}(\Omega ^{\varepsilon })}\leq
C\varepsilon ^{\frac{1}{2}},\ \ \ \ \ \ \ \ \ \ \ \ \ \ \ \ \ \ \ \ \ \ \ \
\ \ \ \ \ \ \ \ \ \   \label{4.10}
\end{equation}

\item[(ii)] If $K_{\varepsilon }\ll \varepsilon ^{2}$, then 
\begin{equation}
\left\Vert p_{\varepsilon }\right\Vert _{L^{2}(\Omega ^{\varepsilon })}\leq
C\varepsilon ^{\frac{3}{2}}K_{\varepsilon }^{-1/2}.\ \ \ \ \ \ \ \ \ \ \ \ \
\ \ \ \ \ \ \ \ \ \ \ \ \ \ \   \label{4.11}
\end{equation}%
In \emph{(\ref{4.10})} and \emph{(\ref{4.11})}, $C$ is a positive constant
independent of $\varepsilon $.
\end{itemize}
\end{proposition}

\begin{proof}
Since $p_{\varepsilon }\in L_{0}^{2}(\Omega ^{\varepsilon })$, we appeal to
Lemma \ref{l4.10} to derive the existence of $\varphi _{\varepsilon }\in
H_{0}^{1}(\Omega ^{\varepsilon })^{3}$ with $\func{div}\varphi _{\varepsilon
}=p_{\varepsilon }$ and 
\begin{equation*}
\left\Vert \varphi _{\varepsilon }\right\Vert _{L^{2}(\Omega ^{\varepsilon
})^{3}}\leq C\left\Vert p_{\varepsilon }\right\Vert _{L^{2}(\Omega
^{\varepsilon })}\text{, }\left\Vert \nabla \varphi _{\varepsilon
}\right\Vert _{L^{2}(\Omega ^{\varepsilon })^{3\times 3}}\leq \frac{C}{%
\varepsilon }\left\Vert p_{\varepsilon }\right\Vert _{L^{2}(\Omega
^{\varepsilon })}.
\end{equation*}%
We choose $\varphi _{\varepsilon }$ as a test function in the variational
form of (\ref{4.1}) to get 
\begin{eqnarray*}
\left\Vert p_{\varepsilon }\right\Vert _{L^{2}(\Omega ^{\varepsilon })}^{2}
&=&-\left\langle \nabla p_{\varepsilon },\varphi _{\varepsilon
}\right\rangle =\int_{\Omega ^{\varepsilon }}A^{\varepsilon }\nabla 
\boldsymbol{u}_{\varepsilon }\cdot \nabla \varphi _{\varepsilon }dx+\frac{%
\mu }{K_{\varepsilon }}\int_{\Omega ^{\varepsilon }}\boldsymbol{u}%
_{\varepsilon }\varphi _{\varepsilon }dx \\
&&+\frac{\rho }{\phi ^{2}}\int_{\Omega ^{\varepsilon }}(\boldsymbol{u}%
_{\varepsilon }\cdot \nabla )\boldsymbol{u}_{\varepsilon }\cdot \varphi
_{\varepsilon }dx-\int_{\Omega ^{\varepsilon }}\boldsymbol{f}\varphi
_{\varepsilon }dx \\
&=&I_{1}+I_{2}+I_{3}+I_{4}.
\end{eqnarray*}%
One has 
\begin{equation*}
\left\vert I_{1}\right\vert \leq C\left\Vert \nabla \boldsymbol{u}%
_{\varepsilon }\right\Vert _{L^{2}(\Omega ^{\varepsilon })^{3\times
3}}\left\Vert \nabla \varphi _{\varepsilon }\right\Vert _{L^{2}(\Omega
^{\varepsilon })^{3\times 3}}\leq C\varepsilon ^{\frac{1}{2}}\left\Vert
p_{\varepsilon }\right\Vert _{L^{2}(\Omega ^{\varepsilon })},
\end{equation*}%
\begin{equation*}
\left\vert I_{2}\right\vert \leq \frac{\mu }{K_{\varepsilon }}\left\Vert 
\boldsymbol{u}_{\varepsilon }\right\Vert _{L^{2}(\Omega ^{\varepsilon
})^{3}}\left\Vert \varphi _{\varepsilon }\right\Vert _{L^{2}(\Omega
^{\varepsilon })^{3}}\leq \frac{C}{K_{\varepsilon }}\min (\varepsilon ^{%
\frac{5}{2}},\varepsilon ^{\frac{3}{2}}K_{\varepsilon }^{\frac{1}{2}%
})\left\Vert p_{\varepsilon }\right\Vert _{L^{2}(\Omega ^{\varepsilon })},
\end{equation*}%
\begin{eqnarray*}
\left\vert I_{3}\right\vert &\leq &\left\Vert \boldsymbol{u}_{\varepsilon
}\right\Vert _{L^{4}(\Omega ^{\varepsilon })^{3}}\left\Vert \nabla 
\boldsymbol{u}_{\varepsilon }\right\Vert _{L^{2}(\Omega ^{\varepsilon
})^{3\times 3}}\left\Vert \varphi _{\varepsilon }\right\Vert _{L^{4}(\Omega
^{\varepsilon })^{3}} \\
&\leq &C\varepsilon \left\Vert \nabla \boldsymbol{u}_{\varepsilon
}\right\Vert _{L^{2}(\Omega ^{\varepsilon })^{3\times 3}}^{2}\left\Vert
\nabla \varphi _{\varepsilon }\right\Vert _{L^{2}(\Omega ^{\varepsilon
})^{3\times 3}}\leq C\varepsilon ^{3}\left\Vert p_{\varepsilon }\right\Vert
_{L^{2}(\Omega ^{\varepsilon })},
\end{eqnarray*}%
and 
\begin{equation*}
\left\vert I_{4}\right\vert \leq C\varepsilon ^{\frac{1}{2}}\left\Vert
p_{\varepsilon }\right\Vert _{L^{2}(\Omega ^{\varepsilon })}.\ \ \ \ \ \ \ \
\ \ \ \ \ \ \ \ \ \ \ \ \ \ \ \ \ \ \ \ \ \ \ \ \ \ \ \ \ \ \ \ \ \ \ \ \ \
\ \ \ \ \ \ \ \ \ 
\end{equation*}%
It follows that 
\begin{equation}
\left\Vert p_{\varepsilon }\right\Vert _{L^{2}(\Omega ^{\varepsilon })}\leq
C\left( \varepsilon ^{\frac{1}{2}}+\varepsilon ^{3}+\frac{1}{K_{\varepsilon }%
}\min (\varepsilon ^{\frac{5}{2}},\varepsilon ^{\frac{3}{2}}K_{\varepsilon
}^{\frac{1}{2}})\right) .  \label{e4.9}
\end{equation}%
We note that the precise estimates should depend on the magnitude of $%
K_{\varepsilon }$ with respect to $\varepsilon $. Precisely,

\begin{itemize}
\item if $K_{\varepsilon }=O(\varepsilon ^{2})$, then (\ref{e4.9}) yields 
\begin{equation*}
\left\Vert p_{\varepsilon }\right\Vert _{L^{2}(\Omega ^{\varepsilon })}\leq
C\varepsilon ^{\frac{1}{2}};\ \ \ \ \ \ \ \ \ \ \ \ \ \ \ \ \ \ \ \ \ \ \ \
\ \ \ \ \ \ \ \ \ \ \ \ \ \ \ 
\end{equation*}

\item if $K_{\varepsilon }\ll \varepsilon ^{2}$, then $\varepsilon ^{\frac{3%
}{2}}K_{\varepsilon }^{\frac{1}{2}}<\varepsilon ^{\frac{5}{2}}$ and $%
\varepsilon ^{\frac{1}{2}}<\varepsilon ^{\frac{3}{2}}K_{\varepsilon }^{-%
\frac{1}{2}}$, so that 
\begin{equation*}
\left\Vert p_{\varepsilon }\right\Vert _{L^{2}(\Omega ^{\varepsilon })}\leq
C(\varepsilon ^{\frac{1}{2}}+\varepsilon ^{\frac{3}{2}}K_{\varepsilon }^{-%
\frac{1}{2}})\leq C\varepsilon ^{\frac{3}{2}}K_{\varepsilon }^{-\frac{1}{2}};
\end{equation*}

\item if $K_{\varepsilon }\gg \varepsilon ^{2}$, then 
\begin{equation*}
\left\Vert p_{\varepsilon }\right\Vert _{L^{2}(\Omega ^{\varepsilon })}\leq
C\varepsilon ^{\frac{1}{2}}.\ \ \ \ \ \ \ \ \ \ \ \ \ \ \ \ \ \ \ \ \ \ \ \
\ \ \ \ \ \ \ \ \ \ \ \ \ \ \ 
\end{equation*}
\end{itemize}
\end{proof}

According to the estimates in Propositions \ref{p4.1} and \ref{p4.2}, three
different regimes should be considered: 1) $K_{\varepsilon }=O(\varepsilon
^{2})$, 2) $K_{\varepsilon }<<\varepsilon ^{2}$ and 3) $K_{\varepsilon }\gg
\varepsilon ^{2}$. Each of these special cases will be worked out separately
in the following subsections.

For the sequel we identify $\Omega $ with $\Omega \times \{0\}$ so that any
point $(\overline{x},0)$ should be merely written as $\overline{x}$. We also
set $I=\left( -1,1\right) $.

\subsection{Homogenization results in the case when $K_{\protect\varepsilon %
}=O(\protect\varepsilon ^{2})$\label{subsec4.2}}

We assume that 
\begin{equation}
\frac{K_{\varepsilon }}{\varepsilon ^{2}}\rightarrow K\text{ when }%
\varepsilon \rightarrow 0\text{, with }0<K<\infty .  \label{4.12}
\end{equation}%
According to Propositions \ref{p4.1} and \ref{p4.2}, the following uniform
estimates hold: there exists a positive constant $C$ such that for all $%
\varepsilon >0$, 
\begin{equation}
\left\Vert \boldsymbol{u}_{\varepsilon }\right\Vert _{L^{2}(\Omega
^{\varepsilon })^{3}}\leq C\varepsilon ^{\frac{5}{2}},\ \left\Vert \nabla 
\boldsymbol{u}_{\varepsilon }\right\Vert _{L^{2}(\Omega ^{\varepsilon
})^{3\times 3}}\leq C\varepsilon ^{\frac{3}{2}}\text{ and }\left\Vert
p_{\varepsilon }\right\Vert _{L^{2}(\Omega ^{\varepsilon })}\leq
C\varepsilon ^{\frac{1}{2}}.  \label{4.13}
\end{equation}

Let $\mathcal{A}$ be an ergodic algebra with mean value on $\mathbb{R}^{2}$.
In view of Proposition \ref{p2.2} and Theorem \ref{t2.1}, given an ordinary
sequence $E$, there exist a subsequence $E^{\prime }$ of $E$ and $%
\boldsymbol{u}_{0}\in \lbrack L^{2}(\Omega ;\mathcal{B}_{\mathcal{A}}^{1,2}(%
\mathbb{R}^{2};H_{0}^{1}(I)))]^{3}$, $p_{0}\in L^{2}(\Omega ;\mathcal{B}_{%
\mathcal{A}}^{2}(\mathbb{R}^{2};L^{2}(I)))$ such that, as $E^{\prime }\ni
\varepsilon \rightarrow 0$, 
\begin{equation}
\frac{\boldsymbol{u}_{\varepsilon }}{\varepsilon ^{2}}\rightarrow 
\boldsymbol{u}_{0}\text{ in }L^{2}(\Omega ^{\varepsilon })^{3}\text{-weak }%
\Sigma _{\mathcal{A}},  \label{4.14}
\end{equation}%
\begin{equation}
\frac{1}{\varepsilon }\nabla \boldsymbol{u}_{\varepsilon }\rightarrow 
\overline{\nabla }_{y}\boldsymbol{u}_{0}\text{ in }L^{2}(\Omega
^{\varepsilon })^{3\times 3}\text{-weak }\Sigma _{\mathcal{A}},  \label{4.15}
\end{equation}%
\begin{equation}
p_{\varepsilon }\rightarrow p_{0}\text{ in }L^{2}(\Omega ^{\varepsilon })%
\text{-weak }\Sigma _{\mathcal{A}},  \label{4.16}
\end{equation}%
where in (\ref{4.15}) we put $y=(\overline{y},y_{3})$, so that $\overline{%
\nabla }_{y}=(\overline{\nabla }_{\overline{y}},\partial /\partial y_{3})$, $%
\overline{\nabla }_{\overline{y}}=(\overline{\partial }/\partial
y_{i})_{1\leq i\leq 2}$, $\overline{\partial }/\partial y_{i}$ being defined
in (\ref{e3}). Since $\func{div}\boldsymbol{u}_{\varepsilon }=0$ in $\Omega
^{\varepsilon }$, it follows that $\overline{\func{div}}_{y}\boldsymbol{u}%
_{0}=0$ in $\Omega \times \mathbb{R}^{2}\times I$. Indeed, setting 
\begin{equation*}
\boldsymbol{u}_{\varepsilon }^{\prime }=(u_{\varepsilon ,1},u_{\varepsilon
,2}),\ \ \ \ \ \ \ \ \ \ \ \ \ \ \ 
\end{equation*}%
we have, for $\boldsymbol{\varphi }\in \mathcal{C}_{0}^{\infty }(\Omega
)\otimes \mathcal{A}^{\infty }(\mathbb{R}^{2};\mathcal{C}_{0}^{\infty }(I))$%
, 
\begin{eqnarray*}
0 &=&\int_{\Omega ^{\varepsilon }}\func{div}\boldsymbol{u}_{\varepsilon }(x)%
\boldsymbol{\varphi }\left( \overline{x},\frac{x}{\varepsilon }\right) dx \\
&=&-\int_{\Omega ^{\varepsilon }}\boldsymbol{u}_{\varepsilon }^{\prime
}\cdot (\nabla _{\overline{x}}\boldsymbol{\varphi })^{\varepsilon }dx+\frac{1%
}{\varepsilon }\int_{\Omega ^{\varepsilon }}\boldsymbol{u}_{\varepsilon
}\cdot (\nabla _{y}\boldsymbol{\varphi })^{\varepsilon }dx,
\end{eqnarray*}%
where $\boldsymbol{\varphi }^{\varepsilon }(x)=\boldsymbol{\varphi }\left( 
\overline{x},\frac{x}{\varepsilon }\right) $ for $x\in \Omega ^{\varepsilon
} $. Multiplying the last equality above by $\varepsilon ^{-2}$ and letting $%
E^{\prime }\ni \varepsilon \rightarrow 0$ yields 
\begin{equation}
\int_{\Omega }\int_{-1}^{1}M(\boldsymbol{u}_{0}(\overline{x},\cdot
,y_{3})\cdot \nabla _{y}\boldsymbol{\varphi }(\overline{x},\cdot
,y_{3}))dy_{3}d\overline{x}=0.  \label{4.17}
\end{equation}%
This amounts to $\overline{\func{div}}_{y}\boldsymbol{u}_{0}=0$ in $\Omega
\times \mathbb{R}^{2}\times I$, where $\overline{\func{div}}_{y}\boldsymbol{u%
}_{0}=\overline{\func{div}}_{\overline{y}}\boldsymbol{u}_{0}^{\prime }+\frac{%
\partial u_{0,3}}{\partial y_{3}}$ with $\boldsymbol{u}_{0}^{\prime
}=(u_{0,i})_{1\leq i\leq 2}$.

Now, set 
\begin{eqnarray}
\boldsymbol{u}(\overline{x}) &=&\int_{-1}^{1}M(\boldsymbol{u}_{0}(\overline{x%
},\cdot ,y_{3}))dy_{3}\text{ for }\overline{x}\in \Omega  \label{4.18} \\
&=&(u_{i}(\overline{x}))_{1\leq i\leq 3}.  \notag
\end{eqnarray}%
Then $\boldsymbol{u}\in L^{2}(\Omega )^{3}$. Moreover 
\begin{equation}
\func{div}_{\overline{x}}\boldsymbol{u}^{\prime }=0\text{ in }\Omega \text{
and }\boldsymbol{u}^{\prime }\cdot \nu =0\text{ on }\partial \Omega ,
\label{4.19}
\end{equation}%
where $\nu $ is the outward unit normal to $\partial \Omega $. Here $%
\boldsymbol{u}=(\boldsymbol{u}^{\prime },u_{3})$. First of all, one has 
\begin{equation}
u_{3}=0\text{ in }\Omega \text{.\ \ \ \ \ \ \ \ \ \ \ \ \ \ \ \ \ \ \ \ \ \
\ \ \ }  \label{4.19'}
\end{equation}%
Indeed, from the equality $\overline{\func{div}}_{y}\boldsymbol{u}_{0}=0$ in 
$\Omega \times \mathbb{R}^{2}\times I$, we have $M(\overline{\func{div}}_{y}%
\boldsymbol{u}_{0})=0$, that is $\frac{\partial }{\partial y_{3}}M(u_{0,3}(%
\overline{x},\cdot ,y_{3}))=0$. This shows that $u_{0,3}$ is independent of $%
y_{3}$. But $u_{\varepsilon ,3}=0$ on $\Omega \times \{\varepsilon \}$, so
that $M(u_{0,3}(\overline{x},\cdot ,y_{3}))=0$ on $\Omega \times \{1\}$,
i.e. $M(u_{0,3}(\overline{x},\cdot ))=0$ in $\Omega $ since $u_{0,3}$ does
not depend on $y_{3}$. This shows that $\boldsymbol{u}=(\boldsymbol{u}%
^{\prime },0)$.

This being so, let us check (\ref{4.19}). To that end, let $\varphi \in 
\mathcal{D}(\overline{\Omega })$. Using the Stokes formula together with the
equality $\func{div}\boldsymbol{u}_{\varepsilon }=0$ in $\Omega
^{\varepsilon }$, we obtain 
\begin{equation*}
\int_{\Omega ^{\varepsilon }}\boldsymbol{u}_{\varepsilon }^{\prime }(x)\cdot
\nabla _{\overline{x}}\varphi (\overline{x})dx=0.\ \ \ \ \ \ \ \ \ \ \ \ \ \ 
\end{equation*}%
Dividing the last equality above by $\varepsilon ^{3}$ and letting $%
E^{\prime }\ni \varepsilon \rightarrow 0$, we are led to 
\begin{equation*}
\int_{\Omega }\boldsymbol{u}^{\prime }(x)\cdot \nabla _{\overline{x}}\varphi
(\overline{x})d\overline{x}=0.\ \ \ \ \ \ \ \ \ \ \ \ \ \ \ \ \ 
\end{equation*}%
This yields at once (\ref{4.19}).

We are now able to pass to the limit in the critical case $K_{\varepsilon
}=O(\varepsilon ^{2})$.

\begin{theorem}
\label{t4.1}Let $\mathcal{A}$ be an ergodic algebra with mean value on $%
\mathbb{R}^{2}$. Assume that \emph{(\textbf{A3})} and \emph{(\ref{4.12})}
hold true. Let $(\boldsymbol{u}_{\varepsilon },p_{\varepsilon })$ be
determined by \emph{(\ref{4.1})}. Then, when $E^{\prime }\ni \varepsilon
\rightarrow 0$, we have \emph{(\ref{4.14})}, \emph{(\ref{4.15})} and \emph{(%
\ref{4.16})}. Moreover there exists $q\in L^{2}(\Omega ;\mathcal{B}_{%
\mathcal{A}}^{2}(\mathbb{R}^{2};L^{2}(I)))$ such that $(\boldsymbol{u}%
_{0},p_{0},q)$ solves the equation 
\begin{equation}
\left\{ 
\begin{array}{l}
-\overline{\func{div}}_{y}\left( A(y)\overline{\nabla }_{y}\boldsymbol{u}%
_{0}\right) +\frac{\mu }{K}\boldsymbol{u}_{0}+\overline{\nabla }_{y}q=%
\boldsymbol{f}_{1}-\nabla _{\overline{x}}p_{0}\text{ in }\Omega \times 
\mathbb{R}^{2}\times I, \\ 
\\ 
\overline{\func{div}}_{y}\boldsymbol{u}_{0}=0\text{ in }\Omega \times 
\mathbb{R}^{2}\times I, \\ 
\\ 
\func{div}_{\overline{x}}\left( \int_{-1}^{1}M(\boldsymbol{u}_{0}(\overline{x%
},\cdot ,y_{3}))dy_{3}\right) =0\text{ in }\Omega , \\ 
\\ 
\left( \int_{-1}^{1}M(\boldsymbol{u}_{0}(\overline{x},\cdot
,y_{3}))dy_{3}\right) \cdot \nu =0\text{ on }\partial \Omega \text{.}%
\end{array}%
\right.  \label{4.20}
\end{equation}
\end{theorem}

\begin{proof}
Let $(\boldsymbol{u}_{0},p_{0})$ be determined by (\ref{4.14})-(\ref{4.16})
such that (\ref{4.18}) and (\ref{4.19}) are satisfied. Let us first show
that $p_{0}$ is independent of $y=(\overline{y},y_{3})$. To that end, let $%
\boldsymbol{\varphi }\in (\mathcal{C}_{0}^{\infty }(\Omega )\otimes \mathcal{%
A}^{\infty }(\mathbb{R}^{2};\mathcal{C}_{0}^{\infty }(I)))^{3}$. Testing (%
\ref{4.1}) against $\boldsymbol{\varphi }^{\varepsilon }(x)=\boldsymbol{%
\varphi }(\overline{x},x/\varepsilon )$ ($x\in \Omega ^{\varepsilon }$), we
have 
\begin{equation}
\begin{array}{l}
\int_{\Omega ^{\varepsilon }}A\left( \frac{x}{\varepsilon }\right) \nabla 
\boldsymbol{u}_{\varepsilon }\cdot \left[ (\nabla _{\overline{x}}\boldsymbol{%
\varphi })^{\varepsilon }+\frac{1}{\varepsilon }(\nabla _{y}\boldsymbol{%
\varphi })^{\varepsilon }\right] dx+\mu \frac{\varepsilon ^{2}}{%
K_{\varepsilon }}\int_{\Omega ^{\varepsilon }}\frac{\boldsymbol{u}%
_{\varepsilon }}{\varepsilon ^{2}}\boldsymbol{\varphi }^{\varepsilon }dx \\ 
\\ 
+\frac{\rho }{\phi ^{2}}\int_{\Omega ^{\varepsilon }}(\boldsymbol{u}%
_{\varepsilon }\cdot \nabla )\boldsymbol{u}_{\varepsilon }\boldsymbol{%
\varphi }^{\varepsilon }dx-\int_{\Omega ^{\varepsilon }}p_{\varepsilon }%
\left[ (\func{div}_{\overline{x}}\boldsymbol{\varphi })^{\varepsilon }+\frac{%
1}{\varepsilon }(\func{div}_{y}\boldsymbol{\varphi })^{\varepsilon }\right]
dx \\ 
\\ 
=\int_{\Omega ^{\varepsilon }}\boldsymbol{f}\boldsymbol{\varphi }%
^{\varepsilon }dx.%
\end{array}
\label{4.21}
\end{equation}%
Using the first two estimates in (\ref{4.13}) together with (\ref{4.3}) give 
\begin{eqnarray*}
\left\vert \int_{\Omega ^{\varepsilon }}(\boldsymbol{u}_{\varepsilon }\cdot
\nabla )\boldsymbol{u}_{\varepsilon }\boldsymbol{\varphi }^{\varepsilon
}dx\right\vert &\leq &\left\Vert \boldsymbol{u}_{\varepsilon }\right\Vert
_{L^{4}(\Omega ^{\varepsilon })}\left\Vert \nabla \boldsymbol{u}%
_{\varepsilon }\right\Vert _{L^{2}(\Omega ^{\varepsilon })}\left\Vert 
\boldsymbol{\varphi }^{\varepsilon }\right\Vert _{L^{4}(\Omega ^{\varepsilon
})} \\
&\leq &C\varepsilon \left\Vert \nabla \boldsymbol{u}_{\varepsilon
}\right\Vert _{L^{2}(\Omega ^{\varepsilon })}^{2}\left\Vert \nabla 
\boldsymbol{\varphi }^{\varepsilon }\right\Vert _{L^{2}(\Omega ^{\varepsilon
})} \\
&\leq &C\varepsilon ^{3}.
\end{eqnarray*}%
Thus 
\begin{equation}
\frac{1}{\varepsilon }\int_{\Omega ^{\varepsilon }}(\boldsymbol{u}%
_{\varepsilon }\cdot \nabla )\boldsymbol{u}_{\varepsilon }\boldsymbol{%
\varphi }^{\varepsilon }dx\rightarrow 0\text{ when }E^{\prime }\ni
\varepsilon \rightarrow 0.  \label{4.22}
\end{equation}%
Hence, passing to the limit \ when $E^{\prime }\ni \varepsilon \rightarrow 0$
in (\ref{4.21}) using (\ref{4.14})-(\ref{4.16}) leads to 
\begin{equation*}
\int_{\Omega }\int_{-1}^{1}M(p_{0}(\overline{x},\cdot ,y_{3})\func{div}_{y}%
\boldsymbol{\varphi }(\overline{x},\cdot ,y_{3}))dy_{3}d\overline{x}=0,
\end{equation*}%
which means that $p_{0}$ does not depend on $y=(\overline{y},y_{3})$.

Now, coming back to (\ref{4.21}) and choosing there $\boldsymbol{\varphi }$
such that $\func{div}_{y}\boldsymbol{\varphi }=0$, and next dividing both
sides of the resulting equation by $\varepsilon $, we get 
\begin{equation}
\begin{array}{l}
\frac{1}{\varepsilon }\int_{\Omega ^{\varepsilon }}A\left( \frac{x}{%
\varepsilon }\right) \nabla \boldsymbol{u}_{\varepsilon }\cdot \left[
(\nabla _{\overline{x}}\boldsymbol{\varphi })^{\varepsilon }+\frac{1}{%
\varepsilon }(\nabla _{y}\boldsymbol{\varphi })^{\varepsilon }\right] dx+\mu 
\frac{\varepsilon ^{2}}{K_{\varepsilon }}\frac{1}{\varepsilon }\int_{\Omega
^{\varepsilon }}\frac{\boldsymbol{u}_{\varepsilon }}{\varepsilon ^{2}}%
\boldsymbol{\varphi }^{\varepsilon }dx \\ 
+\frac{\rho }{\phi ^{2}}\frac{1}{\varepsilon }\int_{\Omega ^{\varepsilon }}(%
\boldsymbol{u}_{\varepsilon }\cdot \nabla )\boldsymbol{u}_{\varepsilon }%
\boldsymbol{\varphi }^{\varepsilon }dx-\frac{1}{\varepsilon }\int_{\Omega
^{\varepsilon }}p_{\varepsilon }(\func{div}_{\overline{x}}\boldsymbol{%
\varphi })^{\varepsilon }dx=\frac{1}{\varepsilon }\int_{\Omega ^{\varepsilon
}}\boldsymbol{f}\boldsymbol{\varphi }^{\varepsilon }dx.%
\end{array}
\label{4.23}
\end{equation}%
We pass to the limit in (\ref{4.23}) by considering each term separately.
First, since $A\in (B_{\mathcal{A}}^{2}(\mathbb{R}^{2};L^{2}(I))\cap
L^{\infty }(\mathbb{R}^{2}\times I))^{3\times 3}$, the matrix-functions $%
A\nabla _{\overline{x}}\boldsymbol{\varphi }$ and $A\nabla _{y}\boldsymbol{%
\varphi }$ can be seen as test functions for the weak $\Sigma _{\mathcal{A}}$%
-convergence. Therefore, appealing to (\ref{4.14})-(\ref{4.16}) in
conjunction with (\ref{4.12}), we pass to the limit when $E^{\prime }\ni
\varepsilon \rightarrow 0$ in (\ref{4.23}) and obtain 
\begin{equation}
\begin{array}{l}
\int_{\Omega }\int_{-1}^{1}M(A\overline{\nabla }_{y}\boldsymbol{u}_{0}\cdot
\nabla _{y}\boldsymbol{\varphi })dy_{3}d\overline{x}+\frac{\mu }{K}%
\int_{\Omega }\int_{-1}^{1}M(\boldsymbol{u}_{0}\boldsymbol{\varphi })dy_{3}d%
\overline{x} \\ 
-\int_{\Omega }\int_{-1}^{1}p_{0}(\overline{x})M(\func{div}_{\overline{x}}%
\boldsymbol{\varphi })dy_{3}d\overline{x}=\int_{\Omega }\int_{-1}^{1}M(%
\boldsymbol{f}\boldsymbol{\varphi })dy_{3}d\overline{x}.%
\end{array}
\label{4.24}
\end{equation}%
Since (\ref{4.24}) holds for every $\boldsymbol{\varphi }\in (\mathcal{C}%
_{0}^{\infty }(\Omega )\otimes \mathcal{A}^{\infty }(\mathbb{R}^{2};\mathcal{%
C}_{0}^{\infty }(I)))^{3}$ with $\func{div}_{y}\boldsymbol{\varphi }=0$, we
deduce from Proposition \ref{p2.1} the existence of $q\in L^{2}(\Omega ;%
\mathcal{B}_{\mathcal{A}}^{2}(\mathbb{R}^{2};L^{2}(I)))$ such that (\ref%
{4.20}) holds. The proof is therefore completed.
\end{proof}

We are now in a position to prove the main result of the current subsection,
which besides, is one of the main result of the work.

\begin{theorem}
\label{t4.2}Under the assumptions of Theorem \emph{\ref{t4.1}}, the sequence 
$(\boldsymbol{u}_{\varepsilon }/\varepsilon ^{2},p_{\varepsilon
})_{\varepsilon >0}$ weakly $\Sigma _{\mathcal{A}}$-converges in $%
L^{2}(\Omega ^{\varepsilon })^{3}\times L_{0}^{2}(\Omega ^{\varepsilon })$
towards $(\boldsymbol{u}_{0},p_{0})$ determined by \emph{(\ref{4.14})-(\ref%
{4.16})}. Moreover, $p_{0}\in H^{1}(\Omega )$ and, defining $\boldsymbol{u}=(%
\boldsymbol{u}^{\prime },u_{3})$ by \emph{(\ref{4.18})}, one has $u_{3}=0$
and $(\boldsymbol{u}^{\prime },p_{0})$ is the unique solution of the
effective problem 
\begin{equation}
\left\{ 
\begin{array}{l}
\boldsymbol{u}^{\prime }=\widehat{A}(\boldsymbol{f}_{1}-\nabla _{\overline{x}%
}p_{0})\text{ in }\Omega \\ 
\func{div}_{\overline{x}}\boldsymbol{u}^{\prime }=0\text{ in }\Omega \text{
and }\boldsymbol{u}^{\prime }\cdot \nu =0\text{ on }\partial \Omega ,%
\end{array}%
\right.  \label{4.26}
\end{equation}%
where $\widehat{A}=(\widehat{a}_{ij})_{1\leq i,j\leq 2}$ is a symmetric,
positive definite $2\times 2$ matrix defined by its entries 
\begin{equation*}
\widehat{a}_{ij}=\int_{-1}^{1}M(A\overline{\nabla }_{y}\boldsymbol{w}%
_{i}\cdot \overline{\nabla }_{y}\boldsymbol{w}_{j})dy_{3}+\frac{\mu }{K}%
\int_{-1}^{1}M(\boldsymbol{w}_{i}\boldsymbol{w}_{j})dy_{3},\ 1\leq i,j\leq 2.
\end{equation*}%
Here $\boldsymbol{w}_{i}$ $(1\leq i\leq 2)$ is the unique solution in $(%
\mathcal{B}_{\mathcal{A}}^{1,2}(\mathbb{R}^{2};H_{0}^{1}(I)))^{3}$ of the
Stokes-Brinkmann system 
\begin{equation}
\left\{ 
\begin{array}{l}
-\overline{\func{div}}_{y}\left( A(y)\overline{\nabla }_{y}\boldsymbol{w}%
_{i}\right) +\frac{\mu }{K}\boldsymbol{w}_{i}+\overline{\nabla }_{y}\pi
_{i}=e_{i}\text{ in }\mathbb{R}^{2}\times I \\ 
\overline{\func{div}}_{y}\boldsymbol{w}_{i}=0\text{ in }\mathbb{R}^{2}\times
I,%
\end{array}%
\right.  \label{4.27}
\end{equation}%
$e_{i}$ being the $i$\emph{th} vector of the canonical basis in $\mathbb{R}%
^{3}$.
\end{theorem}

\begin{proof}
First and foremost, let us prove that (\ref{4.27}) possesses a unique
solution $\boldsymbol{w}_{i}\in (\mathcal{B}_{\mathcal{A}}^{1,2}(\mathbb{R}%
^{2};H_{0}^{1}(I)))^{3}$. It is a fact that (\ref{4.27}) is equivalent to 
\begin{equation}
\left\{ 
\begin{array}{l}
\int_{-1}^{1}M(A\overline{\nabla }_{y}\boldsymbol{w}_{i}\cdot \overline{%
\nabla }_{y}\boldsymbol{v})dy_{3}+\frac{\mu }{K}\int_{-1}^{1}M(\boldsymbol{w}%
_{i}\boldsymbol{v})dy_{3}=\int_{-1}^{1}M(\boldsymbol{v})e_{i}dy_{3} \\ 
\\ 
\text{for all }\boldsymbol{v}\in (\mathcal{B}_{\mathcal{A}}^{1,2}(\mathbb{R}%
^{2};H_{0}^{1}(I)))^{3}\text{ with }\overline{\func{div}}_{y}\boldsymbol{v}%
=0.%
\end{array}%
\right.  \label{4.27'}
\end{equation}%
In view of the assumption (\textbf{A1}) on $A$, (\ref{4.27'}) possesses a
unique solution in $\mathcal{B}_{\func{div}}^{1,2}=\{\boldsymbol{v}\in (%
\mathcal{B}_{\mathcal{A}}^{1,2}(\mathbb{R}^{2};H_{0}^{1}(I)))^{3}:\overline{%
\func{div}}_{y}\boldsymbol{v}=0\}$.

Next, we recall that by density, (\ref{4.24}) still holds for $\boldsymbol{%
\varphi }\in L^{2}(\Omega ;\mathcal{B}_{\func{div}}^{1,2})$, so that,
choosing in (\ref{4.24}) the test function $\boldsymbol{\varphi }=\psi
\otimes \boldsymbol{w}_{i}$ with $\psi \in \mathcal{C}_{0}^{\infty }(\Omega
) $ and $\boldsymbol{w}_{i}$ ($1\leq i\leq 3$) defined by (\ref{4.27'}), we
obtain, after integrating by parts, 
\begin{eqnarray}
&&\int_{-1}^{1}M(A\overline{\nabla }_{y}\boldsymbol{u}_{0}\cdot \overline{%
\nabla }_{y}\boldsymbol{w}_{i})dy_{3}+\frac{\mu }{K}\int_{-1}^{1}M(%
\boldsymbol{u}_{0}\boldsymbol{w}_{i})dy_{3}+\nabla _{\overline{x}%
}p_{0}\int_{-1}^{1}M(\boldsymbol{w}_{i})dy_{3}  \label{4.28} \\
&=&\boldsymbol{f}(x)\int_{-1}^{1}M(\boldsymbol{w}_{i})dy_{3}.  \notag
\end{eqnarray}%
Taking in (\ref{4.27'}) the test function $\boldsymbol{u}_{0}(\overline{x}%
,\cdot )$, we obtain 
\begin{equation}
\left\{ 
\begin{array}{l}
\int_{-1}^{1}M(A\overline{\nabla }_{y}\boldsymbol{w}_{i}\cdot \overline{%
\nabla }_{y}\boldsymbol{u}_{0})dy_{3}+\frac{\mu }{K}\int_{-1}^{1}M(%
\boldsymbol{u}_{0}\boldsymbol{w}_{i})dy_{3}=\int_{-1}^{1}M(\boldsymbol{u}%
_{0})e_{i}dy_{3} \\ 
\\ 
\ \ \ =\int_{-1}^{1}M(u_{0,i})dy_{3}=u_{i}(\overline{x})\text{, }1\leq i\leq
3\text{.}%
\end{array}%
\right.  \label{4.29}
\end{equation}%
Using the fact that $A$ is symmetric, we obtain 
\begin{equation*}
\int_{-1}^{1}M(A\overline{\nabla }_{y}\boldsymbol{u}_{0}\cdot \overline{%
\nabla }_{y}\boldsymbol{w}_{i})dy_{3}=\int_{-1}^{1}M(A\overline{\nabla }_{y}%
\boldsymbol{w}_{i}:\overline{\nabla }_{y}\boldsymbol{u}_{0})dy_{3},
\end{equation*}%
so that, comparing (\ref{4.28}) and (\ref{4.29}), we get 
\begin{equation}
u_{i}(\overline{x})=\sum_{j=1}^{2}\left( \int_{-1}^{1}M(\boldsymbol{w}%
_{i})e_{j}dy_{3}\right) \left( f_{1,j}(\overline{x})-\frac{\partial p_{0}}{%
\partial \overline{x}_{j}}(\overline{x})\right) .  \label{4.29'}
\end{equation}%
But 
\begin{equation*}
\int_{-1}^{1}M(\boldsymbol{w}_{i})e_{j}dy_{3}=\int_{-1}^{1}M(A\overline{%
\nabla }_{y}\boldsymbol{w}_{i}:\overline{\nabla }_{y}\boldsymbol{w}%
_{j})dy_{3}+\frac{\mu }{K}\int_{-1}^{1}M(\boldsymbol{w}_{i}\boldsymbol{w}%
_{j})dy_{3},
\end{equation*}%
in such a way that, setting 
\begin{equation*}
\widehat{a}_{ij}=\int_{-1}^{1}M(A\overline{\nabla }_{y}\boldsymbol{w}%
_{i}\cdot \overline{\nabla }_{y}\boldsymbol{w}_{j})dy_{3}+\frac{\mu }{K}%
\int_{-1}^{1}M(\boldsymbol{w}_{i}\boldsymbol{w}_{j})dy_{3},
\end{equation*}%
and accounting of (\ref{4.19'}) (that is $\int_{-1}^{1}M(u_{0,3})dy_{3}=0$),
we deduce $\widehat{a}_{i3}=0$ for $1\leq i\leq 2$. Since $\widehat{a}_{ij}=%
\widehat{a}_{ji}$, we also deduce that $\widehat{a}_{3i}=0$ for $1\leq i\leq
2$. This shows that $\widehat{A}=(\widehat{a}_{ij})_{1\leq i,j\leq 2}$ is a $%
2\times 2$ symmetric and positive definite matrix. From (\ref{4.29'}) we get
at once 
\begin{equation}
\boldsymbol{u}^{\prime }(\overline{x})=\widehat{A}(\boldsymbol{f}_{1}(%
\overline{x})-\nabla _{\overline{x}}p_{0}(\overline{x})).\ \ \ \ \ \ \ \ \ \
\ \ \ \ \ \ \   \label{4.30'}
\end{equation}%
Now using (\ref{4.19}) together with the fact that (\ref{4.30'}) holds in
the classical sense of distributions in $\Omega $, we get that $p_{0}$
solves the problem 
\begin{equation}
\func{div}_{\overline{x}}(\widehat{A}(\boldsymbol{f}_{1}(\overline{x}%
)-\nabla _{\overline{x}}p_{0}(\overline{x})))=0\text{ in }\Omega \text{ and }%
(\widehat{A}(\boldsymbol{f}_{1}(\overline{x})-\nabla _{\overline{x}}p_{0}(%
\overline{x})))\cdot \nu =0\text{ on }\partial \Omega ,  \label{4.30''}
\end{equation}%
which shows that $p_{0}\in H^{1}(\Omega )$ is uniquely determined by (\ref%
{4.30''}). The convergence of the whole sequence $(\boldsymbol{u}%
_{\varepsilon }/\varepsilon ^{2},p_{\varepsilon })_{\varepsilon >0}$ is a
consequence of the uniqueness of the solution to (\ref{4.30''}) (and hence
to (\ref{4.26})). This completes the proof of the theorem.
\end{proof}

\subsection{Homogenization results: case when $K_{\protect\varepsilon }\gg 
\protect\varepsilon ^{2}$\label{subsec4.3}}

In view of Propositions \ref{p4.1} and \ref{p4.2}, it holds that 
\begin{eqnarray}
\left\Vert \boldsymbol{u}_{\varepsilon }\right\Vert _{L^{2}(\Omega
^{\varepsilon })^{3}} &\leq &C\varepsilon ^{\frac{5}{2}}\text{, }\left\Vert
\nabla \boldsymbol{u}_{\varepsilon }\right\Vert _{L^{2}(\Omega ^{\varepsilon
})^{3\times 3}}\leq C\varepsilon ^{\frac{3}{2}}\text{ and }\left\Vert
p_{\varepsilon }\right\Vert _{L^{2}(\Omega ^{\varepsilon })^{3}}\leq
C\varepsilon ^{\frac{1}{2}}  \label{4.30} \\
\text{for all }\varepsilon &>&0,  \notag
\end{eqnarray}%
where $C>0$ is independent of $\varepsilon >0$. Thus, given an ordinary
sequence $E$, there exist a subsequence $E^{\prime }$ of $E$ and $%
\boldsymbol{u}_{0}\in (L^{2}(\Omega ;\mathcal{B}_{\mathcal{A}}^{1,2}(\mathbb{%
R}^{2};H_{0}^{1}(I))))^{3}$, $p_{0}\in L^{2}(\Omega ;\mathcal{B}_{\mathcal{A}%
}^{2}(\mathbb{R}^{2};L^{2}(I)))$ such that, when $E^{\prime }\ni \varepsilon
\rightarrow 0$, 
\begin{equation}
\frac{\boldsymbol{u}_{\varepsilon }}{\varepsilon ^{2}}\rightarrow 
\boldsymbol{u}_{0}\text{ in }L^{2}(\Omega ^{\varepsilon })^{3}\text{-weak }%
\Sigma _{\mathcal{A}},  \label{4.31}
\end{equation}%
\begin{equation}
\frac{1}{\varepsilon }\nabla \boldsymbol{u}_{\varepsilon }\rightarrow 
\overline{\nabla }_{y}\boldsymbol{u}_{0}\text{ in }L^{2}(\Omega
^{\varepsilon })^{3\times 3}\text{-weak }\Sigma _{\mathcal{A}}  \label{4.32}
\end{equation}%
and 
\begin{equation}
p_{\varepsilon }\rightarrow p_{0}\text{ in }L^{2}(\Omega ^{\varepsilon })%
\text{-weak }\Sigma _{\mathcal{A}}.  \label{4.33}
\end{equation}%
Defining the function $\boldsymbol{u}$ as in (\ref{4.18}), we have that (\ref%
{4.19}) holds and $u_{3}=0$. The following is the first main result when $%
K_{\varepsilon }\gg \varepsilon ^{2}$.

\begin{theorem}
\label{t4.3}Let $A$ be an ergodic algebra with mean value on $\mathbb{R}^{2}$%
. Assume that \emph{(\textbf{A3})} holds and $K_{\varepsilon }\gg
\varepsilon ^{2}$. Let $(\boldsymbol{u}_{\varepsilon },p_{\varepsilon })$ be
a solution of \emph{(\ref{4.1})}. Then we have \emph{(\ref{4.31})-(\ref{4.33}%
)}. Furthermore there exists $q\in L^{2}(\Omega ;\mathcal{B}_{\mathcal{A}%
}^{2}(\mathbb{R}^{2};L^{2}(I)))$ such that $(\boldsymbol{u}_{0},p_{0},q)$
solves the system 
\begin{equation}
\left\{ 
\begin{array}{l}
-\overline{\func{div}}_{y}\left( A(y)\overline{\nabla }_{y}\boldsymbol{u}%
_{0}\right) +\overline{\nabla }_{y}q=\boldsymbol{f}_{1}-\nabla _{\overline{x}%
}p_{0}\text{ in }\Omega \times \mathbb{R}^{2}\times I, \\ 
\\ 
\overline{\func{div}}_{y}\boldsymbol{u}_{0}=0\text{ in }\Omega \times 
\mathbb{R}^{2}\times I, \\ 
\\ 
\func{div}_{\overline{x}}\left( \int_{-1}^{1}M(\boldsymbol{u}_{0}(\overline{x%
},\cdot ,y_{3}))dy_{3}\right) =0\text{ in }\Omega , \\ 
\\ 
\left( \int_{-1}^{1}M(\boldsymbol{u}_{0}(\overline{x},\cdot
,y_{3}))dy_{3}\right) \cdot \nu =0\text{ on }\partial \Omega \text{.}%
\end{array}%
\right.  \label{4.34}
\end{equation}
\end{theorem}

\begin{proof}
We proceed as in the proof of Theorem \ref{t4.1}. Firstly, we see that $%
p_{0} $ is independent of $y$. Next, choosing in (\ref{4.21}) a test
function $\boldsymbol{\varphi }$ such that $\func{div}_{y}\boldsymbol{%
\varphi }=0$ and then multiplying the resulting equality by $1/\varepsilon $%
, we appeal once again to (\ref{4.31})-(\ref{4.33}) to obtain 
\begin{equation}
\left\{ 
\begin{array}{l}
\int_{\Omega }\int_{-1}^{1}M(A\overline{\nabla }_{y}\boldsymbol{u}%
_{0}:\nabla _{y}\boldsymbol{\varphi })dy_{3}d\overline{x}-\int_{\Omega
}\int_{-1}^{1}p_{0}(\overline{x})M(\func{div}_{\overline{x}}\boldsymbol{%
\varphi })dy_{3}d\overline{x} \\ 
\\ 
\ \ \ =\int_{\Omega }\int_{-1}^{1}M(\boldsymbol{f}\boldsymbol{\varphi }%
)dy_{3}d\overline{x}.%
\end{array}%
\right.  \label{4.35}
\end{equation}%
This yields as before, the existence of $q$ such that (\ref{4.34}) holds.
\end{proof}

As in the previous subsection, let us consider the following Stokes system: 
\begin{equation}
\left\{ 
\begin{array}{l}
\text{Find }\boldsymbol{w}_{j}\in B_{\#\mathcal{A}}^{1,2}(\mathbb{R}%
^{2};H_{0}^{1}(I))^{3}\text{ such that } \\ 
-\overline{\func{div}}_{y}(A\overline{\nabla }_{y}\boldsymbol{w}_{j})+%
\overline{\nabla }_{y}\pi _{j}=e_{j}\text{ in }\mathbb{R}^{2}\times I, \\ 
\overline{\func{div}}_{y}\boldsymbol{w}_{j}=0\text{ in }\mathbb{R}^{2}\times
I%
\end{array}%
\right. \ \ \ \ \ \ \ \ \ \ \ \ \ \ \ \ \ \ \ \ \ \ \ \ \ \ \ \ \ \ 
\label{4.36}
\end{equation}%
Then in view of \cite[Lemma 4.1]{CJW2024}, system (\ref{4.36}) possesses a
unique solution $\boldsymbol{w}_{j}\in B_{\#\mathcal{A}}^{1,2}(\mathbb{R}%
^{2};H_{0}^{1}(I))^{3}$. Also, proceeding as in Subsection \ref{subsec4.2},
we are able to see that, defining $\boldsymbol{u}$ by (\ref{4.18}) and
setting 
\begin{equation}
\widehat{A}=(\widehat{a}_{ij})_{1\leq i,j\leq 2}\text{ with }\widehat{a}%
_{ij}=\int_{-1}^{1}M(A\overline{\nabla }_{y}\boldsymbol{w}_{i}:\overline{%
\nabla }_{y}\boldsymbol{w}_{j})dy_{3},\ 1\leq i,j\leq 2,  \label{e4.38}
\end{equation}
($\widehat{A}=(\widehat{a}_{ij})_{1\leq i,j\leq 2}$ is a symmetric positive
definite $2\times 2$ matrix), we have $u_{3}=0$ and 
\begin{equation}
\left\{ 
\begin{array}{l}
\boldsymbol{u}^{\prime }=\widehat{A}(\boldsymbol{f}_{1}-\nabla _{\overline{x}%
}p_{0})\text{ in }\Omega \\ 
\func{div}_{\overline{x}}\boldsymbol{u}^{\prime }=0\text{ in }\Omega \text{
and }\boldsymbol{u}^{\prime }\cdot \nu =0\text{ on }\partial \Omega .%
\end{array}%
\right.  \label{e4.37}
\end{equation}%
The homogenized result in this case is therefore stated as follows.

\begin{theorem}
\label{t4.6}Under the assumptions of Theorem \emph{\ref{t4.3}}, the sequence 
$(\boldsymbol{u}_{\varepsilon }/\varepsilon ^{2},p_{\varepsilon
})_{\varepsilon >0}$ weakly $\Sigma _{\mathcal{A}}$-converges in $%
L^{2}(\Omega ^{\varepsilon })^{3}\times L_{0}^{2}(\Omega ^{\varepsilon })$
towards $(\boldsymbol{u}_{0},p_{0})$ determined by \emph{(\ref{4.31})-(\ref%
{4.33})}. In addition $p_{0}\in H^{1}(\Omega )$ and, defining $\boldsymbol{u}%
=(\boldsymbol{u}^{\prime },u_{3})$ by \emph{(\ref{4.18})}, one has $u_{3}=0$
and $(\boldsymbol{u}^{\prime },p_{0})$ is the unique solution of the
effective problem \emph{(\ref{e4.37})} where $\widehat{A}$ is given by \emph{%
(\ref{e4.38})}.
\end{theorem}

\subsection{Homogenization results: case when $K_{\protect\varepsilon }\ll 
\protect\varepsilon ^{2}$}

Throughout this subsection, we assume that $\varepsilon ^{2}/K_{\varepsilon
}\rightarrow 0$ as $\varepsilon \rightarrow 0$. In that case the following
estimates hold: 
\begin{eqnarray}
\left\Vert \boldsymbol{u}_{\varepsilon }\right\Vert _{L^{2}(\Omega
^{\varepsilon })^{3}} &\leq &C\varepsilon ^{\frac{3}{2}}K_{\varepsilon }^{%
\frac{1}{2}}\text{, }\left\Vert \nabla \boldsymbol{u}_{\varepsilon
}\right\Vert _{L^{2}(\Omega ^{\varepsilon })^{3\times 3}}\leq C\varepsilon ^{%
\frac{3}{2}}\text{, }\left\Vert p_{\varepsilon }\right\Vert _{L^{2}(\Omega
^{\varepsilon })^{3}}\leq C\frac{\varepsilon ^{\frac{3}{2}}}{K_{\varepsilon
}^{\frac{1}{2}}},  \label{4.43} \\
\text{for all }\varepsilon &>&0,  \notag
\end{eqnarray}%
$C$ being a positive constant independent of $\varepsilon $.

In view of (\ref{4.43}), the following result holds.

\begin{lemma}
\label{l4.2}Given an ordinary sequence $E$, there exists a subsequence $%
E^{\prime }$ of $E$ and functions $\boldsymbol{u}_{0}\in (L^{2}(\Omega ;%
\mathcal{B}_{\mathcal{A}}^{2}(\mathbb{R}^{2};L^{2}(I))))^{3}$, $\boldsymbol{u%
}_{1}\in (L^{2}(\Omega ;B_{\#\mathcal{A}}^{1,2}(\mathbb{R}%
^{2};H^{1}(I))))^{3}$ and $p_{0}\in L^{2}(\Omega ;\mathcal{B}_{\mathcal{A}%
}^{2}(\mathbb{R}^{2};L^{2}(I)))$ such that, as $E^{\prime }\ni \varepsilon
\rightarrow 0$, 
\begin{equation}
\frac{\boldsymbol{u}_{\varepsilon }}{\varepsilon K_{\varepsilon }^{\frac{1}{2%
}}}\rightarrow \boldsymbol{u}_{0}\text{ in }L^{2}(\Omega ^{\varepsilon })^{3}%
\text{-weak }\Sigma _{\mathcal{A}}\ \ \ \ \ \ \   \label{4.44}
\end{equation}%
\begin{equation}
\frac{\nabla \boldsymbol{u}_{\varepsilon }}{\varepsilon }\rightarrow \nabla
_{y}\boldsymbol{u}_{1}\text{ in }L^{2}(\Omega ^{\varepsilon })^{3\times 3}%
\text{-weak }\Sigma _{\mathcal{A}}\ \ \ \ \ \ \   \label{4.45}
\end{equation}%
\begin{equation}
\frac{K_{\varepsilon }^{\frac{1}{2}}}{\varepsilon }p_{\varepsilon
}\rightarrow p_{0}\text{ in }L^{2}(\Omega ^{\varepsilon })\text{-weak }%
\Sigma _{\mathcal{A}}.\ \ \ \ \ \ \   \label{4.46}
\end{equation}%
Moreover 
\begin{equation}
\overline{\func{div}}_{y}\boldsymbol{u}_{0}=0\text{ and }\func{div}_{y}%
\boldsymbol{u}_{1}=0\text{ in }\Omega \times \mathbb{R}^{2}\times I.
\label{4.46'}
\end{equation}
\end{lemma}

\begin{proof}
Applying Theorem \ref{t2.1}, we derive the existence of a subsequence $%
E^{\prime }$ of $E$ and of a triple $(\boldsymbol{u}_{0},\boldsymbol{v}%
,p_{0})\in (L^{2}(\Omega ;\mathcal{B}_{\mathcal{A}}^{1,2}(\mathbb{R}%
^{2};H^{1}(I))))^{3}\times (L^{2}(\Omega ;\mathcal{B}_{\mathcal{A}}^{1,2}(%
\mathbb{R}^{2};H^{1}(I))))^{3\times 3}\times L^{2}(\Omega ;\mathcal{B}_{%
\mathcal{A}}^{2}(\mathbb{R}^{2};L^{2}(I)))$ such that (\ref{4.44}) and (\ref%
{4.46}) hold, and further 
\begin{equation}
\frac{\nabla \boldsymbol{u}_{\varepsilon }}{\varepsilon }\rightarrow 
\boldsymbol{v}\text{ in }L^{2}(\Omega ^{\varepsilon })^{3\times 3}\text{%
-weak }\Sigma _{\mathcal{A}}\text{. \ \ \ \ \ \ \ \ \ \ \ \ \ \ }
\label{4.45'}
\end{equation}%
Let us characterize $\boldsymbol{v}$. To that end, let $\Phi \in (\mathcal{C}%
_{0}^{\infty }(\Omega )\otimes \mathcal{A}^{\infty }(\mathbb{R}^{2};\mathcal{%
C}_{0}^{\infty }(I)))^{3}$ be such that $\func{div}_{y}\Phi =0$. Then 
\begin{eqnarray}
\frac{1}{\varepsilon }\int_{\Omega ^{\varepsilon }}\frac{1}{\varepsilon }%
\nabla \boldsymbol{u}_{\varepsilon }\cdot \Phi ^{\varepsilon }dx &=&-\frac{1%
}{\varepsilon }\int_{\Omega ^{\varepsilon }}\frac{1}{\varepsilon }%
\boldsymbol{u}_{\varepsilon }(\func{div}_{\overline{x}}\Phi )^{\varepsilon
}dx  \label{4.47} \\
&=&-K_{\varepsilon }^{\frac{1}{2}}\frac{1}{\varepsilon }\int_{\Omega
^{\varepsilon }}\frac{\boldsymbol{u}_{\varepsilon }}{\varepsilon
K_{\varepsilon }^{\frac{1}{2}}}(\func{div}_{\overline{x}}\Phi )^{\varepsilon
}dx  \notag
\end{eqnarray}%
where in (\ref{4.47}), $\Phi ^{\varepsilon }(x)=\Phi (\overline{x}%
,x/\varepsilon )$ for $x\in \Omega ^{\varepsilon }$ and $\nabla \boldsymbol{u%
}_{\varepsilon }\cdot \Phi ^{\varepsilon }$ is viewed as the vector $(\nabla
u_{\varepsilon ,i}\cdot \Phi ^{\varepsilon })_{1\leq i\leq 3}$ and $%
\boldsymbol{u}_{\varepsilon }(\func{div}_{\overline{x}}\Phi )^{\varepsilon }$
is viewed as the vector $(u_{\varepsilon ,i}(\func{div}_{\overline{x}}\Phi
)^{\varepsilon })_{1\leq i\leq 3}$, with $\boldsymbol{u}_{\varepsilon
}=(u_{\varepsilon ,i})_{1\leq i\leq 3}$. Letting $E^{\prime }\ni \varepsilon
\rightarrow 0$ in (\ref{4.47}) using (\ref{4.44}) and (\ref{4.45'}) together
with the fact that $K_{\varepsilon }\rightarrow 0$, we are led to 
\begin{equation*}
\int_{\Omega }\int_{-1}^{1}M(\boldsymbol{v}(\overline{x},\cdot ,y_{3})\cdot
\Phi (\overline{x},\cdot ,y_{3}))dy_{3}d\overline{x}=0\ \ \ \ \ \ \ \ \ \ \
\ \ \ \ \ \ \ \ \ \ \ \ \ \ 
\end{equation*}%
for all $\Phi \in (\mathcal{C}_{0}^{\infty }(\Omega )\otimes \mathcal{A}%
^{\infty }(\mathbb{R}^{2};\mathcal{C}_{0}^{\infty }(I)))^{3}$ with $\func{div%
}_{y}\Phi =0$. We infer from Corollary \ref{c2.1} the existence of $%
\boldsymbol{u}_{1}\in (L^{2}(\Omega ;B_{\#\mathcal{A}}^{1,2}(\mathbb{R}%
^{2};H^{1}(I))))^{3}$ such that $\boldsymbol{v}=\nabla _{y}\boldsymbol{u}%
_{1} $. Next, from the equality $\func{div}\boldsymbol{u}_{\varepsilon }=0$,
we deduce 
\begin{equation*}
\sum_{i=1}^{3}\frac{1}{\varepsilon }\int_{\Omega ^{\varepsilon }}\frac{1}{%
\varepsilon }\frac{\partial u_{\varepsilon ,i}}{\partial x_{i}}(x)\psi (%
\overline{x},\frac{x}{\varepsilon })dx=0\text{ for all }\psi \in \mathcal{C}%
_{0}^{\infty }(\Omega )\otimes \mathcal{A}^{\infty }(\mathbb{R}^{2};\mathcal{%
C}_{0}^{\infty }(I)),
\end{equation*}%
so that, in view of (\ref{4.45}), we have, as $E^{\prime }\ni \varepsilon
\rightarrow 0$, 
\begin{equation*}
\sum_{i=1}^{3}\int_{\Omega }\int_{-1}^{1}M\left( \frac{\partial u_{1,i}}{%
\partial y_{i}}\psi \right) dy_{3}d\overline{x}=0\text{ for all }\psi \in 
\mathcal{C}_{0}^{\infty }(\Omega )\otimes \mathcal{A}^{\infty }(\mathbb{R}%
^{2};\mathcal{C}_{0}^{\infty }(I)).
\end{equation*}%
This amounts to $\func{div}_{y}\boldsymbol{u}_{1}=0$ in $\Omega \times 
\mathbb{R}^{2}\times I$. Also, the same equality $\func{div}\boldsymbol{u}%
_{\varepsilon }=0$ yields 
\begin{equation}
\int_{\Omega ^{\varepsilon }}\boldsymbol{u}_{\varepsilon }\left( (\nabla _{%
\overline{x}}\psi )^{\varepsilon }+\frac{1}{\varepsilon }(\nabla _{y}\psi
)^{\varepsilon }\right) dx=0\text{ for }\psi \text{ as above.}  \label{4.47'}
\end{equation}%
Dividing both members of \ (\ref{4.47'}) by $\varepsilon K_{\varepsilon
}^{1/2}$ and letting $E^{\prime }\ni \varepsilon \rightarrow 0$ gives 
\begin{equation*}
\int_{\Omega }\int_{-1}^{1}M(\boldsymbol{u}_{0}\cdot \nabla _{y}\psi )dy_{3}d%
\overline{x}=0,\ \ \ \ \ \ \ \ \ \ \ \ \ \ \ \ \ \ \ \ 
\end{equation*}%
which amounts to $\overline{\func{div}}_{y}\boldsymbol{u}_{0}=0$ in $\Omega
\times \mathbb{R}^{2}\times I$. This concludes the proof.
\end{proof}

Now we go back to the variational form (\ref{4.21}), i.e. 
\begin{equation}
\left\{ 
\begin{array}{l}
\int_{\Omega ^{\varepsilon }}A\left( \frac{x}{\varepsilon }\right) \nabla 
\boldsymbol{u}_{\varepsilon }\cdot \left[ (\nabla _{\overline{x}}\boldsymbol{%
\varphi })^{\varepsilon }+\frac{1}{\varepsilon }(\nabla _{y}\boldsymbol{%
\varphi })^{\varepsilon }\right] dx+\frac{\mu }{K_{\varepsilon }}%
\int_{\Omega ^{\varepsilon }}\boldsymbol{u}_{\varepsilon }\boldsymbol{%
\varphi }^{\varepsilon }dx \\ 
\\ 
+\frac{\rho }{\phi ^{2}}\int_{\Omega ^{\varepsilon }}(\boldsymbol{u}%
_{\varepsilon }\cdot \nabla )\boldsymbol{u}_{\varepsilon }\boldsymbol{%
\varphi }^{\varepsilon }dx-\int_{\Omega ^{\varepsilon }}p_{\varepsilon }%
\left[ (\func{div}_{\overline{x}}\boldsymbol{\varphi })^{\varepsilon }+\frac{%
1}{\varepsilon }(\func{div}_{y}\boldsymbol{\varphi })^{\varepsilon }\right]
dx \\ 
\\ 
=\int_{\Omega ^{\varepsilon }}\boldsymbol{f}\boldsymbol{\varphi }%
^{\varepsilon }dx.%
\end{array}%
\right.  \label{4.48}
\end{equation}%
We assume that (\ref{4.44}), (\ref{4.45}) and (\ref{4.46}) hold. Since $%
K_{\varepsilon }/\varepsilon ^{2}\rightarrow 0$ when $\varepsilon
\rightarrow 0$, we use (\ref{4.44}) to see that, as $E^{\prime }\ni
\varepsilon \rightarrow 0$, 
\begin{equation*}
\frac{K_{\varepsilon }}{\varepsilon ^{2}}\frac{\mu }{K_{\varepsilon }}%
\int_{\Omega ^{\varepsilon }}\boldsymbol{u}_{\varepsilon }\boldsymbol{%
\varphi }^{\varepsilon }dx=\frac{\mu }{\varepsilon ^{2}}\int_{\Omega
^{\varepsilon }}\boldsymbol{u}_{\varepsilon }\boldsymbol{\varphi }%
^{\varepsilon }dx=\mu \left( \frac{K_{\varepsilon }}{\varepsilon ^{2}}%
\right) ^{\frac{1}{2}}\int_{\Omega ^{\varepsilon }}\frac{\boldsymbol{u}%
_{\varepsilon }}{\varepsilon K_{\varepsilon }^{\frac{1}{2}}}\boldsymbol{%
\varphi }^{\varepsilon }dx\rightarrow 0.
\end{equation*}%
Thus, if we multiply (\ref{4.48}) by $K_{\varepsilon }^{1/2}/\varepsilon $
and we let $E^{\prime }\ni \varepsilon \rightarrow 0$ in the resulting
identity, taking into account (\ref{4.44}), (\ref{4.45}) and (\ref{4.46})
together with the fact that $K_{\varepsilon }\rightarrow 0$ as $\varepsilon
\rightarrow 0$, we get 
\begin{equation*}
\int_{\Omega }\int_{-1}^{1}M(p_{0}\func{div}_{y}\boldsymbol{\varphi })dy_{3}d%
\overline{x}=0,\ \ \ \ \ \ \ \ \ \ \ \ \ \ \ \ \ \ 
\end{equation*}%
thereby showing that $p_{0}$ does not depend on $y$. Next, choosing in (\ref%
{4.48}) a test function $\boldsymbol{\varphi }$ such that $\func{div}_{y}%
\boldsymbol{\varphi }=0$ and then multiplying the resulting equality by $%
K_{\varepsilon }^{1/2}/\varepsilon ^{2}$, we appeal once again to (\ref{4.44}%
)-(\ref{4.46}) to obtain 
\begin{equation}
\mu \int_{\Omega }\int_{-1}^{1}M(\boldsymbol{u}_{0}\boldsymbol{\varphi }%
)dy_{3}d\overline{x}-\int_{\Omega }\int_{-1}^{1}p_{0}(\overline{x})M(\func{%
div}_{\overline{x}}\boldsymbol{\varphi })dy_{3}d\overline{x}=0.
\label{e4.49}
\end{equation}%
This yields as before, the existence of $q\in L^{2}(\Omega ;\mathcal{B}_{%
\mathcal{A}}^{2}(\mathbb{R}^{2};L^{2}(I)))$ such that $(\boldsymbol{u}%
_{0},p_{0},q)$ solves the system 
\begin{equation}
\left\{ 
\begin{array}{l}
\mu \boldsymbol{u}_{0}+\overline{\nabla }_{y}q=-\nabla _{\overline{x}}p_{0}%
\text{ in }\Omega \times \mathbb{R}^{2}\times I \\ 
\overline{\func{div}}_{y}\boldsymbol{u}_{0}=0\text{ in }\Omega \times 
\mathbb{R}^{2}\times I \\ 
\func{div}_{\overline{x}}\left( \int_{-1}^{1}M(\boldsymbol{u}_{0}(\overline{x%
},\cdot ,y_{3}))dy_{3}\right) =0\text{ in }\Omega \\ 
\left( \int_{-1}^{1}M(\boldsymbol{u}_{0}(\overline{x},\cdot
,y_{3}))dy_{3}\right) \cdot \nu =0\text{ on }\partial \Omega \text{.\ }%
\end{array}%
\right. \ \ \ \ \ \ \ \ \ \ \ \ \ \ \ \ \ \ \ \ \   \label{4.49}
\end{equation}

We have therefore proved the following result.

\begin{theorem}
\label{t4.5}Let $\mathcal{A}$ be an ergodic algebra with mean value on $%
\mathbb{R}^{2}$. Assume that \emph{(\textbf{A3})} holds and $K_{\varepsilon
}\ll \varepsilon ^{2}$. Let $(\boldsymbol{u}_{\varepsilon },p_{\varepsilon
}) $ be a solution of \emph{(\ref{4.1})}. Then we have \emph{(\ref{4.44})-(%
\ref{4.46})}. Furthermore there exists $q\in L^{2}(\Omega ;\mathcal{B}_{%
\mathcal{A}}^{2}(\mathbb{R}^{2};L^{2}(I)))$ such that $(\boldsymbol{u}%
_{0},p_{0},q)$ solves the system \emph{(\ref{4.49})}.
\end{theorem}

In order to derive the effective equation in the case when $K_{\varepsilon
}<<\varepsilon ^{2}$, we consider the local problems: For $1\leq i\leq 3$,
find $\boldsymbol{w}_{i}\in \mathcal{B}_{\mathcal{A}}^{2}(\mathbb{R}%
^{2};L^{2}(I))^{3}$ such that 
\begin{equation}
\mu \boldsymbol{w}_{i}+\overline{\nabla }_{y}\pi _{i}=e_{i}\text{ in }%
\mathbb{R}^{2}\times I\text{ and }\overline{\func{div}}_{y}\boldsymbol{w}%
_{i}=0\text{ in }\mathbb{R}^{2}\times I.  \label{e4.36}
\end{equation}%
Then the existence and uniqueness of $\boldsymbol{w}_{i}$ satisfying (\ref%
{e4.36}) follows a two steps process described below:

\begin{itemize}
\item[1)] We approximate (\ref{e4.36}) as follows: for any integer $n\geq 1$%
, there exists a unique $\boldsymbol{w}_{i,n}\in \mathcal{B}_{\func{div}%
}^{1,2}$ solution of 
\begin{equation}
\left\{ 
\begin{array}{l}
-\frac{1}{n^{2}}\overline{\Delta }_{y}\boldsymbol{w}_{i,n}+\mu \boldsymbol{w}%
_{i,n}+\overline{\nabla }_{y}\pi _{i,n}=e_{i}\text{ in }\mathbb{R}^{2}\times
I \\ 
\overline{\func{div}}_{y}\boldsymbol{w}_{i,n}=0\text{ in }\mathbb{R}%
^{2}\times I%
\end{array}%
\right.  \label{4.37}
\end{equation}%
in the sense of (\ref{4.27'}), where $\overline{\Delta }_{y}=\overline{\func{%
div}}_{y}(\overline{\nabla }_{y})$. Testing (\ref{4.37}) with $\boldsymbol{w}%
_{i,n}$ yields the estimates 
\begin{equation}
\sup_{n\geq 1}\left( \frac{1}{n}\left\Vert \overline{\nabla }_{y}\boldsymbol{%
w}_{i,n}\right\Vert _{\mathcal{B}_{\mathcal{A}}^{2}(\mathbb{R}%
^{2};L^{2}(I))^{3\times 3}}+\left\Vert \boldsymbol{w}_{i,n}\right\Vert _{%
\mathcal{B}_{\mathcal{A}}^{2}(\mathbb{R}^{2};L^{2}(I))^{3}}\right) \leq C%
\text{,}  \label{4.38}
\end{equation}%
where $C>0$ is independent of $n$.

\item[2)] We pass to the limit in the variational formulation 
\begin{equation*}
\begin{array}{l}
\frac{1}{n^{2}}\int_{-1}^{1}M(\overline{\nabla }_{y}\boldsymbol{w}%
_{i,n}\cdot \overline{\nabla }_{y}\boldsymbol{v})dy_{3}+\mu \int_{-1}^{1}M(%
\boldsymbol{w}_{i,n}\boldsymbol{v})dy_{3}=\int_{-1}^{1}M(\boldsymbol{v}%
)e_{i}dy_{3} \\ 
\text{for all }\boldsymbol{v}\in \mathcal{B}_{\func{div}}^{1,2}%
\end{array}%
\end{equation*}%
using (\ref{4.38}) and the reflexivity of the Hilbert space $\mathcal{B}_{%
\mathcal{A}}^{2}(\mathbb{R}^{2};L^{2}(I))$ to derive the existence and
uniqueness of a $\boldsymbol{w}_{i}\in (\mathcal{B}_{\mathcal{A}}^{2}(%
\mathbb{R}^{2};L^{2}(I)))^{3}$ solution to (\ref{e4.36}).
\end{itemize}

Now, proceeding as in the proof of Theorem \ref{t4.2}, we may choose in (\ref%
{4.49}) the test function $\boldsymbol{\varphi }=\psi \otimes \boldsymbol{w}%
_{i}$; then 
\begin{equation}
\mu \int_{-1}^{1}M(\boldsymbol{u}_{0}\boldsymbol{w}_{i})dy_{3}+\nabla _{%
\overline{x}}p_{0}\int_{-1}^{1}M(\boldsymbol{w}_{i})dy_{3}=0.  \label{4.39}
\end{equation}%
Next, take $\boldsymbol{u}_{0}(\overline{x},\cdot )$ as test function in (%
\ref{e4.36}): 
\begin{equation}
\mu \int_{-1}^{1}M(\boldsymbol{u}_{0}\boldsymbol{w}_{i})dy_{3}=%
\int_{-1}^{1}M(\boldsymbol{u}_{0})e_{i}dy_{3}\equiv u_{i}(\overline{x}).
\label{4.40}
\end{equation}%
Putting together (\ref{4.39}) and (\ref{4.40}) yields 
\begin{equation*}
u_{i}(\overline{x})=-\nabla _{\overline{x}}p_{0}\int_{-1}^{1}M(\boldsymbol{w}%
_{i})dy_{3}=-\sum_{j=1}^{3}\frac{\partial p_{0}}{\partial \overline{x}_{j}}%
\int_{-1}^{1}M(\boldsymbol{w}_{i})e_{j}dy_{3}.
\end{equation*}%
Still in (\ref{e4.36}) with $\boldsymbol{w}_{j}$ taken as test function, we
are led to 
\begin{equation*}
\int_{-1}^{1}M(\boldsymbol{w}_{i})e_{j}dy_{3}=\mu \int_{-1}^{1}M(\boldsymbol{%
w}_{i}\boldsymbol{w}_{j})dy_{3}.
\end{equation*}%
So, setting 
\begin{equation}
\widehat{A}=(\widehat{a}_{ij})_{1\leq i,j\leq 2}\text{ with }\widehat{a}%
_{ij}=\mu \int_{-1}^{1}M(\boldsymbol{w}_{i}\boldsymbol{w}_{j})dy_{3},
\label{4.41}
\end{equation}%
we obtain a $2\times 2$ symmetric matrix. We recall that the fact that $%
u_{d}=0$ amounts to $\widehat{a}_{id}=0$, and so $\widehat{a}_{di}=0$ for $%
1\leq i\leq 2$.

We have thus proved that the following result, which besides is the main
result of the work in the case when $K_{\varepsilon }\ll \varepsilon ^{2}$.

\begin{theorem}
\label{t4.4}Under the assumptions of Theorem \emph{\ref{t4.5}}, the sequence 
$(\frac{\boldsymbol{u}_{\varepsilon }}{\varepsilon K_{\varepsilon }^{1/2}},%
\frac{K_{\varepsilon }^{1/2}}{\varepsilon ^{2}}p_{\varepsilon
})_{\varepsilon >0}$ weakly $\Sigma _{\mathcal{A}}$-converges in $%
L^{2}(\Omega ^{\varepsilon })^{3}\times L_{0}^{2}(\Omega ^{\varepsilon })$
towards $(\boldsymbol{u}_{0},p_{0})$ determined by \emph{(\ref{4.44}) and (%
\ref{4.46})}. Moreover $p_{0}\in H^{1}(\Omega )$ and, setting 
\begin{eqnarray*}
\boldsymbol{u}(\overline{x}) &=&\int_{-1}^{1}M(\boldsymbol{u}_{0}(\overline{x%
},\cdot ,y_{3})dy_{3}\text{\ \ }(\overline{x}\in \Omega ) \\
&=&(u_{i}(\overline{x}))_{1\leq i\leq 3}\text{ and }\boldsymbol{u}^{\prime
}=(u_{i})_{1\leq i\leq 2,}
\end{eqnarray*}%
we have $u_{3}=0$ and $(\boldsymbol{u}^{\prime },p_{0})$ is a solution of
the homogenized problem 
\begin{equation}
\boldsymbol{u}^{\prime }=-\widehat{A}\nabla _{\overline{x}}p_{0}\text{ in }%
\Omega \text{, }\func{div}_{\overline{x}}\boldsymbol{u}^{\prime }=0\text{ in 
}\Omega \text{ and }\boldsymbol{u}^{\prime }\cdot \nu =0\text{ on }\partial
\Omega ,  \label{4.42}
\end{equation}%
where $\widehat{A}$ is defined by \emph{(\ref{4.41})}, $\boldsymbol{w}_{i}$ $%
(1\leq i\leq 2)$ being the unique solution of \emph{(\ref{e4.36})}.
\end{theorem}

\section{Homogenization of the Darcy-Lapwood-Brinkmann equation in thin
heterogeneous domain with highly oscillating boundaries\label{sec5}}

\subsection{Introduction}

In this section we deal with problem (\ref{4.1}), but this time, stated in
the following domain 
\begin{equation}
\Omega ^{\varepsilon }=\left\{ x=(\overline{x},x_{3})\in \Omega \times 
\mathbb{R}:\varepsilon h_{1}\left( \frac{\overline{x}}{\varepsilon }\right)
<x_{3}<\varepsilon h_{2}\left( \frac{\overline{x}}{\varepsilon }\right)
\right\} ,  \label{5.0}
\end{equation}%
where $\varepsilon >0$ is a small parameter, $\Omega \subset \mathbb{R}^{2}$
is a bounded open Lipschitz domain, $h_{1}$, $h_{2}\in W^{1,\infty }(\mathbb{%
R}^{2})$ are two bounded Lipschitz continuous functions on $\mathbb{R}^{2}$
and satisfying 
\begin{equation}
\max_{\mathbb{R}^{2}}h_{1}<\min_{\mathbb{R}^{2}}h_{2};\ \ \ \ \ \ \ \ \ \ \
\ \ \ \ \ \ \ \ \ \ \ \ \ \ \ \ \ \ \ \ \ \ \ \ \ \ \ \ \ \ \   \label{5.1}
\end{equation}%
\begin{equation}
h_{1},h_{2}\in \mathcal{A}\text{ with }M(h_{2}-h_{1})\neq 0\text{.\ \ \ \ \
\ \ \ \ \ \ \ \ \ }  \label{5.2}
\end{equation}%
We set 
\begin{equation}
h_{1}^{-}=\min_{\mathbb{R}^{2}}h_{1}\text{, }h_{2}^{+}=\max_{\mathbb{R}%
^{2}}h_{2}\text{ and }I=(h_{1}^{-},h_{2}^{+}),  \label{5.2'}
\end{equation}%
and we define $G_{\varepsilon }=\Omega \times (\varepsilon
h_{1}^{-},\varepsilon h_{2}^{+})$. Then, as seen in Section \ref{sec3}, $%
G_{\varepsilon }$ has flat lateral boundaries $y_{3}=\varepsilon
h_{1}^{-},\varepsilon h_{2}^{+}$, with $\Omega ^{\varepsilon }\subset
G_{\varepsilon }$. When $\varepsilon \rightarrow 0$, $G_{\varepsilon }$
shrinks to $G_{0}=\Omega \times \left\{ 0\right\} $ which can be identified
to $\Omega $ through the identification $\overline{x}\equiv (\overline{x},0)$%
. This will be the case in the sequel.

The problem to be investigated here is stated by (\ref{4.1}) in the thin
domain $\Omega ^{\varepsilon }$ given above by (\ref{5.0}). As in Section %
\ref{sec4}, the corresponding problem possesses at least a weak solution $%
\boldsymbol{u}_{\varepsilon }\in H_{0}^{1}(\Omega ^{\varepsilon })^{3}$.
Moreover, to each $\boldsymbol{u}_{\varepsilon }$ is associated a unique $%
p_{\varepsilon }\in L_{0}^{2}(\Omega ^{\varepsilon })$ such that (\ref{4.1})
is satisfied by the couple $(\boldsymbol{u}_{\varepsilon },p_{\varepsilon })$%
. We extend $\boldsymbol{u}_{\varepsilon }$ to $G_{\varepsilon }$ by zero
off $\Omega ^{\varepsilon }$, and we still denote this extension by $%
\boldsymbol{u}_{\varepsilon }$; this is fully justified by the fact that $%
\boldsymbol{u}_{\varepsilon }=0$ on $\partial \Omega ^{\varepsilon }$. It is
therefore easily seen that, proceeding exactly as in Subsection \ref%
{subsec4.1}, we derive the existence of a positive constant $C$ independent
of $\varepsilon $ such that 
\begin{equation}
\left\Vert \boldsymbol{u}_{\varepsilon }\right\Vert _{L^{2}(\Omega
^{\varepsilon })^{3}}\leq C\min (\varepsilon ^{\frac{5}{2}},\varepsilon ^{%
\frac{3}{2}}K_{\varepsilon }^{\frac{1}{2}})\text{ and }\left\Vert \nabla 
\boldsymbol{u}_{\varepsilon }\right\Vert _{L^{2}(\Omega ^{\varepsilon
})^{3\times 3}}\leq C\varepsilon ^{\frac{3}{2}},  \label{5.4'}
\end{equation}%
and so, 
\begin{equation}
\left\Vert \boldsymbol{u}_{\varepsilon }\right\Vert _{L^{2}(G_{\varepsilon
})^{3}}\leq C\min (\varepsilon ^{\frac{5}{2}},\varepsilon ^{\frac{3}{2}%
}K_{\varepsilon }^{\frac{1}{2}})\text{ and }\left\Vert \nabla \boldsymbol{u}%
_{\varepsilon }\right\Vert _{L^{2}(G_{\varepsilon })^{3\times 3}}\leq
C\varepsilon ^{\frac{3}{2}}.  \label{5.4}
\end{equation}%
Concerning the pressure, since the domain $\Omega ^{\varepsilon }$ satisfies
the assumptions of \cite[Corollary 3.4]{Casado2020} (in fact, $\Omega $ is
Lipschitz and connected), there exist functions $p_{\varepsilon }^{0}\in
H^{1}(\Omega )$ and $p_{\varepsilon }^{1}\in L^{2}(\Omega ^{\varepsilon })$
such that 
\begin{equation}
p_{\varepsilon }=p_{\varepsilon }^{0}+\varepsilon p_{\varepsilon }^{1}\text{
in }\Omega ^{\varepsilon },\ \ \ \ \ \ \ \ \ \ \ \ \ \ \ \ \ \ \ \ \ \ \ \ \
\ \ \ \ \ \ \   \label{5.5}
\end{equation}%
and 
\begin{equation}
\varepsilon ^{\frac{3}{2}}\left\Vert p_{\varepsilon }^{0}\right\Vert
_{H^{1}(\Omega )}+\varepsilon \left\Vert p_{\varepsilon }^{1}\right\Vert
_{L^{2}(\Omega ^{\varepsilon })}\leq C\left\Vert \nabla p_{\varepsilon
}\right\Vert _{H^{-1}(\Omega ^{\varepsilon })^{3}},  \label{5.6}
\end{equation}%
where the positive constant $C$ in (\ref{5.6}) is independent of $%
\varepsilon $. It remains to find a bound for $\left\Vert \nabla
p_{\varepsilon }\right\Vert _{H^{-1}(\Omega ^{\varepsilon })^{3}}$. For that
purpose, let $\boldsymbol{v}\in H_{0}^{1}(\Omega ^{\varepsilon })^{3}$; then
appealing to (\ref{4.1})$_{1}$, we have 
\begin{eqnarray}
\left\langle \nabla p_{\varepsilon },\boldsymbol{v}\right\rangle
&=&\int_{\Omega ^{\varepsilon }}\boldsymbol{f}\cdot \boldsymbol{v}%
dx-\int_{\Omega ^{\varepsilon }}A^{\varepsilon }\nabla \boldsymbol{u}%
_{\varepsilon }\cdot \nabla \boldsymbol{v}dx-\frac{\rho }{\phi ^{2}}%
\int_{\Omega ^{\varepsilon }}(\boldsymbol{u}_{\varepsilon }\cdot \nabla )%
\boldsymbol{u}_{\varepsilon }\cdot \boldsymbol{v}dx  \label{5.7} \\
&&-\frac{\mu }{K_{\varepsilon }}\int_{\Omega ^{\varepsilon }}\boldsymbol{u}%
_{\varepsilon }\cdot \boldsymbol{v}dx.  \notag
\end{eqnarray}%
It is easy to see that 
\begin{equation*}
\left\vert \int_{\Omega ^{\varepsilon }}\boldsymbol{f}\cdot \boldsymbol{v}%
dx\right\vert \leq C\varepsilon ^{\frac{3}{2}}\left\Vert \nabla \boldsymbol{v%
}\right\Vert _{L^{2}(\Omega ^{\varepsilon })^{3\times 3}},\ \ \ \ \ \ \ \ \
\ \ \ \ \ \ \ \ \ \ \ \ \ \ \ \ \ \ 
\end{equation*}%
\begin{equation*}
\left\vert \int_{\Omega ^{\varepsilon }}A^{\varepsilon }\nabla \boldsymbol{u}%
_{\varepsilon }\cdot \nabla \boldsymbol{v}dx\right\vert \leq C\varepsilon ^{%
\frac{3}{2}}\left\Vert \nabla \boldsymbol{v}\right\Vert _{L^{2}(\Omega
^{\varepsilon })^{3\times 3}},\ \ \ \ \ \ \ \ \ \ \ \ \ 
\end{equation*}%
and 
\begin{equation*}
\left\vert \int_{\Omega ^{\varepsilon }}\boldsymbol{u}_{\varepsilon }\cdot 
\boldsymbol{v}dx\right\vert \leq C\min (\varepsilon ^{\frac{5}{2}%
},\varepsilon ^{\frac{3}{2}}K_{\varepsilon }^{\frac{1}{2}})\left\Vert \nabla 
\boldsymbol{v}\right\Vert _{L^{2}(\Omega ^{\varepsilon })^{3\times 3}}.\ \ \
\ \ \ \ \ \ \ \ \ \ \ \ 
\end{equation*}%
As for the third term on the right-hand side of (\ref{5.7}), one has 
\begin{eqnarray*}
\left\vert \int_{\Omega ^{\varepsilon }}(\boldsymbol{u}_{\varepsilon }\cdot
\nabla )\boldsymbol{u}_{\varepsilon }\cdot \boldsymbol{v}dx\right\vert &\leq
&\left\Vert \boldsymbol{u}_{\varepsilon }\right\Vert _{L^{4}(\Omega
^{\varepsilon })^{3}}\left\Vert \nabla \boldsymbol{u}_{\varepsilon
}\right\Vert _{L^{2}(\Omega ^{\varepsilon })^{3\times 3}}\left\Vert 
\boldsymbol{v}\right\Vert _{L^{4}(\Omega ^{\varepsilon })^{3}} \\
&\leq &C\varepsilon \left\Vert \nabla \boldsymbol{u}_{\varepsilon
}\right\Vert _{L^{2}(\Omega ^{\varepsilon })^{3\times 3}}^{2}\left\Vert
\nabla \boldsymbol{v}\right\Vert _{L^{2}(\Omega ^{\varepsilon })^{3\times 3}}
\\
&\leq &C\varepsilon ^{4}\left\Vert \nabla \boldsymbol{v}\right\Vert
_{L^{2}(\Omega ^{\varepsilon })^{3\times 3}},
\end{eqnarray*}%
where we used Lemma \ref{l4.1} and inequality (\ref{5.4'}). It follows that 
\begin{equation*}
\left\Vert \nabla p_{\varepsilon }\right\Vert _{H^{-1}(\Omega ^{\varepsilon
})^{3}}\leq C\left( \varepsilon ^{\frac{3}{2}}+\varepsilon ^{4}+\frac{%
\varepsilon }{K_{\varepsilon }}\min (\varepsilon ^{\frac{5}{2}},\varepsilon
^{\frac{3}{2}}K_{\varepsilon }^{\frac{1}{2}})\right) .
\end{equation*}%
Now, if $K_{\varepsilon }=O(\varepsilon ^{2})$ or $K_{\varepsilon }\gg
\varepsilon ^{2}$, then 
\begin{equation*}
\frac{\varepsilon }{K_{\varepsilon }}\min (\varepsilon ^{\frac{5}{2}%
},\varepsilon ^{\frac{3}{2}}K_{\varepsilon }^{\frac{1}{2}})\leq C\frac{%
\varepsilon ^{\frac{7}{2}}}{K_{\varepsilon }}\leq C\varepsilon ^{\frac{3}{2}%
},
\end{equation*}%
so that 
\begin{equation*}
\left\Vert \nabla p_{\varepsilon }\right\Vert _{H^{-1}(\Omega ^{\varepsilon
})^{3}}\leq C\varepsilon ^{\frac{3}{2}}.\ \ \ \ \ \ \ \ \ \ \ \ \ \ \ \ \ \
\ \ \ \ \ \ 
\end{equation*}%
If $K_{\varepsilon }\ll \varepsilon ^{2}$, then 
\begin{equation*}
\varepsilon ^{\frac{3}{2}}K_{\varepsilon }^{\frac{1}{2}}<\varepsilon ^{\frac{%
5}{2}}\text{ and }\varepsilon ^{\frac{3}{2}}<\frac{\varepsilon ^{\frac{5}{2}}%
}{K_{\varepsilon }^{\frac{1}{2}}},\ \ \ \ \ \ \ \ \ \ \ \ \ \ \ \ \ \ \ \ \
\ \ \ \ \ \ 
\end{equation*}%
so that 
\begin{equation*}
\left\Vert \nabla p_{\varepsilon }\right\Vert _{H^{-1}(\Omega ^{\varepsilon
})^{3}}\leq C(\varepsilon ^{\frac{3}{2}}+\frac{\varepsilon ^{\frac{5}{2}}}{%
K_{\varepsilon }^{\frac{1}{2}}})\leq C\frac{\varepsilon ^{\frac{5}{2}}}{%
K_{\varepsilon }^{\frac{1}{2}}}.\ \ \ \ \ \ \ \ \ \ \ \ \ \ \ \ \ \ \ \ \ \
\ \ \ \ \ \ \ \ \ \ \ \ \ \ \ 
\end{equation*}%
We have shown that 
\begin{equation}
\begin{array}{cc}
\left\Vert \nabla p_{\varepsilon }\right\Vert _{H^{-1}(\Omega ^{\varepsilon
})^{3}}\leq & \left\{ 
\begin{array}{l}
C\varepsilon ^{\frac{3}{2}}\text{ \ if }K_{\varepsilon }=O(\varepsilon ^{2})%
\text{ or }K_{\varepsilon }\gg \varepsilon ^{2}, \\ 
C\frac{\varepsilon ^{\frac{5}{2}}}{K_{\varepsilon }^{\frac{1}{2}}}\text{ if }%
K_{\varepsilon }\ll \varepsilon ^{2}.%
\end{array}%
\right.%
\end{array}
\label{5.8}
\end{equation}%
It follows from (\ref{5.6}) and (\ref{5.8}) that 
\begin{equation}
\begin{array}{cc}
\left\Vert p_{\varepsilon }^{0}\right\Vert _{H^{1}(\Omega )}\leq & \left\{ 
\begin{array}{l}
C\text{ \ if }K_{\varepsilon }=O(\varepsilon ^{2})\text{ or }K_{\varepsilon
}\gg \varepsilon ^{2}, \\ 
C\frac{\varepsilon }{K_{\varepsilon }^{1/2}}\text{ if }K_{\varepsilon }\ll
\varepsilon ^{2};%
\end{array}%
\right.%
\end{array}
\label{5.9}
\end{equation}%
and 
\begin{equation}
\begin{array}{cc}
\left\Vert p_{\varepsilon }^{1}\right\Vert _{L^{2}(\Omega ^{\varepsilon
})}\leq & \left\{ 
\begin{array}{l}
C\varepsilon ^{\frac{1}{2}}\text{ \ if }K_{\varepsilon }=O(\varepsilon ^{2})%
\text{ or }K_{\varepsilon }\gg \varepsilon ^{2}, \\ 
C\frac{\varepsilon ^{\frac{3}{2}}}{K_{\varepsilon }^{\frac{1}{2}}}\text{ if }%
K_{\varepsilon }\ll \varepsilon ^{2}.%
\end{array}%
\right.%
\end{array}
\label{5.10}
\end{equation}%
We extend $p_{\varepsilon }^{1}$ to $G_{\varepsilon }$ by zero and we denote
the corresponding extension by $\widetilde{p}_{\varepsilon }^{1}$ to see
that inequality (\ref{5.10}) holds mutatis mutandis (change $p_{\varepsilon
}^{1}$ into $\widetilde{p}_{\varepsilon }^{1}$ and $\Omega ^{\varepsilon }$
into $G_{\varepsilon }$). We summarize the estimates obtained above, in the
following lines: 
\begin{equation}
\begin{array}{cc}
\left\Vert \boldsymbol{u}_{\varepsilon }\right\Vert _{L^{2}(G_{\varepsilon
})^{3}}\leq & \left\{ 
\begin{array}{l}
C\varepsilon ^{\frac{5}{2}}\text{ \ if }K_{\varepsilon }=O(\varepsilon ^{2})%
\text{ or }K_{\varepsilon }\gg \varepsilon ^{2}, \\ 
C\varepsilon ^{\frac{3}{2}}K_{\varepsilon }^{\frac{1}{2}}\text{ if }%
K_{\varepsilon }\ll \varepsilon ^{2};%
\end{array}%
\right.%
\end{array}
\label{5.11}
\end{equation}%
\begin{equation}
\left\Vert \nabla \boldsymbol{u}_{\varepsilon }\right\Vert
_{L^{2}(G_{\varepsilon })^{3\times 3}}\leq C\varepsilon ^{\frac{3}{2}};\ \ \
\ \ \ \ \ \ \ \ \ \ \ \ \ \ \ \ \ \ \ \ \ \ \ \ \ \ \ \ \ \ \ \ \ \ \ \ \ \
\ \ \ \ \ \ \ \ \ \ \ \ \ \ \ \ \   \label{5.12}
\end{equation}%
\begin{equation}
\begin{array}{cc}
\left\Vert p_{\varepsilon }^{0}\right\Vert _{H^{1}(\Omega )}\leq & \left\{ 
\begin{array}{l}
C\text{ \ if }K_{\varepsilon }=O(\varepsilon ^{2})\text{ or }K_{\varepsilon
}\gg \varepsilon ^{2}, \\ 
C\frac{\varepsilon }{K_{\varepsilon }^{\frac{1}{2}}}\text{ if }%
K_{\varepsilon }\ll \varepsilon ^{2};%
\end{array}%
\right.%
\end{array}
\label{5.13}
\end{equation}%
and 
\begin{equation}
\begin{array}{cc}
\left\Vert \widetilde{p}_{\varepsilon }^{1}\right\Vert
_{L^{2}(G_{\varepsilon })}\leq & \left\{ 
\begin{array}{l}
C\varepsilon ^{\frac{1}{2}}\text{ \ if }K_{\varepsilon }=O(\varepsilon ^{2})%
\text{ or }K_{\varepsilon }\gg \varepsilon ^{2}, \\ 
C\frac{\varepsilon ^{\frac{3}{2}}}{K_{\varepsilon }^{\frac{1}{2}}}\text{ if }%
K_{\varepsilon }\ll \varepsilon ^{2}.%
\end{array}%
\right.%
\end{array}
\label{5.14}
\end{equation}

In the light of the above estimates, we follow the same steps as in the
preceding section. This is declined below in the following subsection.

\subsection{Passage to the limit and proof of Theorem \protect\ref{t1.3}}

\subsubsection{\textbf{Case when }$K_{\protect\varepsilon }=O(\protect%
\varepsilon ^{2})$}

We assume (\ref{4.12}), i.e. $\frac{K_{\varepsilon }}{\varepsilon ^{2}}%
\rightarrow K$ when $\varepsilon \rightarrow 0$, with $0<K<\infty $. Still
denoting by $\boldsymbol{u}_{\varepsilon }$ the extension of $\boldsymbol{u}%
_{\varepsilon }$ on $G_{\varepsilon }$, we have $\boldsymbol{u}_{\varepsilon
}\in H_{0}^{1}(G_{\varepsilon })^{3}$ with (\ref{5.11})$_{1}$, (\ref{5.12}),
(\ref{5.13})$_{1}$ and (\ref{5.14})$_{1}$. Let $\mathcal{A}$ be an algebra
with mean value on $\mathbb{R}^{2}$. Given an ordinary sequence $E$, we
derive the existence of a subsequence $E^{\prime }$ of $E$ and of functions $%
\boldsymbol{u}_{0}\in L^{2}(\Omega ;\mathcal{B}_{\mathcal{A}}^{1,2}(\mathbb{R%
}^{2};H_{0}^{1}(I)))^{3}$, $p_{0}\in H^{1}(\Omega )$, $p_{1}^{0}\in
L^{2}(\Omega ;B_{\#\mathcal{A}}^{1,2}(\mathbb{R}^{2}))$ and $p_{1}\in
L^{2}(\Omega ;\mathcal{B}_{\mathcal{A}}^{2}(\mathbb{R}^{2};L^{2}(I)))$ such
that, when $E^{\prime }\ni \varepsilon \rightarrow 0$, 
\begin{equation}
\frac{\boldsymbol{u}_{\varepsilon }}{\varepsilon ^{2}}\rightarrow 
\boldsymbol{u}_{0}\text{ in }L^{2}(G_{\varepsilon })^{3}\text{-weak }\Sigma
_{\mathcal{A}},\ \ \ \ \ \ \ \ \ \ \ \ \ \ \   \label{5.15}
\end{equation}%
\begin{equation}
\frac{1}{\varepsilon }\nabla \boldsymbol{u}_{\varepsilon }\rightarrow 
\overline{\nabla }_{y}\boldsymbol{u}_{0}\text{ in }L^{2}(G_{\varepsilon
})^{3\times 3}\text{-weak }\Sigma _{\mathcal{A}},\ \ \ \   \label{5.16}
\end{equation}%
\begin{equation}
p_{\varepsilon }^{0}\rightarrow p_{0}\text{ in }H^{1}(\Omega )\text{-weak
and in }L^{2}(\Omega )\text{-strong,}  \label{5.17}
\end{equation}%
\begin{equation}
\nabla _{\overline{x}}p_{\varepsilon }^{0}\rightarrow \nabla _{\overline{x}%
}p_{0}+\nabla _{\overline{y}}p_{1}^{0}\text{ in }L^{2}(\Omega )^{2}\text{%
-weak }\Sigma ,\ \ \ \ \ \   \label{5.18}
\end{equation}%
\begin{equation}
p_{\varepsilon }^{1}\rightarrow p_{1}\text{ in }L^{2}(G_{\varepsilon })\text{%
-weak }\Sigma _{\mathcal{A}}.\ \ \ \ \ \ \ \ \ \ \ \ \ \ \ \   \label{5.19}
\end{equation}

This being so, let us recall the definition of the set $\mathbb{J}:$%
\begin{equation*}
\mathbb{J}=\left\{ y=(\overline{y},y_{3})\in \mathbb{R}^{3}:\overline{y}\in 
\mathbb{R}^{2}\text{ and }h_{1}(\overline{y})<y_{3}<h_{2}(\overline{y}%
)\right\} .
\end{equation*}%
To any $u\in L_{loc}^{r}(\mathbb{J})$ ($1\leq r<\infty $) is associated the
transform $u^{b}$ defined by 
\begin{equation*}
u^{b}(\overline{y},t)=u(\overline{y},(1-t)h_{1}(\overline{y})+th_{2}(%
\overline{y}))\text{, }\overline{y}\in \mathbb{R}^{2}\text{ and }0<t<1,
\end{equation*}%
so that $u^{b}\in L_{loc}^{r}(\mathbb{R}^{2};L^{r}(0,1))$. This allows us to
define the following \ Besicovitch-type spaces: let $\mathcal{A}$ be an
algebra with mean value on $\mathbb{R}^{2}$ such that $h_{1},h_{2}\in 
\mathcal{A}$. By $B_{\mathcal{A}}^{r}(\mathbb{J})$ we mean the space of
those $u\in L_{loc}^{r}(\mathbb{J})$ such that $u^{b}\in B_{\mathcal{A}}^{r}(%
\mathbb{R}^{2};L^{r}(0,1))$. To $B_{\mathcal{A}}^{r}(\mathbb{J})$ we
associate the Sobolev-Besicovitch space 
\begin{equation*}
B_{\mathcal{A}}^{1,r}(\mathbb{J})=\left\{ u\in B_{\mathcal{A}}^{r}(\mathbb{J}%
):\nabla _{y}u\in B_{\mathcal{A}}^{r}(\mathbb{J})^{3}\right\} .\ \ \ \ \ \ \
\ \ \ \ \ \ \ \ \ 
\end{equation*}%
We recall that each of these spaces is a complete semi-normed space, the
seminorm in $B_{\mathcal{A}}^{r}(\mathbb{J})$ being defined by 
\begin{equation*}
\left\Vert u\right\Vert _{B_{\mathcal{A}}^{r}(\mathbb{J})}=\left\Vert
u^{b}\right\Vert _{B_{\mathcal{A}}^{r}(\mathbb{R}^{2};L^{r}(0,1))}.\ \ \ \ \
\ \ \ \ \ \ \ \ \ \ \ \ \ \ \ 
\end{equation*}%
We also define the Banach counterpart of $B_{\mathcal{A}}^{r}(\mathbb{J})$
that we denote by $\mathcal{B}_{\mathcal{A}}^{r}(\mathbb{J})$, as follows: $%
u\in \mathcal{B}_{\mathcal{A}}^{r}(\mathbb{J})$ iff $u^{b}\in \mathcal{B}_{%
\mathcal{A}}^{r}(\mathbb{R}^{2};L^{r}(0,1))$. Finally, the space 
\begin{equation*}
B_{\#}^{1,r}(\mathbb{J})=\left\{ u\in B_{\mathcal{A}}^{1,r}(\mathbb{J}):u=0%
\text{ on }\partial \mathbb{J}\right\} \ \ \ \ \ \ \ \ \ \ \ \ \ \ \ 
\end{equation*}%
will be of special interest in the sequel. We recall that $\partial \mathbb{J%
}=\{y=(\overline{y},y_{3}):\overline{y}\in \mathbb{R}^{2}$ and $y_{3}=h_{i}(%
\overline{y})$, $i=1,2\}$.

Proceeding as in Subsection \ref{subsec4.2}, we see that $\overline{\func{div%
}}_{y}\boldsymbol{u}_{0}=0$ in $\Omega \times \mathbb{J}$. Moreover, setting 
\begin{eqnarray}
\boldsymbol{u}(\overline{x}) &=&M\left( \int_{h_{1}}^{h_{2}}\boldsymbol{u}%
_{0}(\overline{x},\cdot ,y_{3})dy_{3}\right) ,\ \ \ \overline{x}\in \Omega ,
\label{5.20'} \\
&=&(u_{i}(\overline{x}))_{1\leq i\leq 3}  \notag
\end{eqnarray}%
we have $u_{3}=0$ and $\boldsymbol{u}^{\prime }=(u_{1},u_{2})\in
L^{2}(\Omega )^{2}$ with $\boldsymbol{u}^{\prime }\cdot \nu =0$ on $\partial
\Omega $ and $\func{div}_{\overline{x}}\boldsymbol{u}^{\prime }=0$ in $%
\Omega $.

With this in mind, the first homogenization result in this case reads as
follows.

\begin{theorem}
\label{t5.1}Let $\mathcal{A}$ be an ergodic algebra with mean value on $%
\mathbb{R}^{2}$. Assume that \emph{(\textbf{A3})}, \emph{(\ref{4.12})}, 
\emph{(\ref{5.1})} and \emph{(\ref{5.2})}. Let $(\boldsymbol{u}_{\varepsilon
},p_{\varepsilon })$ be determined by \emph{(\ref{4.1})}. Then $E^{\prime
}\ni \varepsilon \rightarrow 0$, one has \emph{(\ref{5.15})} to \emph{(\ref%
{5.19})}, where the quadruple $(\boldsymbol{u}_{0},p_{0},p_{1}^{0},p_{1})$
solves the system 
\begin{equation}
\left\{ 
\begin{array}{l}
-\overline{\func{div}}_{y}(A\overline{\nabla }_{y}\boldsymbol{u}_{0})+\frac{%
\mu }{K}\boldsymbol{u}_{0}+\overline{\nabla }_{y}(p_{1}^{0}+p_{1})=f-\nabla
_{\overline{x}}p_{0}\text{ in }\Omega \times \mathbb{J}, \\ 
\\ 
\overline{\func{div}}_{y}\boldsymbol{u}_{0}=0\text{ in }\Omega \times 
\mathbb{J}, \\ 
\\ 
\func{div}_{\overline{x}}M\left( \int_{h_{1}}^{h_{2}}\boldsymbol{u}_{0}(%
\overline{x},\cdot ,y_{3})dy_{3}\right) =0\text{ in }\Omega , \\ 
\\ 
M\left( \int_{h_{1}}^{h_{2}}\boldsymbol{u}_{0}(\overline{x},\cdot
,y_{3})dy_{3}\right) \cdot \nu =0\text{ on }\partial \Omega .%
\end{array}%
\right.  \label{5.20}
\end{equation}
\end{theorem}

\begin{proof}
Let the assumptions of Theorem \ref{t5.1} be in force. Let $\varphi \in (%
\mathcal{C}_{0}^{\infty }(\Omega )\otimes B_{\#}^{1,2}(\mathbb{J}))^{3}$.
Defining $\varphi ^{\varepsilon }\in H_{0}^{1}(\Omega ^{\varepsilon })^{3}$
by $\varphi ^{\varepsilon }(x)=\varphi (\overline{x},x/\varepsilon )$ ($x\in
\Omega ^{\varepsilon }$) and choosing $\varphi ^{\varepsilon }$ as test
function in (\ref{4.1}), we get, after dividing both members of the
resulting equality by $\varepsilon $, 
\begin{equation}
\left\{ 
\begin{array}{l}
\frac{1}{\varepsilon }\int_{\Omega ^{\varepsilon }}A^{\varepsilon }\nabla 
\boldsymbol{u}_{\varepsilon }\cdot \left[ (\nabla _{\overline{x}}\varphi
)^{\varepsilon }+\frac{1}{\varepsilon }(\nabla _{y}\varphi )^{\varepsilon }%
\right] dx+\mu \frac{\varepsilon ^{2}}{K_{\varepsilon }}\frac{1}{\varepsilon 
}\int_{\Omega ^{\varepsilon }}\frac{\boldsymbol{u}_{\varepsilon }}{%
\varepsilon ^{2}}\varphi ^{\varepsilon }dx \\ 
\\ 
\ \ \ +\frac{\rho }{\phi ^{2}}\frac{1}{\varepsilon }\int_{\Omega
^{\varepsilon }}(\boldsymbol{u}_{\varepsilon }\cdot \nabla )\boldsymbol{u}%
_{\varepsilon }\varphi ^{\varepsilon }dx+\frac{1}{\varepsilon }\int_{\Omega
^{\varepsilon }}\nabla _{\overline{x}}p_{\varepsilon }^{0}\cdot \varphi
^{\varepsilon }dx \\ 
\\ 
\ \ \ \ \ \ \ -\int_{\Omega ^{\varepsilon }}p_{\varepsilon }^{1}[(\func{div}%
_{\overline{x}}\varphi )^{\varepsilon }+\frac{1}{\varepsilon }(\func{div}%
_{y}\varphi )^{\varepsilon }]dx=\frac{1}{\varepsilon }\int_{\Omega
^{\varepsilon }}\boldsymbol{f}\varphi ^{\varepsilon }dx.%
\end{array}%
\right.  \label{5.21}
\end{equation}%
Let us consider each term in (\ref{5.21}) separately. As for the first term
on the left-hand side, it is equal to 
\begin{equation*}
\int_{\Omega ^{\varepsilon }}A^{\varepsilon }\frac{\nabla \boldsymbol{u}%
_{\varepsilon }}{\varepsilon }\cdot (\nabla _{\overline{x}}\varphi
)^{\varepsilon }dx+\frac{1}{\varepsilon }\int_{\Omega ^{\varepsilon
}}A^{\varepsilon }\frac{\nabla \boldsymbol{u}_{\varepsilon }}{\varepsilon }%
\cdot (\nabla _{y}\varphi )^{\varepsilon }dx=I_{1}+I_{2}.
\end{equation*}%
It is easy to see that $I_{1}\rightarrow 0$ when $E^{\prime }\ni \varepsilon
\rightarrow 0$, while, for $I_{2}$, appealing to Theorem \ref{t3.5}
associated to (\ref{5.16}), we have 
\begin{eqnarray*}
I_{2} &\rightarrow &\int_{\Omega }\int_{h_{1}^{-}}^{h_{2}^{+}}M(\chi _{%
\mathbb{J}}(\cdot ,y_{3})A(\cdot ,y_{3})\overline{\nabla }_{y}\boldsymbol{u}%
_{0}(\overline{x},\cdot ,y_{3})\cdot \nabla _{y}\varphi (\overline{x},\cdot
,y_{3}))dy_{3}d\overline{x} \\
&=&\int_{\Omega }M\left( \int_{h_{1}}^{h_{2}}A\overline{\nabla }_{y}%
\boldsymbol{u}_{0}\cdot \nabla _{y}\varphi dy_{3}\right) d\overline{x}.
\end{eqnarray*}%
It is worth noting that in the last convergence above, we have used $A\nabla
_{y}\varphi $ as test function; indeed, since $A\in (L^{\infty }(\mathbb{R}%
^{3})\cap B_{\mathcal{A}}^{2}(\mathbb{R}^{2};L^{\infty }(I)))^{3\times 3}$
(where here, $I=(h_{1}^{-},h_{2}^{+})$) and $L^{\infty }(\mathbb{R}^{3})\cap
B_{\mathcal{A}}^{2}(\mathbb{R}^{2};L^{\infty }(I))\hookrightarrow L^{\infty
}(\mathbb{R}^{2};L^{\infty }(I))$, we get that $A\nabla _{y}\varphi \in (%
\mathcal{C}_{0}^{\infty }(\Omega )\otimes B_{\mathcal{A}}^{2}(\mathbb{R}%
^{2};L^{2}(I)))^{3\times 3}$, so that it can be taken as test function in
the $\Sigma _{\mathcal{A}}$-convergence. It follows that 
\begin{equation}
\frac{1}{\varepsilon }\int_{\Omega ^{\varepsilon }}A^{\varepsilon }\nabla 
\boldsymbol{u}_{\varepsilon }\cdot \left[ (\nabla _{\overline{x}}\varphi
)^{\varepsilon }+\frac{1}{\varepsilon }(\nabla _{y}\varphi )^{\varepsilon }%
\right] dx\rightarrow \int_{\Omega }M\left( \int_{h_{1}}^{h_{2}}A\overline{%
\nabla }_{y}\boldsymbol{u}_{0}\cdot \nabla _{y}\varphi dy_{3}\right) d%
\overline{x}.  \label{5.22}
\end{equation}%
It is a fact using (\ref{5.15}) that 
\begin{equation}
\frac{1}{\varepsilon }\int_{\Omega ^{\varepsilon }}\frac{\boldsymbol{u}%
_{\varepsilon }}{\varepsilon ^{2}}\varphi ^{\varepsilon }dx\rightarrow
\int_{\Omega }M\left( \int_{h_{1}}^{h_{2}}\boldsymbol{u}_{0}\varphi
dy_{3}\right) d\overline{x},\ \ \ \ \ \ \   \label{5.23}
\end{equation}%
and, as in Section \ref{sec4}, we observe that 
\begin{equation*}
\left\vert \int_{\Omega ^{\varepsilon }}(\boldsymbol{u}_{\varepsilon }\cdot
\nabla )\boldsymbol{u}_{\varepsilon }\varphi ^{\varepsilon }dx\right\vert
\leq C\varepsilon ^{3},\ \ \ \ \ \ \ \ \ \ \ \ \ \ \ \ 
\end{equation*}%
so that 
\begin{equation}
\frac{1}{\varepsilon }\int_{\Omega ^{\varepsilon }}(\boldsymbol{u}%
_{\varepsilon }\cdot \nabla )\boldsymbol{u}_{\varepsilon }\varphi
^{\varepsilon }dx\rightarrow 0\text{ when }E^{\prime }\ni \varepsilon
\rightarrow 0\text{.}  \label{5.24}
\end{equation}%
Concerning the terms involving the pressure, it holds that 
\begin{equation}
\frac{1}{\varepsilon }\int_{\Omega ^{\varepsilon }}\nabla _{\overline{x}%
}p_{\varepsilon }^{0}\cdot \varphi ^{\varepsilon }dx\rightarrow \int_{\Omega
}M\left( \int_{h_{1}}^{h_{2}}(\nabla _{\overline{x}}p_{0}+\nabla _{\overline{%
y}}p_{1}^{0})\cdot \varphi dy_{3}\right) d\overline{x},  \label{5.25}
\end{equation}%
and 
\begin{equation}
\int_{\Omega ^{\varepsilon }}p_{\varepsilon }^{1}[(\func{div}_{\overline{x}%
}\varphi )^{\varepsilon }+\frac{1}{\varepsilon }(\func{div}_{y}\varphi
)^{\varepsilon }]dx\rightarrow \int_{\Omega }M\left(
\int_{h_{1}}^{h_{2}}p_{1}\func{div}_{y}\varphi dy_{3}\right) d\overline{x}.
\label{5.26}
\end{equation}%
Finally, one has, as $E^{\prime }\ni \varepsilon \rightarrow 0$, 
\begin{equation}
\frac{1}{\varepsilon }\int_{\Omega ^{\varepsilon }}\boldsymbol{f}\varphi
^{\varepsilon }dx\rightarrow \int_{\Omega }M\left( \int_{h_{1}}^{h_{2}}%
\boldsymbol{f}\varphi dy_{3}\right) d\overline{x}.\ \ \ \ \ \ \ \ \ \ \ \ \
\ \ \ \ \ \ \   \label{5.27}
\end{equation}%
Collecting the convergence results (\ref{5.22})-(\ref{5.27}), we obtain
(when passing to the limit in (\ref{5.21})) the following variational system 
\begin{equation}
\left\{ 
\begin{array}{l}
\int_{\Omega }M\left( \int_{h_{1}}^{h_{2}}A\overline{\nabla }_{y}\boldsymbol{%
u}_{0}\cdot \nabla _{y}\varphi dy_{3}\right) d\overline{x}+\frac{\mu }{K}%
\int_{\Omega }M\left( \int_{h_{1}}^{h_{2}}\boldsymbol{u}_{0}\varphi
dy_{3}\right) d\overline{x} \\ 
\\ 
\ \ +\int_{\Omega }M\left( \int_{h_{1}}^{h_{2}}\nabla _{\overline{x}%
}p_{0}\cdot \varphi dy_{3}\right) d\overline{x}-\int_{\Omega }M\left(
\int_{h_{1}}^{h_{2}}(p_{1}^{0}+p_{1})\func{div}_{y}\varphi dy_{3}\right) d%
\overline{x} \\ 
\\ 
\ \ \ =\int_{\Omega }M\left( \int_{h_{1}}^{h_{2}}\boldsymbol{f}\varphi
dy_{3}\right) d\overline{x}\text{, for all }\varphi \in (\mathcal{C}%
_{0}^{\infty }(\Omega )\otimes B_{\#}^{1,2}(\mathbb{J}))^{3}.%
\end{array}%
\right.  \label{5.28}
\end{equation}%
By density, (\ref{5.28}) holds true for every $\varphi \in L^{2}(\Omega
;B_{\#}^{1,2}(\mathbb{J}))^{3}$.

Also, proceeding as in Section \ref{sec4}, we derive the following
properties for $\boldsymbol{u}_{0}$: 
\begin{equation}
\func{div}_{\overline{x}}M\left( \int_{h_{1}}^{h_{2}}\boldsymbol{u}_{0}(%
\overline{x},\cdot ,y_{3})dy_{3}\right) =0\text{ in }\Omega ,\ \ \ \ \ \ \ \
\ \ \ \ \ \ \ \ \ \ \ \ \ \ \ \ \ \ \ \   \label{5.29}
\end{equation}%
\begin{equation}
M\left( \int_{h_{1}}^{h_{2}}\boldsymbol{u}_{0}(\overline{x},\cdot
,y_{3})dy_{3}\right) \cdot \nu =0\text{ on }\partial \Omega .\ \ \ \ \ \ \ \
\ \ \ \ \ \ \ \ \ \ \ \ \ \ \ \ \ \ \ \ \ \   \label{5.30}
\end{equation}%
We may also check that $p_{0}\in L_{0}^{2}(\Omega )$. Indeed, since $%
p_{\varepsilon }\in L_{0}^{2}(\Omega ^{\varepsilon })$, we have 
\begin{equation*}
0=\frac{1}{\varepsilon }\int_{\Omega ^{\varepsilon }}p_{\varepsilon
}dx=\int_{\Omega }(h_{2}(\frac{\overline{x}}{\varepsilon })-h_{1}(\frac{%
\overline{x}}{\varepsilon }))p_{\varepsilon }^{0}d\overline{x}+\int_{\Omega
^{\varepsilon }}p_{\varepsilon }^{1}dx.
\end{equation*}%
Letting $E^{\prime }\ni \varepsilon \rightarrow 0$ above yields $%
\int_{\Omega }M(h_{2}-h_{1})p_{0}d\overline{x}=0$, that is, $\int_{\Omega
}p_{0}d\overline{x}=0$, where we have taken (\ref{5.2}) into account. This
shows that $p_{0}\in H^{1}(\Omega )\cap L_{0}^{2}(\Omega )$.

Accounting of (\ref{5.28}), (\ref{5.29}) and (\ref{5.30}), we get readily
that the triple $(\boldsymbol{u}_{0},p_{0},q=p_{1}^{0}+p_{1})$ solves the
system (\ref{5.20}). This completes the proof of the theorem.
\end{proof}

In order to derive the homogenized problem, we follow the same procedure as
in the proof of Theorem \ref{t4.2}. To this end, let $(e_{i})_{1i\leq 3}$ be
the canonical basis in $\mathbb{R}^{3}$. For $1\leq i\leq 3$, consider the
Stokes-Brinkmann system 
\begin{equation}
\left\{ 
\begin{array}{l}
-\overline{\func{div}}_{y}\left( A(y)\overline{\nabla }_{y}\boldsymbol{w}%
_{i}\right) +\frac{\mu }{K}\boldsymbol{w}_{i}+\overline{\nabla }_{y}\pi
_{i}=e_{i}\text{ in }\mathbb{J} \\ 
\overline{\func{div}}_{y}\boldsymbol{w}_{i}=0\text{ in }\mathbb{J}, \\ 
\boldsymbol{w}_{i}=0\text{ on }\partial \mathbb{J}.%
\end{array}%
\right.  \label{5.31}
\end{equation}%
Then (\ref{5.31}) possesses a unique solution in the space 
\begin{equation*}
B_{\#,\func{div}}^{1,2}(\mathbb{J})=\left\{ \boldsymbol{u}\in B_{\#}^{1,2}(%
\mathbb{J})^{3}:\overline{\func{div}}_{y}\boldsymbol{u}=0\text{ in }\mathbb{J%
}\right\} .
\end{equation*}%
With this in mind, if we choose in the variational form of (\ref{5.20}) the
test function $\varphi =\psi \otimes \boldsymbol{w}_{i}$ ($1\leq i\leq 3$)
where $\boldsymbol{w}_{i}$ is determined by (\ref{5.31}) and $\psi \in 
\mathcal{C}_{0}^{\infty }(\Omega )$, and then take in the variational form
of (\ref{5.31}), the test function $\boldsymbol{u}_{0}(\overline{x},\cdot )$%
, and finally we compare the resulting equalities, we get at once 
\begin{equation}
\boldsymbol{u}^{\prime }(\overline{x})=\widehat{A}(\boldsymbol{f}_{1}(%
\overline{x})-\nabla _{\overline{x}}p_{0}(\overline{x}))\text{,\ a.e. }%
\overline{x}\in \Omega \ \ \ \ \ \ \ \ \ \ \ \ \ \ \ \ \ \ \ \ \ \ \ \ 
\label{5.32}
\end{equation}%
where the matrix $\widehat{A}=(\widehat{a}_{ij})_{1\leq i,j\leq 2}$ is
symmetric, positive definite and is defined by 
\begin{equation}
\widehat{a}_{ij}=M\left( \int_{h_{1}}^{h_{2}}A\overline{\nabla }_{y}%
\boldsymbol{w}_{i}\cdot \overline{\nabla }_{y}\boldsymbol{w}%
_{j}dy_{3}\right) +\frac{\mu }{K}M\left( \int_{h_{1}}^{h_{2}}\boldsymbol{w}%
_{i}\cdot \boldsymbol{w}_{j}dy_{3}\right) .  \label{5.33}
\end{equation}%
We have just derived the main result in the case when $K_{\varepsilon
}=O(\varepsilon ^{2})$, and it reads as follows.

\begin{theorem}
\label{t5.2}The assumptions are those of Theorem \emph{\ref{t5.1}}. For any $%
\varepsilon >0$, let $(\boldsymbol{u}_{\varepsilon },p_{\varepsilon
}=p_{\varepsilon }^{0}+\varepsilon p_{\varepsilon }^{1})\in H_{0}^{1}(\Omega
^{\varepsilon })^{3}\times L_{0}^{2}(\Omega ^{\varepsilon })$ be a solution
of \emph{(\ref{4.1})}. Then, still denoting by $\boldsymbol{u}_{\varepsilon
} $ and $p_{\varepsilon }^{1}$ the extension of $\boldsymbol{u}_{\varepsilon
}$ and $p_{\varepsilon }^{1}$ by zero on $G_{\varepsilon }$, the sequence $(%
\boldsymbol{u}_{\varepsilon }/\varepsilon ^{2},p_{\varepsilon })$ weakly $%
\Sigma _{\mathcal{A}}$-converges in $L^{2}(G_{\varepsilon })^{3}\times
L^{2}(G_{\varepsilon })$ towards $(\boldsymbol{u}_{0},p_{0})$ determined by 
\emph{(\ref{5.15})-(\ref{5.18})}. Defining $\boldsymbol{u}=(\boldsymbol{u}%
^{\prime },u_{3})$ by \emph{(\ref{5.20'})}, we have $u_{3}=0$ and $(%
\boldsymbol{u}^{\prime },p_{0})$ is the unique solution of the homogenized
problem 
\begin{equation*}
\left\{ 
\begin{array}{l}
\boldsymbol{u}^{\prime }=\widehat{A}(\boldsymbol{f}_{1}-\nabla _{\overline{x}%
}p_{0})\text{ in }\Omega , \\ 
\func{div}_{\overline{x}}\boldsymbol{u}^{\prime }=0\text{ in }\Omega \text{,
and }\boldsymbol{u}^{\prime }\cdot \nu =0\text{ on }\partial \Omega ,%
\end{array}%
\right.
\end{equation*}%
where the matrix $\widehat{A}$ is defined by \emph{(\ref{5.33})}.
\end{theorem}

\subsubsection{\textbf{Case when }$K_{\protect\varepsilon }\ll \protect%
\varepsilon ^{2}$\textbf{\ or }$K_{\protect\varepsilon }\gg \protect%
\varepsilon ^{2}$}

Following the same steps as in the previous section, we may proceed as in
the previous case to obtain the homogenization result in the two other
regimes. This is left to the reader.

\section{Some concrete applications of Theorems \protect\ref{t1.2} and 
\protect\ref{t1.3}\label{sec6}}

In the previous sections, we have used systematically the algebras with mean
value in assumptions (\textbf{A3}) and (\ref{3.30}) (or (\ref{5.2}))
respectively on the coefficients of the diffusion operator $-\func{div}%
(A^{\varepsilon }\nabla .)$ and on the functions $h_{1}$ and $h_{2}$.
Assumption (\textbf{A3}) shows how the microstructures are distributed
inside the domain $\Omega ^{\varepsilon }$ while (\ref{3.30}) deals with the
way the lateral boundaries behave. We present here below a few concrete
physical situations.

\subsection{Problem 1: pure periodic environment\label{subsec6.1}}

We assume that the heterogeneities are uniformly distributed in $\Omega $.
This means that the distribution function of the microstructures is
periodic, so that the matrix-function $\overline{y}\mapsto A(\overline{y}%
,y_{3})$ is $1$-periodic in each of its $y_{i}$ ($i=1,2$). The underlying
algebra with mean value here is thus the algebra of $Y$-periodic continuous
functions $\mathcal{A}=\mathcal{C}_{per}(Y)$, $Y=(0,1)^{2}$. The mean value
of a function $u\in \mathcal{C}_{per}(Y)$ is given by 
\begin{equation*}
M(u)=\int_{Y}u(y)dy.
\end{equation*}%
The function spaces associated to $\mathcal{A}$ are as follows: $B_{\mathcal{%
A}}^{2}(\mathbb{R}^{2};L^{2}(I))=L_{per}^{2}(Y;L^{2}(I))$ (the space of
functions in $L_{loc}^{2}(\mathbb{R}^{2};L^{2}(I))$ that are $Y$-periodic), $%
B_{\mathcal{A}}^{1,2}(\mathbb{R}^{d-1};H^{1}(I))=H_{per}^{1}(Y;H^{1}(I))$
(the subspace of $H_{loc}^{1}(Y;H^{1}(I))$ consisting of $Y$-periodic
functions). It is worth noting that $B_{\mathcal{A}}^{2}(\mathbb{R}%
^{2};L^{2}(I))=\mathcal{B}_{\mathcal{A}}^{2}(\mathbb{R}^{2};L^{2}(I))$ since 
$L_{per}^{2}(Y;L^{2}(I))$ is a Banach space with the corresponding norm, and
so $B_{\mathcal{A}}^{1,2}(\mathbb{R}^{2};H^{1}(I))=\mathcal{B}_{\mathcal{A}%
}^{1,2}(\mathbb{R}^{2};H^{1}(I))$.

In this case, the sigma-convergence concept is merely the well-known
two-scale convergence method for thin heterogeneous domains defined in \cite%
{RJ2007} (see Remark \ref{r2.2}) as follows: A sequence $(u_{\varepsilon
})_{\varepsilon >0}\subset L^{2}(G_{\varepsilon })$ weakly two-scale
converges in $L^{2}(G_{\varepsilon })$\ towards $u_{0}\in
L^{2}(G_{0};L_{per}^{2}(Y;L^{2}(I)))$\ if, when $\varepsilon \rightarrow 0$, 
\begin{equation*}
\frac{1}{\varepsilon }\int_{G_{\varepsilon }}u_{\varepsilon }(x)f\left( 
\overline{x},\frac{x}{\varepsilon }\right) dx\rightarrow
\int_{G_{0}}\int_{Z}u_{0}(\overline{x},y)f(\overline{x},y)dyd\overline{x}
\end{equation*}%
for any $f\in L^{2}(G_{0};\mathcal{C}_{per}(Y;L^{2}(I)))$, where $Z=Y\times
I $.

The above assumption on the matrix $A$ amounts to $A\in
L_{per}^{2}(Y;L^{2}(I))^{3\times 3}$. We also assume that $h_{1},h_{2}\in
W^{1,\infty }(\mathbb{R}^{2})\cap \mathcal{C}_{per}(Y)$. In order to state
the periodic version of Theorem \ref{t1.3}, we need to define some function
spaces. First of all, let 
\begin{equation*}
\widetilde{Y}=\{y\in \mathbb{R}^{3}:\overline{y}\in Y\text{ and }h_{1}(%
\overline{y})<y_{3}<h_{2}(\overline{y})\}\text{ and }\Gamma =\{y\in \partial 
\mathbb{J}:\overline{y}\in Y\},
\end{equation*}%
where we have assumed that the functions $h_{1}$ and $h_{2}$ are continuous $%
Y$-periodic functions: $h_{1},h_{2}\in \mathcal{C}_{per}(Y)$. The periodic
counterparts of the spaces $B_{\mathcal{A}}^{2}(\mathbb{J})$, $B_{\mathcal{A}%
}^{1,2}(\mathbb{J})$ and $B_{\#}^{1,2}(\mathbb{J})$, denoted below
respectively by $L_{per}^{2}(\widetilde{Y})$, $H_{per}^{1}(\widetilde{Y})$
and $H_{\#}^{1}(\widetilde{Y})$ are defined as follows: 
\begin{equation*}
L_{per}^{2}(\widetilde{Y})=\left\{ u\in L_{loc}^{2}(\mathbb{J}):\int_{%
\widetilde{Y}}\left\vert u\right\vert ^{2}dy<\infty \text{ and }u(\overline{y%
}+k,y_{3})=u(\overline{y},y_{3})\ \forall k\in \mathbb{Z}^{2}\text{, a.e. }%
y\in \mathbb{J}\right\} ,
\end{equation*}%
\begin{equation*}
H_{per}^{1}(\widetilde{Y})=\left\{ u\in H_{loc}^{1}(\mathbb{J}):u\in
L_{per}^{2}(\widetilde{Y}),\ \nabla _{y}u\in L_{per}^{2}(\widetilde{Y}%
)^{3}\right\} ,
\end{equation*}%
\begin{equation*}
H_{\#}^{1}(\widetilde{Y})=\left\{ u\in H_{per}^{1}(\widetilde{Y}):u=0\text{
on }\Gamma \right\} .
\end{equation*}

For the benefit of the reader, we restate the homogenization result in
Theorem \ref{t1.3} in the periodic setting.

\begin{theorem}
\label{t6.1}Assume that $\Omega ^{\varepsilon }$ is given by \emph{(\ref{1.2}%
)} where the functions $h_{1},h_{2}\in \mathcal{C}_{per}(Y)$. Let $(%
\boldsymbol{u}_{\varepsilon },p_{\varepsilon }=p_{\varepsilon
}^{0}+\varepsilon p_{\varepsilon }^{1})\in H_{0}^{1}(\Omega ^{\varepsilon
})^{3}\times L_{0}^{2}(\Omega ^{\varepsilon })$ be a solution of \emph{(\ref%
{1.1})}. Assume that $A\in (L_{per}^{2}(Y;L^{\infty }(I)))^{3\times 3}$.
Then:

\begin{itemize}
\item[(i)] If $K_{\varepsilon }=O(\varepsilon ^{2})$ with $K_{\varepsilon
}/\varepsilon ^{2}\rightarrow K$ when $\varepsilon \rightarrow 0$, $%
0<K<\infty $, then, still denoting by $\boldsymbol{u}_{\varepsilon }$ the
extension by zero of $\boldsymbol{u}_{\varepsilon }$ on $G_{\varepsilon
}=\Omega \times (\varepsilon h_{1}^{-},\varepsilon h_{2}^{+})$, one has 
\begin{equation*}
\frac{\boldsymbol{u}_{\varepsilon }}{\varepsilon ^{2}}\rightarrow 
\boldsymbol{u}_{0}\text{ in }L^{2}(G_{\varepsilon })^{3}\text{-weak }\Sigma
_{\mathcal{A}},\ \ \ \ \ \ \ \ \ \ \ \ \ \ \ 
\end{equation*}%
and 
\begin{equation*}
p_{\varepsilon }^{0}\rightarrow p_{0}\text{ in }H^{1}(\Omega )\text{-weak
and in }L^{2}(\Omega )\text{-strong.}
\end{equation*}%
Defining $\boldsymbol{u}=(\boldsymbol{u}^{\prime },u_{3})$ by $\boldsymbol{u}%
(\overline{x})=\int_{\widetilde{Y}}\boldsymbol{u}_{0}(\overline{x},y)dy$, we
have $u_{3}=0$ and $(\boldsymbol{u}^{\prime },p_{0})$ is the unique solution
of the homogenized problem 
\begin{equation*}
\left\{ 
\begin{array}{l}
\boldsymbol{u}^{\prime }=\widehat{A}(\boldsymbol{f}_{1}-\nabla _{\overline{x}%
}p_{0})\text{ in }\Omega , \\ 
\func{div}_{\overline{x}}\boldsymbol{u}^{\prime }=0\text{ in }\Omega \text{,
and }\boldsymbol{u}^{\prime }\cdot \nu =0\text{ on }\partial \Omega ,%
\end{array}%
\right.
\end{equation*}%
where $\widehat{A}=(\widehat{a}_{ij})_{1\leq i,j\leq 2}$ is a symmetric,
positive definite $2\times 2$ matrix defined by its entries 
\begin{equation*}
\widehat{a}_{ij}=\int_{\widetilde{Y}}A\nabla _{y}\boldsymbol{w}_{i}\cdot
\nabla _{y}\boldsymbol{w}_{j}dy+\frac{\mu }{K}\int_{\widetilde{Y}}%
\boldsymbol{w}_{i}\cdot \boldsymbol{w}_{j}dy.
\end{equation*}%
Here $\boldsymbol{w}_{i}$ ($1\leq i\leq 2$) is the unique solution in $%
H_{\#}^{1}(\widetilde{Y})^{3}$ of the Stokes-Brinkmann system 
\begin{equation*}
\left\{ 
\begin{array}{l}
-\func{div}_{y}\left( A(y)\nabla _{y}\boldsymbol{w}_{i}\right) +\frac{\mu }{K%
}\boldsymbol{w}_{i}+\nabla _{y}\pi _{i}=e_{i}\text{ in }\mathbb{J}, \\ 
\func{div}_{y}\boldsymbol{w}_{i}=0\text{ in }\mathbb{J}, \\ 
\boldsymbol{w}_{i}=0\text{ on }\partial \mathbb{J}.%
\end{array}%
\right.
\end{equation*}%
$e_{i}$ being the $i$\emph{th} vector of the canonical basis in $\mathbb{R}%
^{3}$.

\item[(ii)] If $K_{\varepsilon }\ll \varepsilon ^{2}$, then, up to a
subsequence, one has 
\begin{equation*}
\frac{\boldsymbol{u}_{\varepsilon }}{\varepsilon K_{\varepsilon }^{\frac{1}{2%
}}}\rightarrow \boldsymbol{u}_{0}\text{ in }L^{2}(G_{\varepsilon })^{3}\text{%
-weak }\Sigma _{\mathcal{A}},\ \ \ \ \ \ 
\end{equation*}%
\begin{equation*}
\frac{K_{\varepsilon }^{\frac{1}{2}}}{\varepsilon }p_{\varepsilon
}^{0}\rightarrow p_{0}\text{ in }H^{1}(\Omega )\text{-weak and in }%
L^{2}(\Omega )\text{-strong.}\ \ \ \ \ \ \ 
\end{equation*}%
Furthermore, defining $\boldsymbol{u}$ as in \emph{(i)} above, we have $%
u_{3}=0$ and $(\boldsymbol{u}^{\prime },p_{0})$ is a solution of 
\begin{equation*}
\boldsymbol{u}^{\prime }=-\widehat{A}\nabla _{\overline{x}}p_{0}\text{ in }%
\Omega \text{, }\func{div}_{\overline{x}}\boldsymbol{u}^{\prime }=0\text{ in 
}\Omega \text{ and }\boldsymbol{u}^{\prime }\cdot \nu =0\text{ on }\partial
\Omega ,
\end{equation*}%
where $\widehat{A}$ is a symmetric matrix defined by 
\begin{equation*}
\widehat{A}=(\widehat{a}_{ij})_{1\leq i,j\leq 2}\text{ with }\widehat{a}%
_{ij}=\mu \int_{\widetilde{Y}}\boldsymbol{w}_{i}\boldsymbol{w}_{j}dy,
\end{equation*}%
$\boldsymbol{w}_{i}$ $(1\leq i\leq 2)$ being the unique solution in $%
L_{per}^{2}(\widetilde{Y})^{3}$ of 
\begin{equation*}
\mu \boldsymbol{w}_{i}+\nabla _{y}\pi _{i}=e_{i}\text{ in }\mathbb{J}\text{
and }\func{div}_{y}\boldsymbol{w}_{i}=0\text{ in }\mathbb{J}.
\end{equation*}

\item[(iii)] If $K_{\varepsilon }\gg \varepsilon ^{2}$, then, still denoting
by $\boldsymbol{u}_{\varepsilon }$ and $p_{\varepsilon }^{1}$ the extension
of $\boldsymbol{u}_{\varepsilon }$ and $p_{\varepsilon }^{1}$ by zero on $%
G_{\varepsilon }$,we have, 
\begin{equation*}
\frac{\boldsymbol{u}_{\varepsilon }}{\varepsilon ^{2}}\rightarrow 
\boldsymbol{u}_{0}\text{ in }L^{2}(G_{\varepsilon })^{3}\text{-weak }\Sigma
_{\mathcal{A}},\ \ \ \ \ \ \ \ \ \ \ \ \ \ \ 
\end{equation*}%
\begin{equation*}
\frac{1}{\varepsilon }\nabla \boldsymbol{u}_{\varepsilon }\rightarrow 
\overline{\nabla }_{y}\boldsymbol{u}_{0}\text{ in }L^{2}(G_{\varepsilon
})^{3\times 3}\text{-weak }\Sigma _{\mathcal{A}},\ \ \ \ 
\end{equation*}%
\begin{equation*}
p_{\varepsilon }^{0}\rightarrow p_{0}\text{ in }H^{1}(\Omega )\text{-weak
and in }L^{2}(\Omega )\text{-strong,}
\end{equation*}%
Still defining $\boldsymbol{u}$ as in \emph{(i)} above, it holds that 
\begin{equation*}
\left\{ 
\begin{array}{l}
\boldsymbol{u}^{\prime }=\widehat{A}(\boldsymbol{f}_{1}-\nabla _{\overline{x}%
}p_{0})\text{ in }\Omega \\ 
\func{div}_{\overline{x}}\boldsymbol{u}^{\prime }=0\text{ in }\Omega \text{
and }\boldsymbol{u}^{\prime }\cdot \nu =0\text{ on }\partial \Omega ,%
\end{array}%
\right.
\end{equation*}%
where $\widehat{A}=(\widehat{a}_{ij})_{1\leq i,j\leq 2}$ is given by 
\begin{equation*}
\widehat{a}_{ij}=\int_{\widetilde{Y}}A\nabla _{y}\boldsymbol{w}_{i}\cdot
\nabla _{y}\boldsymbol{w}_{j}dy,\ 1\leq i,j\leq 2,
\end{equation*}%
with $\boldsymbol{w}_{i}$ $(1\leq i\leq 2)$ being the unique solution in $%
H_{\#}^{1}(\widetilde{Y})^{3}$ of the Stokes system 
\begin{equation*}
\left\{ 
\begin{array}{l}
-\func{div}_{y}\left( A(y)\nabla _{y}\boldsymbol{w}_{i}\right) +\nabla
_{y}\pi _{i}=e_{i}\text{ in }\mathbb{J} \\ 
\func{div}_{y}\boldsymbol{w}_{i}=0\text{ in }\mathbb{J}.%
\end{array}%
\right.
\end{equation*}
\end{itemize}
\end{theorem}

\begin{proof}
The result is a mere consequence of the equality $M(u)=\int_{Y}u(y)dy$ for
any $u\in L_{per}^{2}(Y)$.
\end{proof}

\subsection{Problem 2: asymptotic periodic framework}

We assume that the functions $h_{1}$ and $h_{2}$ are asymptotic periodic,
that is each of these functions can be express as a sum of a continuous
periodic function and of a continuous function which vanishes at infinity.
This leads to the consideration of the algebra with mean value $\mathcal{A}=%
\mathcal{C}_{per}(Y)+\mathcal{C}_{0}(\mathbb{R}^{2})\equiv \mathcal{B}%
_{\infty ,per}(\mathbb{R}^{2})$ \cite[Section 5.2.3]{JW2021}, where $%
\mathcal{C}_{0}(\mathbb{R}^{2})$ stands for the Banach algebra of continuous
functions vanishing at infinity. Since $\mathcal{C}_{per}(Y)\subset \mathcal{%
C}_{per}(Y)+\mathcal{C}_{0}(\mathbb{R}^{2})$, we may assume either that $%
A\in L_{per}^{2}(Y;L^{2}(I))^{3\times 3}$ or $A\in B_{\mathcal{B}_{\infty
,per}(\mathbb{R}^{2})}^{2}(\mathbb{R}^{2};L^{2}(I))^{3\times 3}$. Then the
results in Theorems \ref{t1.2} and \ref{t1.3} are obtained with the algebra
wmv $\mathcal{A}=\mathcal{C}_{per}(Y)+\mathcal{C}_{0}(\mathbb{R}^{2})$.

It is very important to note that we may assume different kind of behaviours
on $h_{1}$ and on $h_{2}$. For example, if $h_{1}\in \mathcal{C}_{per}(Y)$
and $h_{2}\in \mathcal{C}_{per}(Y)+\mathcal{C}_{0}(\mathbb{R}^{2})$, then we
reach the same conclusion of Theorems \ref{t1.2} and \ref{t1.3} with $%
\mathcal{A}=\mathcal{C}_{per}(Y)+\mathcal{C}_{0}(\mathbb{R}^{2})$.

\subsection{Problem 3: almost periodic setting}

We assume that the microstructures inside $\Omega $ are distributed in an
almost periodic fashion, that is, the function $\overline{y}\mapsto A(%
\overline{y},y_{3})$ is almost periodic in the Besicovitch sense \cite%
{Besicovitch, Bohr}. The underlying algebra with mean value in $\mathbb{R}%
^{2}$ is the algebra of Bohr continuous almost periodic functions on $%
\mathbb{R}^{2}$ denoted by $\mathcal{A}=\mathrm{AP}(\mathbb{R}^{2})$. It is
worth recalling that $\mathrm{AP}(\mathbb{R}^{2})$ \cite{Besicovitch, Bohr}
is defined as the algebra of functions on $\mathbb{R}^{2}$ that are
uniformly approximated by finite linear combinations of functions in the set 
$\{\cos (k\cdot ),\sin (k\cdot ):k\in \mathbb{R}^{d-1}\}$ where $\cos
(k\cdot )(y)=\cos (2\pi k\cdot y)$ and $\sin (k\cdot )(y)=\sin (2\pi k\cdot
y)$, $y\in \mathbb{R}^{2}$. It is known that $\mathrm{AP}(\mathbb{R}^{2})$
is an algebra wmv called the almost periodic algebra wmv on $\mathbb{R}^{2}$%
. The corresponding generalized Besicovitch space $B_{A}^{p}(\mathbb{R}^{2})$
is precisely the Besicovitch space $\mathcal{B}^{p}(\mathbb{R}^{2})$ defined
in \cite{Besicovitch, Bohr}.

We also assume that the functions $h_{1},h_{2}$ belong to $\mathrm{AP}(%
\mathbb{R}^{2})$. Then under these assumptions, the main results Theorems %
\ref{t1.2} and \ref{t1.3} are valid with the corresponding function spaces.
It is well-known from \cite{Bohr} that the mean value of a function $u\in 
\mathrm{AP}(\mathbb{R}^{2})$ is the unique constant that belongs to the
closed convex hull of the set of translates $\{u(\cdot +a):a\in \mathbb{R}%
^{2}\}$ of $u$.

\subsection{Problem 4: the asymptotic almost periodic environment}

We may deal with the asymptotic almost periodic distribution of
heterogeneities inside $\Omega $ with the corresponding algebra wmv $%
\mathcal{A}=\mathrm{AP}(\mathbb{R}^{2})+\mathcal{C}_{0}(\mathbb{R}^{2})$ 
\cite[Section 5.2.3]{JW2021}. In this case, we may assume that the functions 
$h_{1}$ and $h_{2}$ are either in $\mathrm{AP}(\mathbb{R}^{2})$ or in $%
\mathrm{AP}(\mathbb{R}^{2})+\mathcal{C}_{0}(\mathbb{R}^{2})$. We may also
assume that $h_{1}\in \mathrm{AP}(\mathbb{R}^{2})$ and $h_{2}\in \mathrm{AP}(%
\mathbb{R}^{2})+\mathcal{C}_{0}(\mathbb{R}^{2})$. All this leads to the
validity of Theorems \ref{t1.2} and \ref{t1.3} with $\mathcal{A}=\mathrm{AP}(%
\mathbb{R}^{2})+\mathcal{C}_{0}(\mathbb{R}^{2})$.

\begin{acknowledgement}
\emph{The second author acknowledges the support of the Alexander von
Humboldt Foundation in the framework of the `home office' scheme. He also
acknowledges the support of the Institute of Advanced Scientific Studies
(IHES), Paris, where the revised and complete version of this work has been
done.}
\end{acknowledgement}

\end{document}